\documentclass[12pt,fullpage]{article}

\usepackage{amssymb}
\usepackage{amsmath}
\usepackage{cases}
\usepackage{amsfonts}
\usepackage{amsmath,amssymb}
\usepackage{multicol}
\usepackage{color}
\usepackage{graphicx}
\usepackage{graphics}
\def\R{\mathbb{R}}
\def\N{\mathbb{N}}

\def\epsilon{\varepsilon}

\newcommand{\be}{\begin{equation}}
\newcommand{\ee}{\end{equation}}
\newcommand{\baa}{\begin{array}}
\newcommand{\eaa}{\end{array}}
\newcommand{\ba}{\begin{eqnarray}}
\newcommand{\ea}{\end{eqnarray}}

\newtheorem{theorem}{Theorem}[section]
\newtheorem{lemma}[theorem]{Lemma}
\newtheorem{corollary}[theorem]{Corollary}

\newtheorem{definition}[theorem]{Definition}
\newtheorem{assumption}[theorem]{Assumption}
\newtheorem{remark}[theorem]{Remark}

\numberwithin{equation}{section}
\newenvironment{proof}[1][Proof]{\noindent\textbf{#1.} }{\hfill $\Box$}
\allowdisplaybreaks
\begin{document}
\date{}
\title{\bf{Transition fronts in unbounded domains with multiple branches}}
\author{Hongjun Guo\thanks{Email address: hongjun.guo@etu.univ-amu.fr} \\
\\
\footnotesize{ Department of Mathematics and Statistics, University of Wyoming, Laramie, WY, USA}}
\maketitle

\begin{abstract}\noindent{This paper is concerned with the existence and uniqueness of transition fronts of a general reaction-diffusion-advection equation in domains with multiple branches. In this paper, every branch in the domain is not necessary to be straight and we use the notions of almost-planar fronts to generalize the standard planar fronts. Under some assumptions of existence and uniqueness of almost-planar fronts with positive propagating speeds in extended branches, we prove the existence of entire solutions emanating from some almost-planar fronts in some branches. Then, we get that these entire solutions converge to almost-planar fronts in some of the rest branches as time increases if no blocking occurs in these branches. Finally, provided by the complete propagation of every front-like solution emanating from one almost-planar front in every branch, we prove that there is only one type of transition fronts, that is, the entire solutions emanating from some almost-planar fronts in some branches and converging to almost-planar fronts in the rest branches.}
\noindent{}
\end{abstract}

%%%%%%%%%%%%%%%%%%%%%%%%%%%%%%%%%%%%%%%%%%%%%%%
%%%%%%%%%%%%%%%%%%%%%%%%%%%%%%%%%%%%%%%%%%%%%%%

\section{Introduction}
In this paper, we consider the following reaction-diffusion-advection equation in unbounded domains
\begin{eqnarray}\label{eq1.1}
\left\{\begin{array}{lll}
&u_t-$div$(A(x)\nabla u)+q(x)\cdot \nabla u=f(x,u), \quad &t\in\R,\ x\in\Omega,\\
&\nu A(x)\nabla u=0,& t\in\R,\ x\in\partial\Omega,
\end{array}
\right.
\end{eqnarray}
where $\Omega$ is a smooth non-empty open connected subset of $\R^N$ with $N\ge 2$ and $\nu(x)$ denotes the outward unit normal to $\partial\Omega$. More precise assumptions on $\Omega$ will be given later. Such equations arise in various models in combustion, population dynamics and ecology (see \cite{Fisher,KPP,Murray,SK,Xin}), where $u$ typically stands for the temperature or the concentration of a species.

Throughout the paper, $A(x)=(A_{ij}(x))_{1\le i,j\le N}$ denotes a globally $C^{1,\alpha}$ (with $\alpha>0$) matrix field defined in $\overline{\Omega}$ and there exist $0<\beta_1\le \beta_2$ such that
\be\label{eq-A}
\beta_1 |\xi|^2\le \sum_{1\le i,j\le N} A_{i,j}(x)\xi_i\xi_j\le \beta_2 |\xi|^2 \hbox{ for all } x\in \overline{\Omega} \hbox{ and } \xi\in\R^N.
\ee
The vector field $q(x)=(q_i(x))_{1\le i\le N}$ is bounded and of class $C^{0,\alpha}(\overline{\Omega})$. The term $q(x)\cdot\nabla u$ is understood as a transport term, or a driving flow. In some sense, the flow is driven by some exogeneously given flow represented by $q(x)$. The reaction term $f(x,u): \R^N\times[0,1]\rightarrow \R$ is assumed to be of class $C^{0,\alpha}$ in $x\in\R^N$ uniformly in $u\in [0,1]$, and of $C^{1,1}$ in $u$ uniformly in $x\in \R^N$. Assume that $f(x,u)$ is Lipschitz-continuous in $u$ uniformly for $x\in\R^N$. One also assumes that $0$ and $1$ are uniformly (in $x$) stable zeroes of $f(x,\cdot)$ in the sense that there exist $\gamma>0$ and $\sigma\in (0,1/2)$ such that $f(x,u)$ is decreasing in $u$ for $(x,u)\in\R^N\times [0,\sigma]$ and $(x,u)\in\R^N\times [1-\sigma,1]$ and
\begin{eqnarray}\label{eq-F}
\left\{\begin{array}{lll}
&f(x,u)\le -\gamma u, \quad & \hbox{ for all } (x,u)\in\R^N\times [0,\sigma],\\
&f(x,u)\ge \gamma (1-u),& \hbox{ for all } (x,u)\in \R^N\times [1-\sigma,1].
\end{array}
\right.
\end{eqnarray}
A typical example is the homogeneous bistable reaction $f$, that is, there is $\theta\in (0,1)$ such that
\be\label{F1}
f(u)<0 \hbox{ for $u\in (0,\theta)$ and } f(u)>0 \hbox{ for $u\in (\theta,1)$}.
\ee

The underlying domain $\Omega$ is assumed to be a domain with multiple branches. We refer to~\cite{GHS} for the definition of a domain with multiple cylindrical branches, in which every branch is straight. In this paper, we drop the word "cylindrical" such that our domains can contain curved branches.  For any unit vector $e\in \mathbb{S}^{N-1}$, let $P_e$ be the hyperplane of $\R^N$ orthogonal to $e$. Let $\omega(s): \R\rightarrow P_e$ be a family of subsets of $P_e$ which is continuous with respect to $s$. Assume that $\omega(s)$ is a smooth bounded nonempty connected subset of $P_e$ for every $s$ and $0\in \omega(s)$ for some $s$. One also assumes that $|\omega(s)|$ is uniformly bounded for $s\in\R$ in the sense that $\sup_{y\in \omega(s)} |y|<+\infty$ uniformly for $s\in \R$. A branch in a direction $e\in\mathbb{S}^{N-1}$ with $\omega(s)$ and shift $x_0$ is the open unbounded domain of $\R^N$ defined by
\be\label{defbranch}
\mathcal{H}_{e,\omega,x_0}=\big\{x\in\R^N: x-(x\cdot e)e\in\omega(x\cdot e),\ x\cdot e>0\big\}+x_0.
\ee
Notice that the width of $\mathcal{H}_{e,\omega,x_0}$ is bounded since $|\omega(s)|$ is uniformly bounded for $s\in \R$.
A smooth unbounded domain of $\R^N$ is called a domain with multiple branches if there exist a real number $L>0$, an integer $m\ge2$, and $m$ branches $\mathcal{H}_i:=\mathcal{H}_{e_i,\omega_i,x_i}$ (with $i=1,\cdots,m$), such that
\be\label{branches}\left\{\baa{l}
\displaystyle\mathcal{H}_i\cap\mathcal{H}_j=\emptyset\hbox{ for every }i\neq j\in\big\{1,\cdots,m\big\}\ \hbox{ and }\ \Omega\setminus B(0,L)=\mathop{\bigcup}_{i=1}^{m}\mathcal{H}_i\setminus B(0,L),\vspace{3pt}\\
\mathcal{H}_i\setminus B(0,L)\hbox{ is connected for every }i\in\{1,\cdots,m\}.\eaa\right.
\ee
One can refer to Figure 1 as an example of  a  domain with $5$ branches. Remember that in this paper, the branches are not necessary to be straight.

\begin{figure}[ht]\centering
\includegraphics[scale=0.5]{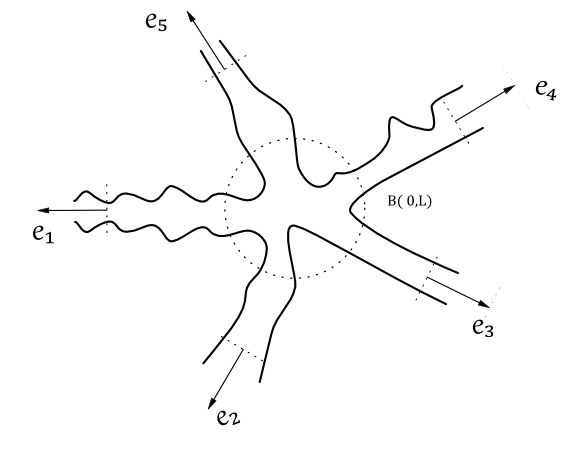}
\caption{An example of a domain with 5 branches.}
\end{figure}

%%%%%%%%%%%%%%%%%%%%%%%%%%%%%%%%%%%%%%%%%%%%%%%%%%%%%%%
We go back to simple cases for recalling some known results with different geometrical conditions of $\Omega$. Let us replace the divergence-type operator $div(A\nabla u)$ by the Laplace operator together with Neumann boundary condition $\partial_{\nu} u:=\frac{\partial u}{\partial \nu}=0$ and set $q(x)\equiv 0$. Assume $f(x,u)=f(u)$ satisfying \eqref{F1}. That is,
\begin{eqnarray}\label{eq1.1-q0}
\left\{\begin{array}{lll}
&u_t-\Delta u=f(u), \quad &t\in\R,\ x\in\Omega,\\
&u_{\nu}=0,& t\in\R,\ x\in\partial\Omega.
\end{array}
\right.
\end{eqnarray}
A simplest example of a domain with multiple branches is a straight infinite cylinder, that is, up to rotation,
\be\label{straight}
\Omega=\big\{(x_1,x'): x_1\in\R,\ x'\in\omega\big\},
\ee
where $\omega\subset\R^{N-1}$ is a smooth bounded non-empty open connected subset of $\R^{N-1}$.
From the pioneering paper \cite{FM}, it is well known that \eqref{eq1.1-q0} admits a planar front $\phi_f(x_1-c_f t)$ satisfying
\be\label{phi_f}
c_f\phi_f'+\phi_f''+f(\phi_f)=0,\ \phi_f(-\infty)=1 \hbox{ and } \phi_f(+\infty)=0.
\ee
The function $\phi_f$ is called the profile and the constant $c_f$ is called the propagation speed. It is also well-known that $\phi_f$ and $c_f$ are uniquely determined by $f$, $\phi_f$ is decreasing and $c_f$ is of the sign of $\int_0^1 f(s) ds$.

Another particular example is a curved cylinder, that is, up to rotation,
\be\label{bilateral}
\Omega=\big\{(x_1,x'): x_1\in\R,\ x'\in\omega(x_1)\big\},
\ee
where $(\omega(x_1))_{x_1\in\R}$ is a family of smooth bounded non-empty open connected subset of~$\R^{N-1}$. One can easily notice that the planar front~$\phi_f(x_1-c_f t)$ is not a solution of~\eqref{eq1.1-q0} in general if $\omega(x_1)$ is not independent of $x_1$. However, there still exist some front-like solutions. For instance, when $\Omega$ is a bilaterally straight cylinder, that is, $\omega(x_1)$ is independent of $x_1$ for $x_1\le-L$ and for $x_1\ge L$, with some $L>0$, there exist entire solutions emanating from the planar front $\phi_f(x_1-c_f t)$ coming from the ``left" part of the domain, see~\cite{BBC,CG,P}. More precisely, there exists a unique solution $u:\R\times\overline{\Omega}\to(0,1)$ of~\eqref{eq1.1-q0} such that
\be\label{eq1.9}
u(t,x)-\phi_f(x_1-c_f t)\rightarrow 0\,\text{ as } t\rightarrow -\infty \text{ uniformly in $\overline{\Omega}$}.
\ee
One can also refer to \cite{P} for the existence of front-like solutions in asymptotically straight cylinders. If the curved cylinder $\Omega$ is with periodic boundaries, that is, $\omega(x_1)$ is periodic with respect to $x_1\in\R$, there may also exist pulsating traveling fronts (as defined in the next paragraph), see \cite{MNL} for some conditions of existence. For a general domain with multiple cylindrical branches (every branch is straight), one knows from \cite{GHS} that there exist entire solutions emanating from planar fronts in some branches. In fact, let $I$ and $J$ be two non-empty sets of $\{1,\cdots,m\}$ ($m$ is the number of branches) such that $I\cap J=\emptyset$ and $I\cup J=\{1,\cdots,m\}$. There exists a time-increasing solution $u(t,x)$ of \eqref{eq1.1-q0} such that
\begin{equation*}
\left\{\baa{rcll}
u(t,x)\!-\!\phi_f(-x\cdot e_i\!-\!c_f t\!+\!\sigma_i) & \!\!\!\to\!\!\! & 0\!\! & \text{uniformly in $\overline{\mathcal{H}_i}\cap\overline{\Omega}\,$ for every $i\in I$},\vspace{3pt}\\
u(t,x) & \!\!\to\!\! & 0\!\! & \displaystyle\text{uniformly in $\overline{\Omega\setminus \mathop{\bigcup}_{i\in I}\mathcal{H}_i}$},\eaa\right.\text{as $t\rightarrow -\infty$}
\end{equation*}
for some real numbers $(\sigma_i)_{i\in I}$.

We now recall some results regarding to the periodic heterogeneity of coefficients. Suppose that the coefficients $A(x)$, $q(x)$, the nonlinear function $f(x,u)$ and the domain $\Omega$ are periodic in the direction $e$. For convenience in presentation, assume that they are periodic in the sense that $A(x+k e)=A(x)$, $q(x+k e)=q(x)$, $f(x+k e,u)=f(x,u)$ and $\Omega+k e=\Omega$ for any $k\in\mathbb{Z}^N$. In this case, one can define pulsating fronts $u(t,x)$ for \eqref{eq1.1}, see \cite{BH0}. A pulsating front $u(t,x)$ facing direction $e$ is a classical solution of \eqref{eq1.1} such that for some $c\neq 0$, there holds
$$u(t-\frac{k}{c},x)=u(t,x+k e) \hbox{ for all $t\in\R$, $x\in\overline{\Omega}$},$$
and such that, for all $t\in\R$,
$$u(t,x)\rightarrow 1 \hbox{ as } x\cdot e\rightarrow -\infty,\ u(t,x)\rightarrow 0 \hbox{ as } x\cdot e\rightarrow +\infty.$$
Similarly, one can define a pulsating front facing direction $-e$. We refer to \cite{DHZ1,Du2, FZ,NR4,X1,Xin0,X3} for some existence results of pulsating fronts in the whole space $\Omega=\R^N$. We also refer to \cite{DHZ,Xin1,XZ,Z3} for nonexistence results. The heterogeneity of coefficients not only effects the profiles of pulsating fronts but also the propagation speeds. Generally speaking, the front facing direction $e$ is not the same as the front facing direction $-e$.

In this paper, we aim to deduce some existence and uniqueness results of solutions of \eqref{eq1.1} under rather general assumptions. For this purpose, we recall the notion of transition fronts which generalizes the standard notion of traveling fronts. Such notion covers the planar fronts, pulsating fronts and also many types of fronts in the whole space, such as conical shaped fronts, pyramidal fronts and so on, see~\cite{HM,HMR1,HMR2,HS,NT1,T1,T2,T3}. Let us first introduce a few notations. The unbounded open connected set $\Omega\subset \R^N$ is assumed to have a globally $C^{2,\beta}$ boundary with $\beta>0$ (this is what we call a smooth domain throughout the paper), that is, there exist $\rho>0$ and $C>0$ such that, for every $y\in\partial \Omega$, there are a rotation $R_y$ of $\R^N$ and a $C^{2,\beta}$ map $\psi_y:\bar{B}=\big\{x'\in\R^{N-1}: |x'|\le2\rho\big\}\rightarrow \R$ such that $\psi_y(0)=0$, $\|\psi_y\|_{C^{2,\beta}(\bar{B})}\le C$ and
$$\Omega\cap B(y,\rho)=\left[y+R_y\big(\{x=(x',x_N)\in\R^N: x'\in\bar{B}, x_N>\psi_y(x')\}\big)\right]\cap B(y,\rho),$$
where
$$B(y,\rho)=\big\{x\in\R^N: |x-y|<\rho\big\}$$
and $|\ \ |$ denotes the Euclidean norm. Let $d_{\Omega}$ be the geodesic distance in $\overline{\Omega}$. For any two subsets~$A$ and~$B$ of~$\overline{\Omega}$, we set
\begin{equation*}
d_{\Omega}(A,B)=\inf\big\{d_{\Omega}(x,y): (x,y)\in A\times B\big\},
\end{equation*}
and $d_{\Omega}(x,A)=d_{\Omega}(\{x\},A)$ for $x\in\R^N$.
Consider now two families $(\Omega_t^-)_{t\in \mathbb{R}}$ and $(\Omega_t^+)_{t\in \mathbb{R}}$ of open non-empty subsets of $\Omega$ such that
\begin{eqnarray}\label{eq1.3}
\forall t\in \mathbb{R},\ \ \left\{\begin{array}{l}
\Omega_t^-\cap \Omega_t^+=\emptyset,\vspace{3pt}\\
\partial \Omega_t^-\cap \Omega=\partial \Omega_t^+\cap \Omega=:\Gamma_t,\vspace{3pt}\\
\Omega_t^-\cup \Gamma_t \cup \Omega_t^+=\Omega,\vspace{3pt}\\
\sup\big\{d_{\Omega}(x,\Gamma_t): x\in \Omega_t^+\big\}=\sup\big\{d_{\Omega}(x,\Gamma_t): x\in \Omega_t^-\big\}=+\infty\end{array}\right.
\end{eqnarray}
and
\begin{eqnarray}\label{eq1.4}
\left\{\begin{array}{l}
\inf\Big\{\sup\big\{d_{\Omega}(y,\Gamma_t): y\in \Omega_t^+, d_{\Omega}(y,x)\leq r\big\}: t\in \mathbb{R},\ x\in \Gamma_t\Big\}\rightarrow +\infty\vspace{3pt}\\
\inf\Big\{\sup\big\{d_{\Omega}(y,\Gamma_t): y\in \Omega_t^-, d_{\Omega}(y,x)\leq r\big\}: t\in \mathbb{R},\ x\in \Gamma_t\Big\}\rightarrow +\infty\end{array}\right.
\text{ as}\ \ r\rightarrow +\infty.
\end{eqnarray}
Notice that the condition~\eqref{eq1.3} implies in particular that the interface $\Gamma_t$ is not empty for every~$t\in \mathbb{R}$. As far as~\eqref{eq1.4} is concerned, it says that for any $M>0$, there is $r_M>0$ such that, for every $t\in \mathbb{R}$ and~$x\in \Gamma_t$, there are $y^{\pm}\in \mathbb{R}^N$ such that
\begin{eqnarray}\label{eq1.5}
y^{\pm}\in \Omega^{\pm}_t,\ \ d_{\Omega}(x,y^{\pm})\leq r_M\ \ \text{and}\ \ d_{\Omega}(y^{\pm},\Gamma_t)\geq M.
\end{eqnarray}
Moreover, in order to avoid interfaces with infinitely many twists, the sets $\Gamma_t$ are assumed to be included in finitely many graphs: there is an integer $n\geq 1$ such that, for each $t\in \mathbb{R}$, there are $n$ open subsets $\omega_{i,t}\subset \mathbb{R}^{N-1}$(for $1\leq i\leq n$), $n$ continuous maps $\psi_{i,t}: \omega_{i,t}\rightarrow \mathbb{R}$ and $n$ rotations $R_{i,t}$ of $\mathbb{R}^N$, with
\begin{equation}\label{eq1.6}
\Gamma_t \subset \bigcup_{1\leq i\leq n} R_{i,t}\left(\big\{x=(x',x_N)\in \mathbb{R}^N: x'\in \omega_{i,t},\ x_N=\psi_{i,t}(x')\big\}\right).
\end{equation}

\begin{definition}\label{TF}{\rm{\cite{BH1,BH2}}}
For problem~\eqref{eq1.1}, a transition front connecting $0$ and $1$ is a classical solution $u:\mathbb{R}\times\overline{\Omega} \rightarrow (0,1)$ for which there exist some sets $(\Omega_t^{\pm})_{t\in \mathbb{R}}$ and $(\Gamma_t)_{t\in \mathbb{R}}$ satisfying~\eqref{eq1.3}-\eqref{eq1.6}, and, for every $\varepsilon>0$, there exists $M_{\varepsilon}>0$ such that
\begin{eqnarray}\label{eq1.7}
\left\{\baa{l}
\forall\,t\in \mathbb{R},\ \ \forall\,x\in \overline{\Omega_t^+}, \ \ \left(d_{\Omega}(x,\Gamma_t)\geq M_{\varepsilon}\right)\Rightarrow \left(u(t,x)\geq 1-\varepsilon\right)\!,\vspace{3pt}\\
\forall\,t\in \mathbb{R},\ \ \forall\,x\in \overline{\Omega_t^-}, \ \ \left(d_{\Omega}(x,\Gamma_t)\geq M_{\varepsilon}\right)\Rightarrow \left(u(t,x)\leq \varepsilon\right)\!.\eaa\right.
\end{eqnarray}
Furthermore, $u$ is said to have a global mean speed $\gamma$ $(\geq 0)$ if
\be\label{defgms}
\frac{d_{\Omega}(\Gamma_t,\Gamma_s)}{|t-s|}\rightarrow \gamma \ \ \text{as}\ \ |t-s|\rightarrow +\infty.
\ee
\end{definition}

As far as the domain with multiple branches is concerned, one can smoothly extend every branch $\mathcal{H}_i$ for $-e_i$ part. Denote the extension of $\mathcal{H}_i$ by $\widetilde{\mathcal{H}}_i$. For instance, one can extend the $-e_i$ part of $\mathcal{H}_i$ to be
\be\label{wH}
\widetilde{\mathcal{H}}_i:=\{x\in\R^N; x-(x\cdot e_i)e_i \in \widetilde{\omega}_i(x\cdot e_i)\}+x_i
\ee
such that $|\widetilde{\omega}_i(s)|$ is uniformly bounded for $s\in\R$ and $\widetilde{\omega}_i(s)=\omega_i(s)$ for $s>0$. Moreover, one can redefine $A(x)$, $q(x)$, $f(x,u)$ to be $\widetilde{A}(x)$, $\widetilde{q}(x)$, $\widetilde{f}(x,u)$ in $\widetilde{\mathcal{H}}_i$ such that
\begin{eqnarray}\label{wA}
\left\{\begin{array}{lll}
&&\widetilde{A}(x)\in C^{1,\alpha}\hbox{ satisfies \eqref{eq-A}},\ \widetilde{q}(x)\in C^{0,\alpha} \hbox{ is bounded},\\
&&\widetilde{f}(x,u)\hbox{ is  $C^{0,\alpha}$ in $x$ and $C^{1,\alpha}$ in $u$ satisfying $\widetilde{f}(x,0)=\widetilde{f}(x,1)=0$ and \eqref{eq-F}},\\
&&\widetilde{A}(x)=A(x),\ \widetilde{q}(x)=q(x),\ \widetilde{f}(x,u)=f(x,u) \hbox{ for $x\in \mathcal{H}_i$},
\end{array}
\right.
\end{eqnarray}
 Notice that there are various ways of extension. As mentioned in front, for the following equation
\begin{eqnarray}\label{eq-extension}
\left\{\begin{array}{lll}
&u_t-$div$(\widetilde{A}(x)\nabla u)+\widetilde{q}(x)\cdot \nabla u=\widetilde{f}(x,u), \quad &t\in\R,\ x\in\widetilde{\mathcal{H}}_i,\\
&\nu \widetilde{A}(x)\nabla u=0,& t\in\R,\ x\in\partial\widetilde{\mathcal{H}}_i,
\end{array}
\right.
\end{eqnarray}
there are many possibilities of existence of front-like solutions including planar fronts and pulsating fronts. These fronts can actually be classified into the almost-planar front of the transition front, defined as following.

\begin{definition}{\rm{\cite{BH1,BH2}}}
For a fixed $i\in\{1,\cdots,m\}$, assume that $\widetilde{\mathcal{H}}_i$ is a smooth extension of $\mathcal{H}_i$ and $\widetilde{A}(x)$, $\widetilde{q}(x)$, $\widetilde{f}(x,u)$ are redefined in $\widetilde{\mathcal{H}}_i$ such that \eqref{wA} holds. A transition front $u(t,x)$ connecting $0$ and $1$ of \eqref{eq-extension} in the sense of Definition~\ref{TF} is called almost-planar if, for every $t\in\R$, the set $\Gamma_t$ can be chosen as
$$\Gamma_t=\{x\in\widetilde{\mathcal{H}}_i; x\cdot e_i=\xi_t\},$$
for some real number $\xi_t$. If $\Omega_t^+$ and $\Omega_t^-$ are defined by $\{x\in\widetilde{\mathcal{H}}_i; x\cdot e_i\lessgtr \xi_t\}$ respectively, we call that $u(t,x)$ is facing direction $e_i$. If $\Omega^+_t$ and $\Omega_t^-$ are defined by $\{x\in\widetilde{\mathcal{H}}_i; x\cdot e_i\gtrless \xi_t\}$ respectively, we call that $u(t,x)$ is facing direction $-e_i$.
\end{definition}

For the existence of entire solutions emanating from almost-planar fronts in some branches and their large time behavior, we make rather general assumptions, that is, assume the existence and uniqueness of almost-planar fronts in every extended branch.

\begin{assumption}\label{assumption-r}
For every $i\in\{1,\cdots,m\}$, assume that $\widetilde{\mathcal{H}}_i$ is a smooth extension of $\mathcal{H}_i$ and $\widetilde{A}(x)$, $\widetilde{q}(x)$, $\widetilde{f}(x,u)$ are redefined in $\widetilde{\mathcal{H}}_i$ such that \eqref{wA} holds. Assume that there is a unique (up to time shifts) almost-planar front $u^i_r(t,x)$ connecting $0$ and $1$ of \eqref{eq-extension} facing to direction $e_i$ with sets $\Omega_t^{\pm}$ and $\Gamma_t$ satisfying
$$\Omega_t^{+}=\{x\in\widetilde{\mathcal{H}}_i; x\cdot e_i<c^i_r t\},\ \Omega_t^{-}=\{x\in\widetilde{\mathcal{H}}_i; x\cdot e_i>c^i_r t\} \hbox{ and } \Gamma_t=\{x\in\widetilde{\mathcal{H}}_i; x\cdot e_i=c^i_r t\},$$
where $c^i_r$ is the propagation speed.
\end{assumption}

\begin{assumption}\label{assumption-l}
For every $i\in\{1,\cdots,m\}$, assume that $\widetilde{\mathcal{H}}_i$ is a smooth extension of $\mathcal{H}_i$ and $\widetilde{A}(x)$, $\widetilde{q}(x)$, $\widetilde{f}(x,u)$ are redefined in $\widetilde{\mathcal{H}}_i$ such that \eqref{wA} holds. Assume that there is a unique (up to time shifts) almost-planar front $u^i_l(t,x)$ connecting $0$ and $1$ of \eqref{eq-extension} facing to direction $-e_i$ with sets $\Omega_t^{\pm}$ and $\Gamma_t$ satisfying
$$\Omega_t^{+}=\{x\in\widetilde{\mathcal{H}}_i; x\cdot e_i>-c^i_l t\},\ \Omega_t^{-}=\{x\in\widetilde{\mathcal{H}}_i; x\cdot e_i<-c^i_l t\} \hbox{ and } \Gamma_t=\{x\in\widetilde{\mathcal{H}}_i; x\cdot e_i=-c^i_l t\},$$
where $c^i_l$ is the propagation speed.
\end{assumption}

\begin{remark}
The uniqueness in Assumptions~\ref{assumption-r} and~\ref{assumption-l} means that if there is a transition front $v(t,x)$ with the same sets $\Gamma_t$ and $\Omega^{\pm}_t$ defined in Assumption~\ref{assumption-r} or~\ref{assumption-l}, then there is $\tau\in \R$ such that $v(t,x)=u^i_r(t+\tau,x)$ or $v(t,x)=u^i_l(t+\tau,x)$ respectively.
\end{remark}

Moreover, we need a technical assumption, that is,

\begin{assumption}\label{assumption-limiting}
For any fixed $i\in\{1,\cdots,m\}$, take a sequence $\{t_n\}_{n\in\mathbb{N}}\subset \R$ such that $t_n\rightarrow +\infty$ as $n\rightarrow +\infty$. Let $\mathcal{H}_i^n=\mathcal{H}_i-t_n e_i$, $A_n(x)=A(x+t_n e_i)$, $q_n(x)=q_n(x+t_n e_i)$ and $f_n(x,\cdot)=f(x+t_n e_i,\cdot)$ for $x\in \mathcal{H}_i^n$. Assume that for any such sequence $\{t_n\}_{n\in \mathbb{N}}$, there is an infinite cylinder $\mathcal{H}_i^{\infty}$ parallel to $e_i$ such that $\mathcal{H}_i^n$ converge locally uniformly to $\mathcal{H}_i^{\infty}$ and there are $A_{\infty}(x)$, $q_{\infty}(x)$, $f_{\infty}(x,\cdot)$ satisfying \eqref{wA} such that $A_n\rightarrow A_{\infty}$, $q_n(x)\rightarrow q_{\infty}(x)$, $f_n\rightarrow f_{\infty}$ locally uniformly in $\mathcal{H}_i^{\infty}$ as $n\rightarrow +\infty$. Assume that there exist a unique (up to time shifts) almost planar front $v_r(t,x)$ facing direction $e_i$ and a unique almost planar front $v_l(t,x)$ facing direction $-e_i$ for the limiting equation
\begin{eqnarray*}
\left\{\begin{array}{lll}
&v_t-$div$(A_{\infty}(x)\nabla v)+q_{\infty}(x)\cdot \nabla v=f_{\infty}(x,v), \quad &t\in\R,\ x\in\mathcal{H}_i^{\infty},\\
&\nu A_{\infty}(x)\nabla v=0,& t\in\R,\ x\in\partial\mathcal{H}_i^{\infty}.
\end{array}
\right.
\end{eqnarray*}
\end{assumption}

This assumption actually holds for many cases such as (i) $\mathcal{H}_i^{\infty}$ is a straight cylinder and $A_{\infty}$, $q_{\infty}$, $f_{\infty}(\cdot,s)$ (for every $s\in\R$) are constant, (ii) $\mathcal{H}_i^{\infty}$ is a straight cylinder and $A_{\infty}$, $q_{\infty}$, $f_{\infty}$ are independent of $x\cdot e_i$, (iii)  $\mathcal{H}_i^{\infty}$, $A_{\infty}$, $q_{\infty}$, $f_{\infty}$ are periodic in $x\cdot e_i$, see \cite[Theorem~1.14]{BH2}. One can see Section~5 for some examples which satisfy all Assumptions~\ref{assumption-r}, \ref{assumption-l} and~\ref{assumption-limiting}.

Now, we claim the existence of entire solutions emanating from almost-planar fronts in some branches.

\begin{theorem}\label{Th1}
Let Assumptions~\ref{assumption-l} and~\ref{assumption-limiting} hold. Let $I$ and $J$ be two non-empty sets of $\{1,\cdots,m\}$ such that $I\cap J=\emptyset$ and $I\cup J=\{1,\cdots,m\}$. If $c^i_l>0$ for all $i\in I$, then there exists a time-increasing entire solution $u(t,x)$ of~\eqref{eq1.1} such that
\be\label{frontlike}
\left\{\baa{rcll}
u(t,x)\!-\!u^i_l(t+\sigma_i,x) & \!\!\!\to\!\!\! & 0\!\! & \text{uniformly in $\overline{\mathcal{H}_i}\cap\overline{\Omega}\,$ for every $i\in I$},\vspace{3pt}\\
u(t,x) & \!\!\to\!\! & 0\!\! & \displaystyle\text{uniformly in $\overline{\Omega\setminus \mathop{\bigcup}_{i\in I}\mathcal{H}_i}$},
\eaa\right.
\ee
as $t\rightarrow -\infty$ for some real numbers $(\sigma_i)_{i\in I}$.
\end{theorem}

We then investigate the large time behavior of the entire solution $u(t,x)$ in Theorem~\ref{Th1}. By the standard parabolic estimates, one knows that there is a $C^2(\overline{\Omega})$ solution $p:\overline{\Omega}\rightarrow (0,1]$ of
\begin{eqnarray}
\left\{\begin{array}{lll}
&-$div$(A(x)\nabla p)+q(x)\cdot \nabla p=f(x,p), \quad &x\in\Omega,\\
&\nu A(x)\nabla p=0,& t\in\R,\ x\in\partial\Omega,
\end{array}
\right.
\end{eqnarray}
such that $u(t,x)\rightarrow p(x)$ as $t\rightarrow +\infty$ locally uniformly in $x\in\overline{\Omega}$.
For the propagation in a domain with multiple cylindrical branches $\Omega$, two cases may occur, that is, the propagation is complete or blocked. Here, we mean the complete propagation by $p\equiv 1$ and we mean the blocked propagation by $p<1$. Both completely propagating and blocking phenomena have been proved to exist in many kinds of domains, such as exterior domains \cite{BHM}, bilaterally straight cylinders \cite{BBC,CG,RRBK} and some periodic domains \cite{DR}, under some geometrical conditions on domains respectively. Except the geometry of domains may block the propagation, the heterogeneity of the coefficients can also block the propagation, see \cite{E1}. If one treats the almost-planar front connecting $0$ and $1$ with positive propagation speed as an invasion of $0$ by $1$ in the sense of \cite{BH2}, then the negative propagation speed means $0$ invading $1$. The entire solution in Theorem~\ref{Th1} means that $1$ invades $0$ from branches $\mathcal{H}_i$. Similar as Theorem~\ref{Th1} and by providing Assumption~\ref{assumption-r}, one can replace the roles of $0$ and $1$ through replacing $u$ and $f(x,u)$ by $1-u$ and $-f(x,1-u)$ to prove that if $c^j_r<0$ for some $j\in J$, $0$ invades $1$ from branches $\mathcal{H}_j$. In other words, $c^j_r<0$ implies that $1$ can not invade $0$ in branches $\mathcal{H}_j$, that is, the propagation is blocked.

\begin{corollary}\label{cor1}
Let Assumptions~\ref{assumption-r}, \ref{assumption-l} and \ref{assumption-limiting} hold. Let $I$ and $J$ be the non-empty sets and $u(t,x)$ be the entire solution satisfying \eqref{frontlike} in Theorem~\ref{Th1}. If $c^i_r< 0$ for some $j\in J$, the solution $u(t,x)$ is blocked in $\mathcal{H}_j$, that is, $\sup_{x\in\mathcal{H}_j} p(x)< 1$.
\end{corollary}

We assert in the following theorem that if the propagation of $u(t,x)$ is unblocked in some branches $\mathcal{H}_j$ for some $j\in J$, that is,
\be\label{p}
\liminf_{x\in \mathcal{H}_j;\ x\cdot e_j\rightarrow +\infty} p(x)=1,
\ee
then $u(t,x)$ will eventually converge to the almost-planar front facing direction $e_j$ in $\mathcal{H}_j$.

\begin{theorem}\label{th2}
Let Assumptions~\ref{assumption-r}, \ref{assumption-l} and \ref{assumption-limiting} hold. Let $I$ and $J$ be the non-empty sets and $u(t,x)$ be the entire solution satisfying \eqref{frontlike} in Theorem~\ref{Th1}. Assume that $J_1$ is a non-empty subset of $J$. If no blocking occurs in the branch $\mathcal{H}_j$ for all $j\in J_1$ in the sense of \eqref{p}, then $c^j_r>0$ and the solution $u(t,x)$ converges to the almost-planar front facing direction $e_j$ in $\mathcal{H}_j$ for every $j\in J_1$ eventually, that is, there exist some real umbers $\tau_j$ such that
\[u(t,x)-u^j_r(t+\tau_j,x)\rightarrow 0,\]
for $x\in\overline{\mathcal{H}_j}$ such that $x\cdot e_j\ge \varepsilon t$ as $t\rightarrow +\infty$, where $\varepsilon$ is an arbitrary positive constant.
\end{theorem}

\begin{corollary}\label{cor2}
Let $u(t,x)$ be the entire solution in Theorem~\ref{th2}.
If $u$ propagates completely in $\overline{\Omega}$, that is, $p\equiv 1$, then $J_1\equiv J$ and the entire solution $u(t,x)$ is a transition front connecting $0$ and $1$ with $(\Gamma_t)_{t\in \R}$, $(\Omega^{\pm}_t)_{t\in \R}$ defined by
\be\label{eq+1.11}
\Gamma_t\!=\!\mathop{\bigcup}_{i\in I}\!\big\{x\!\in\!\mathcal{H}_i\cap\Omega: x\cdot e_i\!=\!c^i_l|t|\!+\!A\big\}\ (t\!\le\!0),\ \Gamma_t\!=\!\mathop{\bigcup}_{j\in J_1}\!\big\{x\!\in\!\mathcal{H}_j\cap\Omega: x\cdot e_j\!=\!c^j_r t\!+\!A\big\}\ (t\!>\!0),
\ee
and
\begin{eqnarray}\label{eq+1.12}
\left\{\baa{llll}
\displaystyle\Omega_t^+=\mathop{\bigcup}_{i\in I}\big\{x\in \mathcal{H}_i\cap\Omega: x\cdot e_i>c^i_l|t|+A\big\}, & \Omega_t^-=\Omega\setminus\overline{\Omega_t^+}, &\text{ for $t\le0$,}\vspace{3pt}\\
\displaystyle\Omega_t^-=\mathop{\bigcup}_{j\in J}\big\{x\in \mathcal{H}_j\cap\Omega: x\cdot e_j>c^j_r t+A\big\}, & \Omega_t^+=\Omega\setminus\overline{\Omega_t^-}, &\text{ for $t>0$,}\eaa\right.
\end{eqnarray}
for some $A>0$. Moreover, there exist some real numbers $(\tau_j)_{j\in J}$ such that
\be\label{largetime}\left\{\baa{rcll}
u(t,x)-u^j_r(t+\tau_j,x) & \!\!\!\to\!\!\! & 0\!\! & \text{uniformly in $\overline{\mathcal{H}_j}\cap\overline{\Omega}\,$ for every $j\in J$},\vspace{3pt}\\
u(t,x) & \!\!\!\to\!\!\! & 1\!\! & \displaystyle\text{uniformly in $\overline{\Omega\setminus \mathop{\bigcup}_{j\in J}\mathcal{H}_j}$},\eaa\right.
\ee
as $t\rightarrow+\infty$.
\end{corollary}

\begin{remark}
Indeed, the conclusion of Corollary~\ref{cor2} covers the results of Theorem~1.7 in \cite{GHS}, where the almost-planar front in every branch is the planar front with the unique speed $c_f$.
\end{remark}

Finally, we prove a general version of Conjecture~1.13 of \cite{GHS}, that is there is only one type of transition fronts connecting $0$ and $1$ by provided the complete propagation of any entire solution emanating from an almost-planar front in every branch.  For each $i\in\{1,\cdots,m\}$, we denote by $u_i:\R\times\overline{\Omega}\rightarrow (0,1)$ the time-increasing solution of \eqref{eq1.1} emanating from the almost-planar front $u^i_l(t,x)$ in the branch $\mathcal{H}_i$, that is,
\be\label{front}
u_i(t,x)\rightarrow u^i_l(t,x)\mathop{\longrightarrow}_{t\to-\infty} 0 \hbox{ uniformly in $\overline{\mathcal{H}_i}\cap \overline{\Omega}$},\ u_i(t,x)\mathop{\longrightarrow}_{t\to-\infty} 0 \hbox{ uniformly in $\overline{\Omega\setminus \mathcal{H}_i}$}.
\ee

\begin{theorem}\label{th3}
Let Assumptions~\ref{assumption-r}, \ref{assumption-l} and \ref{assumption-limiting} hold. If for every $i\in\{1,\cdots,m\}$, the entire solution $u_i$ of \eqref{front} propagates completely in the sense that $u_i(t,x)\rightarrow 1$ as $t\rightarrow +\infty$ locally uniformly in $\overline{\Omega}$, then any transition front of \eqref{eq1.1} connecting $0$ and $1$ is of the  type~\eqref{frontlike}, \eqref{eq+1.11}-\eqref{eq+1.12}, that is, it emanates from the almost-planar fronts coming from some proper subset of branches as $t\to-\infty$ and it converges to the almost-planar fronts in the other branches as $t\to+\infty$.
\end{theorem}

Notice that Theorem~\ref{th3} does not hold in general, without the assumption that every entire solution $u_i$ of \eqref{front} propagates completely, see the counter-example in Remark~1.10 of \cite{GHS}.

We organize this paper as following. In Section~2, we prove the existence of entire solutions emanating from almost-planar fronts in some branches, that is, Theorem~\ref{Th1}. We also show Corollary~\ref{cor1} in this section. Section~3 is devoted to proving Theorem~\ref{th2} and Corollary~\ref{cor2} which indicate the large time behavior of the entire solution of Theorem~\ref{Th1}. In Section~4, we prove the uniqueness of the transition front connecting $0$ and $1$, that is, Theorem~\ref{th3}. Finally, we give some examples in Section~5, to which our results can be applied.

%%%%%%%%%%%%%%%%%%%%%%%%%%%%%%%%%%%%%%%%%%%%%%%
%%%%%%%%%%%%%%%%%%%%%%%%%%%%%%%%%%%%%%%%%%%%%%%
\section{Existence  of entire solutions}

In this section, we only prove the existence of an entire solution emanating from one almost-planar front in one branch. The existence of entire solutions emanating from some almost-planar fronts in some branches, that is, Theorem~\ref{Th1}, can be proved in a similar way. Indeed, from the following constructions of sub- and supersolutions, one knows that the constructions do not depend on the geometrical structure of the domain beyond the initiated branch. It means that one can construct sub- and supersolutions for some branches by simply combining the sub- and supersolutions in each branch of these branches. We assume without loss of generality that the branch $\mathcal{H}_i$ ($i$ could be any integer of $1,\cdots,m$) is the initiated branch, in which the entire solution emanating. By Assumption~\ref{assumption-l}, one knows that for the extension $\widetilde{\mathcal{H}}_i$ of the branch $\mathcal{H}_i$, there is an almost-planar front $u^i_l(t,x)$ connecting $0$ and $1$ facing direction $-e_i$ of the equation \eqref{eq-extension}. Assume that $c^i_l>0$. In the sequel, we are going to prove that there is an entire solution $u(t,x)$ such that
\be\label{frontlike-2.1}
\left\{\baa{rcll}
u(t,x)\!-\!u^i_l(t,x) & \!\!\!\to\!\!\! & 0\!\! & \text{uniformly in $\overline{\mathcal{H}_i}\cap\overline{\Omega}\,$},\vspace{3pt}\\
u(t,x) & \!\!\to\!\! & 0\!\! & \displaystyle\text{uniformly in $\overline{\Omega\setminus \mathcal{H}_i}$},
\eaa\right.
\ee
as $t\rightarrow -\infty$.

\subsection{Construction of sub- and supersolutions}
We first need an auxiliary lemma which will be used frequently.

\begin{lemma}\label{lemma-psi}
For any $i\in\{1,\cdots,m\}$, assume that $\widetilde{\mathcal{H}}_i$ is a smooth extension of $\mathcal{H}_i$ and $\widetilde{A}(x)$, $\widetilde{q}(x)$, $\widetilde{f}(x,u)$ are redefined in $\widetilde{\mathcal{H}}_i$ such that \eqref{wA} holds. Then, for any $\beta>0$, there exist $\lambda_i>0$ and a positive $C^2$ function $\psi_i(x)$ satisfying
\begin{eqnarray}\label{psi}
\left\{\begin{array}{lll}
&-$div$(\widetilde{A}(x) \nabla \psi_i) +\lambda_i ($div$(\widetilde{A}(x) e_i \psi_i )+e_i\widetilde{A}\nabla \psi_i) +\widetilde{q}(x)\cdot \nabla\psi_i&\\
&-\lambda_i \widetilde{q}(x) e_i \psi_i -\lambda_i^2 (e_i\widetilde{A}(x) e_i) \psi_i \ge -\beta \psi_i, & x\in \widetilde{\mathcal{H}}_i,\\
&\nu \widetilde{A}(x)(\lambda_i \psi_i(x) e_i +\nabla \psi_i)\ge 0, &  x\in\partial\widetilde{\mathcal{H}}_i,
\end{array}
\right.
\end{eqnarray}
such that $0<\inf_{\widetilde{\mathcal{H}}_i} \psi_i(x)\le \sup_{\widetilde{\mathcal{H}}_i} \psi_i(x)<+\infty$.
\end{lemma}

\begin{proof}
Take a positive bounded $C^2$ function $\widetilde{\psi}_i(x)$ such that $\partial_{\nu}\widetilde{\psi}_i\ge 1$ for $x\in\partial\widetilde{\mathcal{H}}_i$
and all of its derivatives $\nabla\widetilde{\psi}_i$, $\nabla^2\widetilde{\psi}_i$ are bounded in the sense of $L^{\infty}$ norm. One can apply the classical distance function in \cite{GT} to get a such function. Let $\psi_i(x)=\widetilde{\psi}_i(x)+C$ where $C$ is a positive constant. Notice that $\nabla \psi_i=\nabla\widetilde{\psi}_i$ and $\nabla^2\psi_i=\nabla^2\widetilde{\psi}_i$. For any $\beta>0$, one can take $C$ sufficiently large and $0<\lambda_i\le 1/(\|\widetilde{\psi}_i\|_{L^{\infty}}+C)$ sufficiently small such that
\be\label{ineq}
\frac{1}{\psi_i(x)}\Big(-div(\widetilde{A}(x) \nabla \psi_i) +\lambda_i (div(\widetilde{A}(x) e_i \psi_i )+e_i\widetilde{A}\nabla \psi_i) +\widetilde{q}(x)\cdot \nabla\psi_i\Big)-\lambda_i \|\widetilde{q}(x) \| -\lambda_i^2 \beta_2 +\beta \ge 0,
\ee
for all $x\in\widetilde{\mathcal{H}}_i$ where $\|q(x)\|=\sum_{1\le i\le N} \|q_i(x)\|_{L^{\infty}}$ and $\beta_2$ is defined in \eqref{eq-A}. Then, the first inequality of \eqref{psi} holds by \eqref{ineq} and the second inequality of \eqref{psi} holds by
$$\lambda_i\le 1/(\|\widetilde{\psi}_i\|_{L^{\infty}}+C)\le (\nu\cdot \nabla\psi_i)/\psi_i=(\nu\cdot \nabla\widetilde{\psi}_i)/\psi_i.$$
From above, it is obvious that $0<\inf_{\widetilde{\mathcal{H}}_i}\psi_i(x)\le \sup_{\widetilde{\mathcal{H}}_i}\psi_i(x)<+\infty$.
\end{proof}
\vskip 0.3cm

Moreover, we need the time-monotonicity property of $u^i_l(t,x)$.

\begin{lemma}\label{muil}
Let Assumptions~\ref{assumption-l} and~\ref{assumption-limiting} hold. If $c^i_l>0$, then $(u^i_l)_t(t,x)>0$ for all $t\in \R$ and $x\in\widetilde{H}_i$ and for any positive constant $D$, there exist $T_1<0$ and $k>0$ such that $(u^i_l)_t(t,x)\ge k$ for $t\le T_1$ and $x\in\overline{\widetilde{H}_i}$ such that $|x\cdot e_i+c^i_l t|\le D$.
\end{lemma}

\begin{proof}
Since $c^i_l>0$, one can easily verify that $u^i_l(t,x)$ is an invasion of $0$ by $1$ in the sense of Definition~1.4 of \cite{BH1}. Then, by Theorem~1.11 of \cite{BH1}, one has that $u^i_l(t,x)$ is increasing in time $t$, that is, $(u^i_l)_t(t,x)>0$.

Denote, up to rotation,
$$\widetilde{\mathcal{H}}_i:=\{(x_1,x')\in\R^N; x' \in \widetilde{\omega}_i(x_1)\},$$
where $\widetilde{\mathcal{H}}_i=\mathcal{H}_i$ for $x_1>0$. Notice that
\be\label{disO}
d_{\Omega}(x,y)\ge |x-y| \hbox{ and } \frac{d_{\Omega}(x,y)}{|x-y|}<+\infty \hbox{ for any $x$, $y\in\widetilde{\mathcal{H}}_i$}
\ee
since $|\widetilde{\omega}_i(s)|$ is uniformly bounded for $s\in\R$.
By \cite[Theorem~1.2]{BH1}, we have that for any positive constant $D$, there is $\eta\in(0,1/2]$ such that $\eta\le u^i_l(t,x)\le 1-\eta$ for any $t\in \R$ and $x\in\overline{\widetilde{H}_i}$ such that $|x_1+c^i_l t|\le D$. Take any negative constant $T_1$. Now assume by contradiction that there exist sequences $\{t_n\}_{n\in\mathbb{N}}$ of $(-\infty,T_1]$ and $\{x_n\}_{n\in \mathbb{N}}=\{(x_{n1},x_n')\}_{n\in\mathbb{N}}$ satisfying $|x_{n1}+c^i_l t_n|\le D$ such that $(u^i_l)_t(t_n,x_n)\rightarrow 0$. Then, either $t_n\rightarrow t_*\in (-\infty,T_1]$ or $t_n\rightarrow -\infty$ as $n\rightarrow +\infty$. For the former case, there is $x_*\in\overline{\widetilde{\mathcal{H}}_i}$ such that $x_n\rightarrow x_*$ as $n\rightarrow +\infty$. So, $(u^i_l)_t(t_*,x_*)=0$. If $x_*\in \widetilde{\mathcal{H}}_i$, it contradicts $(u^i_l)_t>0$ for all $t\in\R$ and $x\in\overline{\widetilde{\mathcal{H}}_i}$. Then, $x_*\in \partial\widetilde{\mathcal{H}}_i$ and the Hopf lemma implies that $\partial_{\nu}(u^i_l)_t(t_*,x_*)<0$ which contradicts the boundary condition.

Then, if $t_n\rightarrow -\infty$ as $n\rightarrow +\infty$, $x_{n1}$ converges to $+\infty$ as $n\rightarrow +\infty$. Since $|\omega_i(s)|$ is bounded uniformly for $s\in \R$, there is $x'_*\in \R^{N-1}$ such that $x'_{n}\rightarrow x'_*$ as $n\rightarrow +\infty$. Let $\mathcal{H}_i^n=\widetilde{\mathcal{H}}_i-x_{n1} e_i$, $A_n(x)=\widetilde{A}(x_1+x_{n1},x')$, $q_n(x)=\widetilde{q}(x_1+x_{n1},x')$ and $f_n(x,\cdot)=\widetilde{f}(x_1+x_{n1},x',\cdot)$. Then, by Assumption~\ref{assumption-limiting}, there is an infinite cylinder $\mathcal{H}_i^{\infty}$ parallel to $e_i$ such that $\mathcal{H}_i^n$ converge locally uniformly to $\mathcal{H}_i^{\infty}$ and there are $A_{\infty}(x)$, $q_{\infty}(x)$, $f_{\infty}(x,\cdot)$ satisfying \eqref{wA} such that $A_n\rightarrow A_{\infty}$, $q_n\rightarrow q_{\infty}$, $f_n\rightarrow f_{\infty}$ locally uniformly in $\mathcal{H}_i^{\infty}$. Let $u_n(t,x)=u^i_l(t+t_n,x_1+x_{n1},x')$. Then, $(0,x'_n)\in \mathcal{H}_i^n$ and $(u_n)_t(0,0,x_n')\rightarrow 0$ as $n\rightarrow +\infty$. Since $u^i_l(t,x)$ is an almost-planar by Assumption~\ref{assumption-l}, it follows from Definition~\ref{TF} and \eqref{disO} that for any $\varepsilon>0$, there is $M_{\varepsilon}>0$ such that
\begin{eqnarray*}
\left\{\begin{array}{lll}
u^i_l(t,x)\ge 1-\varepsilon, && \hbox{ for $t\in\R$ and $x\in\widetilde{\mathcal{H}}_i$ such that $x_1+c^i_l t\ge M_{\varepsilon}$},\\
u^i_l(t,x)\le \varepsilon, && \hbox{ for $t\in\R$ and $x\in\widetilde{\mathcal{H}}_i$ such that $x_1+c^i_l t\le -M_{\varepsilon}$}.
\end{array}
\right.
\end{eqnarray*}
Then, by $|x_{n1}+c^i_l t_n|\le D$, one can easily check that
\begin{eqnarray}\label{MD}
\left\{\begin{array}{lll}
u_n(t,x)\ge 1-\varepsilon, && \hbox{ for $t\in\R$ and $x\in\mathcal{H}^n_i$ such that $x_1+c^i_l t\ge M_{\varepsilon}+D$},\\
u_n(t,x)\le \varepsilon, && \hbox{ for $t\in\R$ and $x\in\mathcal{H}^n_i$ such that $x_1+c^i_l t\le -M_{\varepsilon}-D$}.
\end{array}
\right.
\end{eqnarray}
It means that $u_n(t,x)$ is an almost-planar front facing $-e_i$ with speed $c^i_l$ for all $n$. By parabolic estimates, $u_n(t,x)$ converge, up to extraction of a subsequence, to a solution $u_*(t,x)$ of
\begin{eqnarray*}
\left\{\begin{array}{lll}
&u_t-$div$(A_{\infty}(x)\nabla u)+q_{\infty}(x)\cdot \nabla u=f_{\infty}(x,u), \quad &t\in\R,\ x\in\mathcal{H}_i^{\infty},\\
&\nu A_{\infty}(x)\nabla u=0,& t\in\R,\ x\in\partial\mathcal{H}_i^{\infty}.
\end{array}
\right.
\end{eqnarray*}
Then, $u_*(t,x)=v_l(t,x)$ up to shifts by Assumption~\ref{assumption-limiting} and $v_l(t,x)$ is an almost-planar front facing $-e_i$ with speed $c^i_l$ by \eqref{MD}. Since $c^i_l>0$, one has that $(v_l)_t>0$. However, $(0,x'_n)\rightarrow (0,x'_*)\in \mathcal{H}_i^{\infty}$ as $n\rightarrow +\infty$ and $v_t(0,0,x'_*)=(u_*)_t(0,0,x'_*)=0$ which is a contradiction.
This completes the proof.
\end{proof}

\vskip 0.3cm

Similarly, one can get the following lemma for $u^i_r(t,x)$.

\begin{lemma}\label{muir}
Let Assumptions~\ref{assumption-r} and~\ref{assumption-limiting} hold. If $c^i_r>0$, then $(u^i_r)_t(t,x)>0$ for all $t\in \R$ and $x\in\widetilde{H}_i$ and for any positive constant $D$, there exist $T_2>0$ and $k>0$ such that $(u^i_r)_t(t,x)\ge k$ for $t\ge T_2$ and $x\in\overline{\widetilde{H}_i}$ such that $|x\cdot e_i-c^i_r t|\le D$.
\end{lemma}

We then announce some parameters. Remember that $\mathcal{H}_i$ is the initiated branch and $i$ is a fixed integer of $1,\cdots,m$. For convenience, define
$$\mathcal{H}_i(R):=\{x\in\mathcal{H}_i; x\cdot e_i\ge R\},$$
for any $R>0$. By Lemma~\ref{lemma-psi}, there exist $\lambda_i>0$ and $\psi_i(x)>0$ such that \eqref{psi} holds for the extension $\widetilde{\mathcal{H}}_i$ of $\mathcal{H}_i$ and $\beta=\gamma$ where $\gamma$ is defined by \eqref{eq-F}.  Notice that $v(x):=e^{-\lambda_i(x\cdot e_i-L)} \psi_i(x)$ satisfies
\begin{eqnarray}\label{eq-v}
\left\{\begin{array}{lll}
&-$div$(A(x)\nabla v)+q(x)\cdot \nabla v\ge -\gamma v, \quad & x\in \mathcal{H}_i(L),\\
&\nu A(x)\nabla v\ge 0,& x\in\partial\mathcal{H}_i(L),
\end{array}
\right.
\end{eqnarray}
where $L$ is defined by \eqref{branches}.
Let $\delta>0$ be a constant such that
$$\delta\le \min\Big(\frac{\sigma}{2},\lambda_i c^i_l\Big) \hbox{ where $\sigma$ is defined in \eqref{eq-F}}.$$
Let $\widetilde{\delta}=\delta /\|\psi_i\|_{L^{\infty}(\overline{\mathcal{H}_i(L)})}$. Since $u^i_l(t,x)$ is an almost-planar front by Assumption~\ref{assumption-l}, there is $M_{\delta}>0$ such that
\begin{eqnarray}\label{Md}
\left\{\begin{array}{lll}
u^i_l(t,x)\ge 1-\delta, && \hbox{ for $t\in\R$ and $x\in\widetilde{\mathcal{H}}_i$ such that $x\cdot e_i+c^i_l t\ge M_{\delta}$},\\
u^i_l(t,x)\le \delta, && \hbox{ for $t\in\R$ and $x\in\widetilde{\mathcal{H}}_i$ such that $x\cdot e_i+c^i_l t\le -M_{\delta}$}.
\end{array}
\right.
\end{eqnarray}
By Lemma~\ref{muil}, there exist $T_1<0$ and $k>0$ such that
\be\label{eq-k}
(u^i_l)_t(t,x)\ge k \hbox{ for $t\le T_1$ and $x\in\widetilde{\mathcal{H}}_i$ such that $-M_{\delta}\le x\cdot e_i+c^i_l t\le M_{\delta}$}.
\ee
Let $\omega$ be a large positive constant such that
\be\label{eq-omega}
\omega k \ge (\gamma+M) e^{\lambda_i(M_{\delta}+L+1)},
\ee
where $M=\sup_{x\in\R^N, u\in[0,1]}|f_u(x,u)|$.
We now construct a subsolution as the following
\begin{eqnarray*}
\underline{u}(t,x)=\left\{\begin{array}{lll}
\max\{u^i_l(\underline{\zeta}(t),x)-\widetilde{\delta} e^{-\lambda_i (x\cdot e_i-L)}\psi_i(x),0\}, &&\hbox{ in $\overline{\mathcal{H}_i(L)}$}\\
0, &&\hbox{ in $\overline{\Omega}\setminus \overline{\mathcal{H}_i(L)}$},
\end{array}
\right.
\end{eqnarray*}
where $\underline{\zeta}(t)=t-\omega e^{\delta t}$.

\begin{lemma}\label{lemma2.1}
There exists $T<0$ such that $\underline{u}(t,x)$ is a subsolution of \eqref{eq1.1} for all $t\le T$ and $x\in \overline{\Omega}$.
\end{lemma}

\begin{proof}
Take $T\le T_1<0$ such that
\be\label{eq-T}
c^i_l T<-L-M_{\delta'},
\ee
where $\delta':=\widetilde{\delta} \inf_{x\in\widetilde{\mathcal{H}}_i} \psi(x)>0$ (one knows from Lemma~\ref{lemma-psi} that  $ \inf_{x\in\widetilde{\mathcal{H}}_i} \psi(x)>0$) and $M_{\delta'}$ is defined by \eqref{Md} with $\delta$ replaced by $\delta'$. Notice that $\underline{\zeta}(t)\le T_1$ for all $t\le T$.
Let us first check that $\underline{u}(t,x)$ is well-defined and continuous for all $t\le T$ and $x\in \overline{\Omega}$. Notice that the interfaces $\Gamma_{\xi}$ of $u_l^i(\xi,x)$ is defined by $\Gamma_{\xi}=\{x\in\widetilde{\mathcal{H}_i}; x\cdot e_i=-c^i_l\xi\}$. Since $c^i_l\underline{\zeta}(t)<-L-M_{\delta'}$ for $t\le T$ by \eqref{eq-T} and $\widetilde{\mathcal{H}_i}=\mathcal{H}_i$ for $x\cdot e_i\ge L$, one has $x\cdot e_i+c^i_l\underline{\zeta}(t)<-M_{\delta'}$ for $t\le T$ and $x\in\overline{\mathcal{H}_i}$ such that $x\cdot e_i\le L$. Thus, it follows from \eqref{Md} that
 $$u^i_l(\underline{\zeta}(t),x)\le \delta' \hbox{ for $t\le T$ and $x\in\overline{\mathcal{H}_i}$ such that $x\cdot e_i\le L$}.$$
 Meantime, on $x\in\overline{\mathcal{H}_i}$ such that $x\cdot e_i= L$, one has $\widetilde{\delta} e^{-\lambda_i (x\cdot e_i-L)}\psi_i(x)\ge \delta'$. Therefore, $\underline{u}(t,x)=0$ for $t\le T$ and $x\in\overline{\mathcal{H}_i}$ such that $x\cdot e_i= L$. By the definition of $\underline{u}(t,x)$, it is well-defined and continuous in $\overline{\Omega}$. Since $u^i_l(t,x)$ satisfies $\nu A(x)\nabla u^i_l(t,x)=0$ on $x\in\partial \mathcal{H}_i(L)$ and by \eqref{eq-v}, one knows that  $\underline{u}(t,x)$ satisfies $\nu A(x) \nabla\underline{u}(t,x)\le 0$ for any $t\le T$ and $x\in\partial\Omega$.

To prove that $\underline{u}(t,x)$ is a subsolution, one only has to check that
$$N(t,x):=\underline{u}_t-\hbox{ div} (A(x) \nabla \underline{u}) + q(x)\cdot \nabla \underline{u}-f(x,\underline{u})\le 0,$$
for $t\le T$ and $x\in\overline{\Omega}$ such that $\underline{u}(t,x)>0$. From above arguments, $\underline{u}(t,x)>0$ implies that $x\in\overline{\mathcal{H}_i(L)}$ and $\underline{u}(t,x)=u^i_l(\underline{\zeta}(t),x)-\widetilde{\delta} e^{-\lambda_i(x\cdot e_i-L)}\psi_i(x)$. Since $u^i_l(t,x)$ satisfies \eqref{eq1.1} for $x\in\overline{\mathcal{H}_i(L)}$, it follows from some calculations and \eqref{eq-v} that
\begin{align*}
N(t,x)\le -\omega \delta e^{\delta t} (u^i_l)_t(\underline{\zeta}(t),x) +\gamma \widetilde{\delta} e^{-\lambda_i(x\cdot e_i-L)}\psi_i(x) +f(x,u^i_l(\underline{\zeta}(t),x))-f(x,\underline{u}(t,x)).
\end{align*}
For $t\le T$ and $x\in\overline{\mathcal{H}_i(L)}$ such that $x\cdot e_i+c^i_l \underline{\zeta}(t)\le -M_{\delta}$, it follows that $0<u^i_l(\underline{\zeta}(t),x)\le \delta$ and $\underline{u}(t,x)\le \delta$. Then, by $\delta\le \sigma/2$ and \eqref{eq-F},
$$f(x,u^i_l(\underline{\zeta}(t),x))-f(x,\underline{u}(t,x))\le -\gamma\widetilde{\delta} e^{-\lambda_i(x\cdot e_i-L)} \psi_i(x).$$
Thus,
\begin{align*}
N(t,x)\le -\omega\delta e^{\delta t}(u^i_l)_t(\underline{\zeta}(t),x) +\Big(\gamma-\gamma \Big)\widetilde{\delta} e^{-\lambda_i(x\cdot e_i-L)}\psi_i(x)\le 0,
\end{align*}
by $(u^i_l)_t>0$. For $t<T$ and $x\in\overline{\mathcal{H}_i(L)}$ such that $x\cdot e_i+c^i_l \underline{\zeta}(t)\ge M_{\delta}$, it follows that $u^i_l(\underline{\zeta}(t),x)\ge 1-\delta$ and $\underline{u}(t,x)\ge 1-2\delta$. Then, by $\delta\le \sigma/2$ and \eqref{eq-F},
$$f(x,u^i_l(\underline{\zeta}(t),x))-f(x,\underline{u}(t,x))\le -\gamma\widetilde{\delta} e^{-\lambda_i(x\cdot e_i-L)} \psi_i(x).$$
Thus,
$$N(t,x)\le -\omega\delta e^{\delta t} (u^i_l)_t(\underline{\zeta}(t),x)+\Big(\gamma-\gamma\Big)\widetilde{\delta} e^{-\lambda_i(x\cdot e_i-L)}\psi_i(x)\le 0,$$
by $(u^i_l)_t>0$.

Finally, for $t\le T$ and $x\in\overline{\mathcal{H}_i(L)}$ such that $-M_{\delta}\le x\cdot e_i+c^i_l \underline{\zeta}(t)\le M_{\delta}$, one has $(u^i_l)_t(\underline{\zeta}(t),x)\ge k>0$ where $k$ is defined by \eqref{eq-k}. It also implies that $x\cdot e_i\ge -c^i_l(t-\omega e^{\delta t})-M_{\delta}$ and hence $e^{-\lambda_i(x\cdot e_i-L)}\le e^{\lambda_i c^i_l t} e^{\lambda_i(M_{\delta}+L)}$. It is obvious that
$$f(x,u^i_l(\underline{\zeta}(t),x))-f(x,\underline{u}(t,x))\le M\widetilde{\delta} e^{-\lambda_i(x\cdot e_i-L)} \psi_i(x),$$
where $M=\sup_{x\in\R^N,u\in [0,1]} |f_u(x,u)|$.
Therefore,
\begin{align*}
N(t,x)\le& -k\omega \delta e^{\delta t}+\gamma\widetilde{\delta} e^{-\lambda_i(x\cdot e_i-L)}\psi_i(x)+M\widetilde{\delta} e^{-\lambda_i(x\cdot e_i-L)}\psi_i(x)\\
\le& -k\omega \delta e^{\delta t} +(\gamma+M) \widetilde{\delta} e^{\lambda_i c^i_l t} e^{\lambda_i (M_{\delta}+L)}\psi_i(x)\le 0,
\end{align*}
by $\delta\le \lambda_i c^i_l$, $\widetilde{\delta}\psi_i(x)\le \delta$ and \eqref{eq-omega}.
 This completes the proof.
\end{proof}
\vskip 0.3cm

Take any small $\varepsilon>0$ (at least $\varepsilon<\delta'=\widetilde{\delta} \inf_{x\in\widetilde{\mathcal{H}}_i} \psi(x)$). Let $L_{\varepsilon}\ge L$ large enough such that $\delta e^{-\lambda_i(L_{\varepsilon}-L)}\le \varepsilon/2$. We now construct supersolutions as the following
$$\overline{u}_1(t,x)=u_l^i(\overline{\zeta}(t),x)+\widetilde{\delta} e^{-\lambda_i(x\cdot e_i-L)}\psi_i(x), \hbox{ in $\overline{\mathcal{H}_i(L)}$},$$
where $\overline{\zeta}(t)=t+\omega e^{\delta t}$ and $\omega$ is defined by \eqref{eq-omega},
and
$$\overline{u}_2(t,x)=\varepsilon \hbox{ for $x\in\overline{\mathcal{H}_i}$ such that $x\cdot e_i\le L_{\varepsilon}$ and $x\in \overline{\Omega}\setminus\overline{\mathcal{H}_i}$}.$$

\begin{lemma}\label{lemma2.2}
There exists $T_{\varepsilon}<0$ such that $\overline{u}_1(t,x)$ is a supersolution of \eqref{eq1.1} for $t\le T_{\varepsilon}$ and $x\in\overline{\mathcal{H}_i(L)}$ and $\overline{u}_2(t,x)$ is a supersolution of \eqref{eq1.1} for $t\le T_{\varepsilon}$ and $x\in\overline{\mathcal{H}_i}$ such that $x\cdot e_i\le L_{\varepsilon}$ and $x\in\overline{\Omega}\setminus\overline{\mathcal{H}_i}$.
\end{lemma}

\begin{proof}
Take $T_{\varepsilon}\le T_1-\omega<0$ such that $c^i_l \omega e^{\delta t}\le 1$ for all $t\le T_{\varepsilon}$ and
$$c_l^i T_{\varepsilon}\le -L_{\varepsilon}-M_{\varepsilon/2}-1,$$
where $M_{\varepsilon/2}$ is defined by \eqref{Md} with $\delta$ replaced by $\varepsilon/2$.
Notice that $\overline{\zeta}(t)\le T_1$ for all $t\le T_{\varepsilon}$.
Since $f(x,\varepsilon)<0$ by $\varepsilon<\delta'<\delta<\sigma$, one can easily find that $\overline{u}_2(t,x)$ is a supersolution of \eqref{eq1.1} for $t\le T_{\varepsilon}$ and $x\in\overline{\mathcal{H}_i}$ such that $x\cdot e_i\le L_{\varepsilon}$ and $x\in\overline{\Omega}\setminus\overline{\mathcal{H}_i}$. Obviously, $\overline{u}_1(t,x)$ is of $C^2$ in $\overline{\mathcal{H}_i(L)}$. By $\nu A(x)\nabla u^i_l(t,x)=0$ on $\partial \mathcal{H}_i(L)$ and \eqref{eq-v}, it follows that $\nu A(x)\nabla\overline{u}_1(t,x)\ge 0$ for any $t\le T_{\varepsilon}$ and $x\in\partial\mathcal{H}_i(L)$.

Now, we check that
$$N(t,x):= \overline{u}_{1t}-\hbox{div}(A(x)\nabla \overline{u}_1)+q(x)\cdot \nabla \overline{u}_1-f(x,\overline{u}_1)\ge 0,$$
for $t\le T_{\varepsilon}$ and $x\in \overline{\mathcal{H}_i(L)}$. By some calculation, one can obtain that
\begin{align*}
N(t,x)\ge \omega \delta e^{\delta t} (u^i_l)_t(\overline{\zeta}(t),x) -\gamma \widetilde{\delta} e^{-\lambda_i(x\cdot e_i-L)}\psi_i(x) +f(x,u^i_l(\overline{\zeta}(t),x))-f(x,\overline{u}_1(t,x)),
\end{align*}
since $u^i_l(t,x)$ satisfies \eqref{eq1.1} for $x\in\overline{\mathcal{H}_i(L)}$ and by \eqref{eq-v}.
For $t\le T_{\varepsilon}$ and $x\in\overline{\mathcal{H}_i(L)}$ such that $x\cdot e_i+c^i_l \overline{\zeta}(t)\le -M_{\delta}$, it follows that $0<u^i_l(\overline{\zeta}(t),x)\le \delta$. Thus, $\overline{u}_1(t,x)\le 2\delta\le \sigma$ since $\widetilde{\delta} e^{-\lambda_i(x\cdot e_i-L)} \psi_i(x)\le \widetilde{\delta}\psi_i(x)<\delta$. Then, by \eqref{eq-F},
$$f(x,u^i_l(\overline{\zeta}(t),x)-f(x,\overline{u}_1(t,x))\ge \gamma\widetilde{\delta} e^{-\lambda_i(x\cdot e_i-L)} \psi_i(x).$$
Thus,
\begin{align*}
N(t,x)\ge \omega\delta e^{\delta t}(u^i_l)_t(\overline{\zeta}(t),x) +\Big(-\gamma+\gamma  \Big)\widetilde{\delta} e^{-\lambda_i(x\cdot e_i-L)}\psi_i(x)\ge 0,
\end{align*}
by $(u^i_l)_t>0$. For $t\le T_{\varepsilon}$ and $x\in\overline{\mathcal{H}_i(L)}$ such that $x\cdot e_i+c^i_l \overline{\zeta}(t)\ge M_{\delta}$, it follows that $u^i_l(\overline{\zeta}(t),x)\ge 1-\delta$ and $\overline{u}_1(t,x)\ge 1-\delta\ge 1-\sigma$. Then, by \eqref{eq-F},
$$f(x,u^i_l(\overline{\zeta}(t),x))-f(x,\underline{u}_1(t,x))\ge \gamma\widetilde{\delta} e^{-\lambda_i(x\cdot e_i-L)} \psi_i(x).$$
Thus,
$$N(t,x)\ge \omega\delta e^{\delta t} (u^i_l)_t(\overline{\zeta}(t),x)+\Big(-\gamma+\gamma\Big)\widetilde{\delta} e^{-\lambda_i(x\cdot e_i-L)}\psi_i(x)\ge 0,$$
by $(u^i_l)_t>0$.

Finally, for $t\le T_{\varepsilon}$ and $x\in\overline{\mathcal{H}_i(L)}$ such that $-M_{\delta}\le x\cdot e_i+c^i_l \overline{\zeta}(t)\le M_{\delta}$, one has $(u^i_l)_t(\overline{\zeta}(t),x)\ge k>0$ where $k$ is defined by \eqref{eq-k}. It also implies that $x\cdot e_i\ge -c^i_l\overline{\zeta}(t)-M_{\delta}\ge -c^i_l t-1-M_{\delta}$ since $c^i_l \omega e^{\delta t}\le 1$ for $t\le T_{\varepsilon}$ and hence $e^{-\lambda_i(x\cdot e_i-L)}\le e^{\lambda_i c^i_l t} e^{\lambda_i(M_{\delta}+L+1)}$. It is obvious that
$$f(x,u^i_l(\overline{\zeta}(t),x))-f(x,\overline{u}_1(t,x))\ge -M\widetilde{\delta} e^{-\lambda_i(x\cdot e_i-L)} \psi_i(x),$$
where $M=\sup_{x\in\R^N,u\in [0,1]} |f_u(x,u)|$. Therefore,
\begin{align*}
N(t,x)\ge& k\omega \delta e^{\delta t}-\gamma\widetilde{\delta} e^{-\lambda_i(x\cdot e_i-L)}\psi_i(x)-M\widetilde{d} e^{-\lambda_i(x\cdot e_i-L)}\psi_i(x)\\
\ge& k\omega \delta e^{\delta t} -(\gamma+M) \widetilde{\delta} e^{\lambda_i c^i_l t} e^{\lambda_i (M_{\delta}+L+1)}\psi_i(x)\ge 0,
\end{align*}
by $\delta\le \lambda_i c^i_l$, $\widetilde{\delta} \psi_i(x)\le \delta$ and \eqref{eq-omega}.
 This completes the proof.
\end{proof}
\vskip 0.3cm

Notice that $\overline{u}_1(t,x)\ge \delta'>\varepsilon$  for all $t\le T_{\varepsilon}$ and $x\in\overline{\mathcal{H}_i}$ such that $x\cdot e_i=L$. By the definition of $L_{\varepsilon}$, we know that $\widetilde{\delta} e^{-\lambda_i(x\cdot e_i-L)}\psi_i(x)\le \varepsilon/2$ for $x\in\overline{\mathcal{H}_i(L)}$ such that $x\cdot e_i=L_{\varepsilon}$. By the definition of $T_{\varepsilon}$, one has that $x\cdot e_i+c^i_l\overline{\zeta}(t)\le -M_{\varepsilon/2}$ for $t\le T_{\varepsilon}$ and $x\in\overline{\mathcal{H}_i(L)}$ such that $x\cdot e_i=L_{\varepsilon}$. Thus, $u_l^i(\overline{\zeta}(t),x)\le \varepsilon/2$ and hence $\overline{u}_1(t,x)\le \varepsilon$ for $t\le T_{\varepsilon}$ and $x\in\overline{\mathcal{H}_i(L)}$ such that $x\cdot e_i=L_{\varepsilon}$. Let
\begin{eqnarray*}
\overline{u}(t,x)=\left\{\begin{array}{lll}
\overline{u}_1(t,x), &&\hbox{ for $t\le T_{\varepsilon}$ and $x\in\overline{\mathcal{H}_i(L_{\varepsilon})}$}\\
\min\{\overline{u}_1(t,x),\overline{u}_2(t,x)\}, &&\hbox{ for $t\le T_{\varepsilon}$ and $x\in\overline{\mathcal{H}_i(L)}$ such that $x\cdot e_i\le L_{\varepsilon}$},\\
\overline{u}_2(t,x),&&\hbox{for $t\le T_{\varepsilon}$ and $x\in \overline{\Omega}\setminus\overline{\mathcal{H}_i(L)}$}£¬
\end{array}
\right.
\end{eqnarray*}
which is well-defined by above analysis. Moreover, it is a supersolution of \eqref{eq1.1} for $t\le T_{\varepsilon}$ and $x\in\overline{\Omega}$ by Lemma \ref{lemma2.2} and the maximum principle.

\subsection{Existence, monotonicity and uniqueness of the entire solution}
We now prove the existence of an entire solution satisfying \eqref{frontlike-2.1}. Consider a sequence of solutions $u_n$ of \eqref{eq1.1} for $t>-n$ with initial value
$$u_n(-n,x)=\underline{u}(-n,x).$$
It is obvious that $\underline{u}(t,x)$ is increasing in $t$ for $t$ negative enough and $\underline{u}(t,x)\le \overline{u}(t,x)$. Even if it means decreasing $T_{\varepsilon}$, assume $T_{\varepsilon}\le T$ where $T$ and $T_{\varepsilon}$ are defined by Lemma~\ref{lemma2.1} and Lemma~\ref{lemma2.2} respectively. Then, it follows from the comparison principle that
\be\label{lunl}
\underline{u}(t,x)\le u_n(t,x)\le \overline{u}(t,x) \hbox{ for $-n\le t<T_{\varepsilon}$ and $x\in\Omega$},
\ee
and
$$u_n(t,x)\ge u_{n-1}(t,x) \hbox{ for all $t\in [-n+1,+\infty)$, $x\in\Omega$}.$$
Using the monotonicity of the sequence and parabolic estimates on $u_n$, we have that the sequence $u_n$ converges to an entire solution $u(t,x)$ of \eqref{eq1.1}. By \eqref{lunl}, the solution $u(t,x)$ satisfies
$$\underline{u}(t,x)\le u(t,x)\le \overline{u}(t,x) \hbox{ for all $t\in (-\infty,T_{\varepsilon})$ and $x\in\Omega$}.$$
By definition of $\underline{u}$, $\overline{u}$ and remembering that $\varepsilon$ can be arbitrary small, one then has that
\begin{equation*}
\left\{\baa{rcll}
u(t,x)\!-\!u^i_l(t,x) & \!\!\!\to\!\!\! & 0\!\! & \text{uniformly in $\overline{\mathcal{H}_i}\cap\overline{\Omega}\,$},\vspace{3pt}\\
u(t,x) & \!\!\to\!\! & 0\!\! & \displaystyle\text{uniformly in $\overline{\Omega\setminus \mathcal{H}_i}$},
\eaa\right.
\end{equation*}
as $t\rightarrow -\infty$.
Since $\underline{u}_t>0$ for $t$ negative enough, it follows from the maximum principle that $(u_n)_t>0$ for $t>-n$ and $x\in\Omega$. Passing to the limit $n\rightarrow +\infty$, one gets that $u_t\ge 0$ for $t\in\R$ and $x\in\Omega$. Again by the maximum principle, either $u_t>0$ or $u_t\equiv 0$. Since $u(t,x)\rightarrow u_l^i(t,x)$ in $\overline{\mathcal{H}_i}\cap\overline{\Omega}$ as $t\rightarrow -\infty$ and $(u^i_l)_t>0$, $u_t\equiv 0$ is impossible. Therefore, $u_t>0$ for $t\in\R$ and $x\in\Omega$.

For any $0<\delta\le 1/2$, define
$$\Omega_{\delta}(t)=\{x\in\Omega; \delta\le u(t,x)\le 1-\delta\}.$$
One can apply the proof of Lemma~\ref{muil} to get the following lemma.

\begin{lemma}\label{lemma-m}
For all $\delta\in (0,1/2]$, there exist $T_{\delta}<0$ and $k>0$ such that
$$u_t(t,x)\ge k \hbox{ for all $t\in (-\infty,T_{\delta}]$ and $x\in\Omega_{\delta}(t)$.}$$
\end{lemma}

\begin{remark}
Since $u(t,x)$ satisfies \eqref{frontlike-2.1}, one can take $T_{\delta}$ negative enough such that $\Omega_{\delta}(t)\subset \{x\in\mathcal{H}_i; |x\cdot e_i+x^i_l t|\le D\}$ for $t\le T_{\delta}$ and some $D>0$.
\end{remark}

Now, we prove the uniqueness of the entire solution $u$. Let $\delta>0$ be defined as in Section~2.1. Assume that there is another entire solution $v(t,x)$ satisfying \eqref{frontlike-2.1}. Then, for any $0<\varepsilon<\delta$, there is $t_{\varepsilon}<0$ such that
$$\sup_{x\in \Omega} |u(t,x)-v(t,x)|\le \varepsilon, \hbox{ for any $t\le t_{\varepsilon}$}.$$
For any $t_0\le \min(t_{\varepsilon},T_{\delta}-\omega\varepsilon)$, define
$$u^+(t,x)=u(t_0+t+\omega \varepsilon(1-e^{-\delta t}),x)+\varepsilon e^{-\delta t} \hbox{ and } u^-(t,x)=u(t_0+t-\omega \varepsilon(1-e^{-\delta t}),x)-\varepsilon e^{-\delta t},$$
where $\omega>0$ is a constant such that $\omega k\delta\ge \delta+M$, $k$ is defined by Lemma~\ref{lemma-m} and $M=\sup_{x\in\R^N, u\in[0,1]}|f_u(x,u)|$. One can check that $u^+(t,x)$ and $u^-(t,x)$ are sup- and subsolutions of the problem satisfied by $v(t_0+t,x)$ for $t\in [0,T_{\delta}-t_0-\omega \varepsilon]$. We omit the details of the checking process by referring to similar arguments as in Section~3 of \cite{BHM}. Then, by the comparison principle, one has that
$$u^-(t,x)\le v(t_0+t,x)\le u^+(t,x), \hbox{ for $t\in [0,T_{\delta}-t_0-\omega \varepsilon]$ and $x\in \Omega$}.$$
It implies that
$$u(t-\omega \varepsilon(1-e^{-\delta (t-t_0)}),x)-\varepsilon e^{-\delta (t-t_0)}\le v(t,x)\le u(t+\omega \varepsilon(1-e^{-\delta (t-t_0)}),x)+\varepsilon e^{-\delta (t-t_0)},$$
for $t\in [t_0,T_{\delta}-\omega \varepsilon]$ and $x\in \Omega$.
As $t_0\rightarrow -\infty$, one gets that
\be\label{eq-uvu}
u(t-\omega\varepsilon,x)\le v(t,x)\le u(t+\omega \varepsilon,x)
\ee
for all $t\in(-\infty,T_{\delta}-\omega\varepsilon]$ and $x\in\Omega$. Again by the comparison principle, \eqref{eq-uvu} holds for all $t\in\R$ and $x\in\Omega$. Since $\varepsilon$ is arbitrary, we then get that $v(t,x)\equiv u(t,x)$. This completes the proof of the uniqueness of the entire solution satisfying \eqref{frontlike-2.1}.

\subsection{Proof of Corollary~\ref{cor1}}

We complete this section by proving Corollary~\ref{cor1}.
\vskip 0.3cm

\begin{proof}[Proof of Corollary~\ref{cor1}]
Let $u(t,x)$ be the entire solution satisfying \eqref{frontlike}. Assume that $c^j_r<0$ for some $j\in J=\{1,\cdots,m\}\setminus I$. By replacing $u$ and $f(x,u)$ by $1-u$ and $-f(x,1-u)$ and applying the same arguments as in Section~2.1 and Section~2.2, one can prove that there is an entire solution $v(t,x)$ of \eqref{eq1.1} satisfying
\begin{equation*}
\left\{\baa{rcll}
v(t,x)\!-\!u^j_r(t,x) & \!\!\!\to\!\!\! & 0\!\! & \text{uniformly in $\overline{\mathcal{H}_j}\cap\overline{\Omega}\,$},\vspace{3pt}\\
v(t,x) & \!\!\to\!\! & 1\!\! & \displaystyle\text{uniformly in $\overline{\Omega\setminus \mathcal{H}_j}$},
\eaa\right.
\end{equation*}
as $t\rightarrow -\infty$ and $v(t,x)$ is decreasing as $t$ increases. Then, for any $\varepsilon>0$, there is $t_{\varepsilon}<0$ such that
$$v(t,x)\ge 1-\varepsilon \hbox{ for $t\le t_{\varepsilon}$ and $x\in \overline{\Omega\setminus \mathcal{H}_j}$},$$
and
$$u(t,x)\le \varepsilon \hbox{ for $t\le t_{\varepsilon}$ and $x\in \overline{\Omega\setminus \mathop{\bigcup}_{i\in I} \mathcal{H}_i}$}.$$
Since $I\cap J=\emptyset$ and $u(t,x)\le 1$ for all $t\in\R$ and $x\in\overline{\Omega}$, one has that
$$u(t_0,x)-\varepsilon\le v(t_0,x), \hbox{ for all $x\in\overline{\Omega}$},$$
where $t_0\le \min(t_{\varepsilon},T_{\delta}-\omega\varepsilon)$, $\delta$, $\omega$, $T_{\delta}$ are parameters as defined in Section~2.2. Then, by similar arguments as in Section~3 of \cite{BHM}, one can easily check that the function
$$
u^-(t,x)=u(t_0+t-\omega \varepsilon(1-e^{-\delta t}),x)-\varepsilon e^{-\delta t}
$$
is a subsolution of the problem satisfied by $v(t+t_0,x)$ for  $t\in [0,T_{\delta}-t_0-\omega \varepsilon]$. It follows from the comparison principle that
$$u(t-\omega \varepsilon(1-e^{-\delta (t-t_0)}),x)-\varepsilon e^{-\delta (t-t_0)}\le v(t,x),$$
for $t\in[t_0,T_{\delta}-\omega\varepsilon]$ and $x\in\Omega$. As $t_0\rightarrow -\infty$, one obtains that
$$u(t-\omega \varepsilon,x)\le v(t,x), \hbox{ for $t\in (-\infty,T_{\delta}-\omega\varepsilon]$ and $x\in\Omega$.}$$
Again by the comparison principle, the above inequality holds for all $t\in\R$ and $x\in\Omega$. As $\varepsilon\rightarrow 0$, we gets that
$$u(t,x)\le v(t,x) \hbox{ for all $t\in\R$ and $x\in\Omega$}.$$
By the properties of the almost-planar front $u^j_r(t,x)$ and since $v(t,x)$ is decreasing in time, it follows that
$$\sup_{x\in\mathcal{H}_j} u(t,x)\le \sup_{x\in\mathcal{H}_j} v(t,x)<1 \hbox{ for all $t\in\R$}.$$
This completes the proof.
\end{proof}

%%%%%%%%%%%%%%%%%%%%%%%%%%%%%%%%%%%%%%%%%%%%%%%
\section{Large time behavior of entire solutions}
In this section, let $u(t,x)$ be the entire solution emanating from some almost-planar fronts in branches $\mathcal{H}_i$ ($i\in I$), that is, satisfying \eqref{frontlike}. Let $J=\{1,\cdots,m\}\setminus I$ and $J_1$ be a non-empty subset of $J$. We assume that the propagation of $u(t,x)$ is not blocked by branches $\mathcal{H}_j$ for all $j\in J_1$, that is,
\be\label{noblock}
\hbox{$u(t,x)\rightarrow p(x)$ as $t\rightarrow +\infty$ and }\liminf_{x\in\mathcal{H}_i;\ x\cdot e_j\rightarrow +\infty} p(x)=1.
\ee
By Corollary~\ref{cor1}, we immediately get that $c^j_r>0$ for all $j\in J_1$. In the sequel, we investigate the large time behavior of the solution $u(t,x)$ in branches $\mathcal{H}_j$ for $j\in J_1$.

Recall that $\mathcal{H}_j(R):=\{x\in\mathcal{H}_j; x\cdot e_j\ge R\}$ for $R>0$. Let $\lambda_j>0$ and $\psi_j(x)>0$ be the constant and function satisfying Lemma~\ref{lemma-psi} for the extension $\widetilde{\mathcal{H}_j}$ of $\mathcal{H}_j$ and $\beta=\gamma$ where $\gamma$ is defined by \eqref{eq-F}. Then, $v(x):=e^{-\lambda_j(x\cdot e_j-L_1)} \psi_j(x)$ satisfies
\begin{eqnarray}\label{eq-v1}
\left\{\begin{array}{lll}
&-$div$(A(x)\nabla v)+q(x)\cdot \nabla v\ge -\gamma v, \quad & x\in \mathcal{H}_j(L),\\
&\nu A(x)\nabla v\ge 0,& x\in\partial\mathcal{H}_j(L),
\end{array}
\right.
\end{eqnarray}
for any positive constant $L_1$.

\begin{lemma}\label{lemma3.1}
For every $j\in J_1$, there exist $\delta>0$, $\widetilde{\delta}>0$, $L_1\ge L>0$, $t_1\in\R$, $t_2\in\R$, $\tau_1\in\R$ and $\tau_2\in \R$ such that
\be\label{eq-down}
u(t,x)\ge u_r^j(t-t_1+\tau_1,x)-\delta e^{-\delta (t-t_1)}-\widetilde{\delta} e^{- \lambda_j (x\cdot e_j-L_1)} \psi_j(x)
\ee
for $t\ge t_1$ and $x\in\overline{\mathcal{H}_j(L_1)}$ and
\be\label{eq-up}
u(t,x)\le u_r^j(t-t_2+\tau_2,x)+\delta e^{-\delta (t-t_2)}+\widetilde{\delta} e^{- \lambda_j (x\cdot e_j-L_1)}\psi_j(x)
\ee
for $t\ge t_2$ and $x\in\overline{\mathcal{H}_j(L_1)}$.
\end{lemma}

\begin{proof}
{\it Step 1: some parameters.} Fix any $j\in J_1$. Let $\delta>0$ be a constant such that
\be\label{eq-delta}
\delta\le \min_{j\in J_1}\Big(\lambda_j c^j_r,\gamma,\frac{\sigma}{3}\Big),
\ee
where $\gamma$ and $\sigma$ are defined by \eqref{eq-F}.
Define 
$$\widetilde{\delta}:=\frac{\delta}{\min_{j\in J_1}\|\psi_j\|_{L^{\infty}(\overline{\mathcal{H}_j(L)})}} \hbox{ and } \delta'=\widetilde{\delta}\min_{j\in J_1}\inf_{\overline{\mathcal{H}_j(L)}} \psi_j(x).$$ 
Since $u^j_r(t,x)$ is an almost-planar front defined by Assumption~\ref{assumption-r}, it follows from Definition~\ref{TF} and \eqref{disO} that there is $M_{\delta}>0$ such that
 \begin{eqnarray}\label{Md-s3}
\left\{\begin{array}{lll}
u^j_r(t,x)\ge 1-\delta, && \hbox{ for $t\in\R$ and $x\in\widetilde{\mathcal{H}}_j$ such that $x\cdot e_j-c^i_r t\le -M_{\delta}$},\\
u^j_r(t,x)\le \delta, && \hbox{ for $t\in\R$ and $x\in\widetilde{\mathcal{H}}_j$ such that $x\cdot e_j-c^i_r t\ge M_{\delta}$}.
\end{array}
\right.
\end{eqnarray}
Since $c^j_r>0$ and by Lemma~\ref{muir}, one has that $(u^j_r)_t(t,x)>0$ and there exist $T_2>0$ and $k>0$ such that 
\be\label{eq-k2}
(u^j_r)_t(t,x)\ge k \hbox{ for $t\ge T_2$ and $x\in\widetilde{\mathcal{H}}_j$ such that $-M_{\delta}\le x\cdot e_j-c^j_r t\le M_{\delta}$}.
\ee 
Remember that $\widetilde{\mathcal{H}}_j=\mathcal{H}_j$ for $x\cdot e_j\ge L$. Let $\omega>0$ such that
\be\label{eq-omega1}
k\omega\ge \delta+\gamma+2M,
\ee
where $M=\sup_{x\in\R^N, u\in[0,1]}|f_u(x,u)|$.
By \eqref{noblock}, there are $t_1\in\R$ and $L_1\ge L$ such that
\be\label{eq-hL}
u(t,x)\ge 1-\delta'\ge 1-\delta,\hbox{  for $t\ge t_1$ and $x\in\overline{\mathcal{H}_j}$ such that $L_1\le x\cdot e_j\le L_1+2M_{\delta}+R$},
\ee
where $R$ is a fixed constant such that $R\ge c^j_r\omega$. Even if it means increasing $L_1$, assume that $L_1/c^j_r\ge T_2$.

{\it Step 2: proof of \eqref{eq-down}.} For $t\ge t_1$ and $x\in\overline{\mathcal{H}_j(L_1)}$, we set
$$\underline{u}(t,x)=\max\Big(u^j_r(\underline{\zeta}_1(t),x) -\delta e^{-\delta (t-t_1)} -\widetilde{\delta} e^{-\lambda_j(x\cdot e_j-L_1)}\psi_j(x),0\Big),$$
where $\underline{\zeta}_1(t)=t-t_1+\omega e^{-\delta (t-t_1)} -\omega+\tau_1$ and $\tau_1=(L_1+M_{\delta}+R)/ c^j_{r}$. Notice that $\underline{\zeta}_1(t)\ge T_2$ for all $t\ge t_1$.
We prove that $\underline{u}(t,x)$ is a subsolution of \eqref{eq1.1} for $t\ge t_1$ and $x\in\overline{\mathcal{H}_j(L)}$.

At the time $t=t_1$, it follows from \eqref{eq-hL} that
$$\underline{u}(t_1,x)\le \max\Big(1-\delta-\widetilde{\delta} e^{-\lambda_j(x\cdot e_j-L_1)}\psi_j(x),0\Big)\le 1-\delta\le u(t_1,x),$$
for $x\in\overline{\mathcal{H}_j}$ such that $L_1\le x\cdot e_j\le L_1+2M_{\delta}+R$. Since $c^j_r\underline{\zeta}_1(t)\ge L_1\ge L$ for $t\ge t_1$,  the interfaces $\Gamma_{\underline{\zeta}_1}$ of $u^j_r(\underline{\zeta}_1(t),x)$ for $t\ge t_1$ are defined by $\{x\in\mathcal{H}_j; x\cdot e_j=c^j_r \underline{\zeta}_1(t)\}$. For $x\in\overline{\mathcal{H}_j(L_1+2M_{\delta}+R)}$, one has that $x\cdot e_j-c^j_r \tau_1\ge M_{\delta}$. Then, by \eqref{Md-s3},
$$\underline{u}(t_1,x)\le \max\Big(\delta-\delta-\widetilde{\delta} e^{-\lambda_j(x\cdot e_j-L_1)}\psi_j(x),0\Big)=0\le u(t_1,x),$$
for $x\in\overline{\mathcal{H}_j(L_1+2M_{\delta}+R)}$. Thus,
$$\underline{u}(t_1,x)\le u(t_1,x), \hbox{ for all $x\in\overline{\mathcal{H}_j(L_1)}$}.$$
Since $\nu A(x)\nabla u^j_r(t,x)=0$ for $x\in\partial\mathcal{H}_j(L)$ and by \eqref{eq-v1}, one can notice that $\underline{u}(t,x)$ satisfies $\nu A(x)\nabla\underline{u}(t,x)\le 0$ for $t\ge t_1$ and $x\in\partial\mathcal{H}_j(L_1)$. For $x\in \overline{\mathcal{H}_j}$ such that $x\cdot e_j=L_1$, one has $\underline{u}(t,x)\le 1-\widetilde{\delta}\psi_j(x)\le 1-\delta'\le u(t,x)$ for all $t\ge t_1$.

Now let us check that
$$N(t,x):=\underline{u}_t-\hbox{div}(A(x)\nabla\underline{u})+q(x)\cdot \nabla\underline{u}-f(x,\underline{u})\le 0,$$
for $t\ge t_1$ and $x\in \overline{\mathcal{H}_j(L_1)}$ such that $\underline{u}(t,x)>0$. Notice that $u^j_r(t,x)$ satisfies \eqref{eq1.1} for $x\in\overline{\mathcal{H}_j(L_1)}$. After some calculation, it follows from \eqref{eq-v1} that
\begin{align*}
N(t,x)\le& -\omega\delta e^{-\delta(t-t_1)} (u^j_r)_t(\underline{\zeta}_1(t),x) +\delta^2 e^{-\delta (t-t_1)}+\gamma\widetilde{\delta} e^{-\lambda_j(x\cdot e_j-L_1)} \psi_j(x)\\ &+f(x,u^j_r(\underline{\zeta}_1(t),x))-f(x,\underline{u}(t,x)).
\end{align*}
For $t\ge t_1$ and $x\in \overline{\mathcal{H}_j(L_1)}$ such that $x\cdot e_j-c^j_r\underline{\zeta}_1(t)\le -M_{\delta}$, one has that $u^j_r(\underline{\zeta}_1(t),x)\ge 1-\delta$ and hence $\underline{u}(t,x)\ge 1-3\delta\ge 1-\sigma$. Thus, by \eqref{eq-F},
$$f(x,u^j_r(\underline{\zeta}_1(t),x))-f(x,\underline{u}(t,x))\le -\gamma\Big(\delta e^{-\delta(t-t_1)}+\widetilde{\delta} e^{-\lambda_j(x\cdot e_j-L_1)}\psi_j(x)\Big).$$
It follows that
\begin{align*}
N(t,x)\le& -\omega\delta e^{-\delta(t-t_1)} (u^j_r)_t(\underline{\zeta}_1(t),x) +\delta (\delta-\gamma) e^{-\delta (t-t_1)}+(\gamma-\gamma)\widetilde{\delta} e^{-\lambda_j(x\cdot e_j-L_1)} \psi_j(x)\le 0,
\end{align*}
 by$(u^j_r)_t>0$ and \eqref{eq-delta}. For $t\ge t_1$ and $x\in \overline{\mathcal{H}_j(L_1)}$ such that $x\cdot e_j-c^j_r \underline{\zeta}_1(t)\ge M_{\delta}$, one has that $u^j_r(\underline{\zeta}_1(t),x)\le \delta$ and hence $\underline{u}(t,x)\le \delta\le \sigma$. Thus, by \eqref{eq-F},
$$f(x,u^j_r(\underline{\zeta}_1(t),x))-f(x,\underline{u}(t,x))\le -\gamma\Big(\delta e^{-\delta(t-t_1)}+\widetilde{\delta} e^{-\lambda_j(x\cdot e_j-L_1)}\psi_j(x)\Big).$$
It follows that
\begin{align*}
N(t,x)\le& -\omega\delta e^{-\delta(t-t_1)} (u^j_r)_t(\underline{\zeta}_1(t),x) +\delta (\delta-\gamma) e^{-\delta (t-t_1)}+(\gamma-\gamma)\widetilde{\delta} e^{-\lambda_j(x\cdot e_j-L_1)} \psi_j(x)\le 0,
\end{align*}
by $(u^j_r)_t>0$ and \eqref{eq-delta}. Finally, for $t\ge t_1$ and $x\in \overline{\mathcal{H}_j(L)}$ such that $-M_{\delta}\le x\cdot e_j-c^j_r\underline{\zeta}_1(t)\le M_{\delta}$, one has that $x\cdot e_j\ge c^j_r\underline{\zeta}_1(t)-M_{\delta}\ge c^j_r(t-t_1)+L_1$ and $(u^j_r)_t(\underline{\zeta}_1(t),x)\ge k>0$ by \eqref{eq-k2}. Then, $\widetilde{\delta}e^{-\lambda_j(x\cdot e_j-L_1)}\psi_j(x)\le \delta e^{-\lambda_j c^j_r (t-t_1)} $.  It is obvious that
$$f(x,u^j_r(\underline{\zeta}_1(t),x))-f(x,\underline{u}(t,x))\le M\Big(\delta e^{-\delta(t-t_1)}+\widetilde{\delta} e^{-\lambda_j(x\cdot e_j-L_1)}\psi_j(x)\Big),$$
where $M=\sup_{x\in\R^N, u\in[0,1]}|f_u(x,u)|$. Thus, by \eqref{eq-delta} and \eqref{eq-omega1}, it follows that
\begin{align*}
N(t,x)\le& -k\omega\delta e^{-\delta(t-t_1)} +\delta (\delta+M) e^{-\delta (t-t_1)}+(\gamma+M)\widetilde{\delta} e^{-\lambda_j(x\cdot e_j-L_1)} \psi_j(x)\\
\le&  -k\omega\delta e^{-\delta(t-t_1)} +\delta (\delta+M) e^{-\delta (t-t_1)}+(\beta+M)\delta e^{-\lambda_jc^j_r(t-t_1)} \le 0.
\end{align*}

By the comparison principle, one obtains that
$$u(t,x)\ge \underline{u}(t,x)\ge u^j_r(t-t_1 -\omega+\tau_1,x) -\delta e^{-\delta (t-t_1)} -\widetilde{\delta} e^{-\lambda_j(x\cdot e_j-L_1)}\psi_j(x),$$
for $t\ge t_1$ and $x\in\overline{\mathcal{H}_j(L_1)}$.

{\it Step 3: proof of \eqref{eq-up}.}
Since $u(t,x)$ satisfies \eqref{frontlike} and $j\not\in I$, there is $t_2\in\R$ such that
$$u(t_2,x)\le \delta,\hbox{ for $x\in\overline{\mathcal{H}_j(L_1)}$}.$$
For $t\ge t_2$ and $x\in\overline{\mathcal{H}_j(L_1)}$, let us set
$$\overline{u}(t,x)=\min\Big(u^j_r(\overline{\zeta}_2(t),x)+\delta e^{-\delta (t-t_2)}+\widetilde{\delta} e^{- \lambda_j (x\cdot e_j-L_1)}\psi_j(x),1\Big),$$
where $\overline{\zeta}_2(t)=t-t_2-\omega e^{-\delta (t-t_2)}+\omega+\tau_2$, $\tau_2=(L_1+M_{\delta'})/c^j_r$ and $M_{\delta'}$ is defined by \eqref{Md-s3} with $\delta$ replacing by $\delta'$.
We prove that $\overline{u}(t,x)$ is a supersolution of \eqref{eq1.1} for $t\ge t_2$ and $x\in\overline{\mathcal{H}_j(L_1)}$.

At the time $t=t_2$, one has that
$$\overline{u}(t_2,x)\ge \delta\ge u(t_2,x), \hbox{ for all $x\in\overline{\mathcal{H}_j(L_1)}$}.$$
Notice that $\overline{u}(t,x)$ satisfies $\nu A(x)\nabla\overline{u}\ge 0$ for $t\ge t_2$ and $x\in\partial\mathcal{H}_j(L_1)$. By the definition of $\tau_2$, one has that $x\cdot e_j-c^j_r\overline{\zeta}_2(t)\le L_1-c^j_r \tau_2\le -M_{\delta'}$ for $x\in\overline{\mathcal{H}_j}$ such that $x\cdot e_j=L_1$. Then, by \eqref{Md-s3}, $\overline{u}(t,x)\ge 1-\delta'+\widetilde{\delta}\psi_j(x)\ge 1\ge u(t,x)$ for all $t\ge t_2$ and $x\in \overline{\mathcal{H}_j}$ such that $x\cdot e_j= L_1$.

Then, one can do the similar arguments as in Step~2 to prove that
$$N(t,x):=\overline{u}_t-\hbox{div}(A(x)\nabla\overline{u})+q(x)\cdot \nabla\overline{u}-f(x,\overline{u})\ge 0,$$
for $t\ge t_2$ and $x\in \overline{\mathcal{H}_j(L_1)}$ such that $\overline{u}(t,x)<1$. By the comparison principle, one obtains that
$$u(t,x)\le \overline{u}(t,x)\le u^j_r(t-t_2+\omega +\tau_2,x) +\delta e^{-\delta (t-t_2)} +\widetilde{\delta} e^{-\lambda_j(x\cdot e_j-L_1)}\psi_j(x),$$
for $t\ge t_2$ and $x\in\overline{\mathcal{H}_j(L_1)}$.
\end{proof}
\vskip 0.3cm

By Lemma~\ref{lemma3.1}, one can actually get the following lemma.

\begin{lemma}\label{lemma3.2}
For every $j\in J_1$ and any $\varepsilon>0$, there exist $L_{\varepsilon}>0$, $t_{\varepsilon}\in\R$ and $\tau_{\varepsilon}\in\R$ such that
$$u(t,x)\ge u^j_r(t-t_{\varepsilon}+\tau_{\varepsilon},x)-\varepsilon e^{-\delta (t-t_{\varepsilon})}-\widetilde{\varepsilon}  e^{- \lambda_j (x\cdot e_j-L_{\varepsilon})}\psi_j(x),$$
for all $t\ge t_{\varepsilon}$ and $x\in\overline{\mathcal{H}_j(L_{\varepsilon})}$, where $\widetilde{\varepsilon}=\varepsilon/\|\psi_j\|_{L^{\infty}(\overline{\mathcal{H}_j(L)})}$ and  $\delta>0$ is defined as in Lemma~$\ref{lemma3.1}$.
\end{lemma}

\begin{proof}
For any $\varepsilon>0$, let $\hat{\varepsilon}:=\varepsilon/\delta$ and $\varepsilon':=\hat{\varepsilon}\widetilde{\delta}\inf_{\overline{\mathcal{H}_j(L)}}\psi_j(x)\le \varepsilon$.
If $\varepsilon\ge \delta$, then it follows from Lemma~\ref{lemma3.1} that the conclusion of Lemma~\ref{lemma3.2} holds for $t_{\varepsilon}=t_2$ as in Lemma~\ref{lemma3.1}. Now we consider $0<\varepsilon<\delta$. By \eqref{noblock},  there are $t_{\varepsilon}>0$ and $L_{\varepsilon}$ such that for $t\ge t_{\varepsilon}$,
$$u(t,x)\ge 1-\varepsilon'\ge 1-\varepsilon, \hbox{ for $x\in\overline{\mathcal{H}_j}$ such that $L_{\varepsilon}\le x\cdot e_j\le L_{\varepsilon}+2M_{\varepsilon}$+R},$$
where $R$ is a fixed constant such that $R\ge \hat{\varepsilon}c^j_r \omega$.
Then, as the proof for Lemma~\ref{lemma3.1}, one can show that the following function
$$\underline{u}(t,x)=\max\Big(u^j_r(t-t_{\varepsilon}+\hat{\varepsilon}\omega e^{-\delta (t-t_{\varepsilon})} -\hat{\varepsilon}\omega+\tau_{\varepsilon},x) -\hat{\varepsilon}\delta e^{-\delta (t-t_{\varepsilon})} -\hat{\varepsilon}\widetilde{\delta} e^{-\lambda_j(x\cdot e_j-L_{\varepsilon})}\psi_j(x),0\Big),$$
where $\tau_{\varepsilon}=(L_{\varepsilon}+M_{\varepsilon}+R)/ c^j_{r}$ is a subsolution of the problem satisfied by $u(t,x)$ for $t\ge t_{\varepsilon}$ and $x\in\overline{\mathcal{H}_j(L_{\varepsilon})}$. Then, the conclusion follows from the comparison principle.
\end{proof}
\vskip 0.3cm

As soon as Lemma \ref{lemma3.1} and Lemma \ref{lemma3.2} provided, one can get the next lemma about the local stability of the almost-planar front in the branch $\mathcal{H}_j$ for $j\in J_1$.

\begin{lemma}\label{lemma3.3}
There is $N\ge0$ such that, if there are $L_1\ge L$, $j\in J_1$, $\varepsilon>0$, $t_0\in\R$ and $\tau\in\R$ such that
$$\sup_{x\in\overline{\mathcal{H}_j(L_1)}}|u(t_0,x)-u^j_r(t_0+\tau,x)|\le \varepsilon$$
together with $t_0$ being sufficiently large such that $u^j_r(t+\tau,x)\ge 1-\epsilon$ and $u(t,x)\ge1-\epsilon$ for all $t\ge t_0$ and $x\in\overline{\mathcal{H}_j}$ with $x\cdot e_j=L_1$, then it holds
$$\sup_{x\in\overline{\mathcal{H}_j(L_1)}} |u(t,x)-u^j_r(t+\tau,x)|\le N\,\varepsilon\ \ \text{for all $t\ge t_0$}.$$
\end{lemma}

\begin{proof}
Let $\delta>0$ and $\omega>0$ be defined as in Lemma~\ref{lemma3.1}. Define $\hat{\epsilon}$ as in Lemma~\ref{lemma3.2}. Since $\sup_{x\in\overline{\mathcal{H}_j(L)}}|u(t_0,x)-u^j_r(t_0+\tau,x)|\le \varepsilon$, it follows from similar arguments to those of Lemmas~\ref{lemma3.1} and~\ref{lemma3.2} that the following functions
$$\max\left(u^j_r(t+\hat{\varepsilon}\omega e^{-\delta (t-t_0)}-\hat{\varepsilon}\omega+\tau,x)-\hat{\varepsilon}\delta e^{-\delta (t-t_0)}-\hat{\varepsilon}\widetilde{\delta} e^{- \lambda_j (x\cdot e_j-L_1)} \psi_j(x), 0\right)$$
and
$$\min\left(u^j_r(t-\hat{\varepsilon}\omega e^{-\delta (t-t_0)}+\hat{\varepsilon}\omega +\tau,x)+\hat{\varepsilon}\delta e^{-\delta (t-t_0)}+\hat{\varepsilon}\widetilde{\delta} e^{- \lambda_j (x\cdot e_j-L_1)}\psi_j(x),1\right)$$
are respectively a sub-solution and a super-solution of the problem satisfied by $u(t,x)$ for all $t\ge t_0$ and $x\in\overline{\mathcal{H}_j(L_1)}$. It then follows  that
\begin{align*}
u^j_r(t+\hat{\varepsilon}\omega e^{-\delta (t-t_0)}-\hat{\varepsilon}\omega&+\tau,x)-\hat{\varepsilon}\delta e^{-\delta (t-t_0)}-\hat{\varepsilon}\widetilde{\delta} e^{- \lambda_j (x\cdot e_j-L_1)} \psi_j(x)\le u(t,x)\\
&\le u^j_r(t-\hat{\varepsilon}\omega e^{-\delta (t-t_0)}+\hat{\varepsilon}\omega +\tau,x)+\hat{\varepsilon}\delta e^{-\delta (t-t_0)}+\hat{\varepsilon}\widetilde{\delta} e^{- \lambda_j (x\cdot e_j-L_1)}\psi_j(x)
\end{align*}
for all $t\ge t_0$ and $x\in\overline{\mathcal{H}_j(L_1)}$. For these $t$ and $x$, since $(u^j_r)_t>0$, one infers that
$$u(t,x)\le u^j_r(t+\hat{\varepsilon}\omega+\tau,x)+2\hat{\varepsilon}\delta \le u^j_r(t+\tau,x)+ \hat{\varepsilon}\omega\|(u^j_r)_t\|_{L^{\infty}}+2\hat{\varepsilon}\delta.$$
Similarly, one can prove that $u(t,x)\ge u^j_r(t+\tau,x)-\omega \hat{\varepsilon}\|(u^j_r)_t\|_{L^{\infty}}-2\hat{\varepsilon}\delta$. As a consequence, one has
$$\sup_{x\in\overline{\mathcal{H}_j(L_1)}} |u(t,x)-u^j_r(t+\tau,x)|\le \omega \hat{\varepsilon}\|(u^j_r)_t\|_{L^{\infty}}+2\hat{\varepsilon}\delta=N\epsilon\ \text{ for all $t\ge t_0$},$$
with the constant $N=\max_{j\in \{1,\cdots,m\}}\Big(\omega\|(u^j_r)_t\|_{L^{\infty}}/\delta+2\Big)$ being independent of $j$, $\epsilon$, $t_0$ and $\tau$.
\end{proof}
\vskip 0.3cm

\begin{proof}[Proof of Theorem~\ref{th2}]
Let $L_1>0$, $t_1\in\R$, $t_2\in\R$, $\tau_1\in\R$,  $\tau_2\in\R$ and $\delta>0$ be as in Lemma~\ref{lemma3.1}. For every $j\in J_1$, $t\ge\max(t_1,t_2)$ and $x\in\overline{\mathcal{H}_j(L_1)}$, there holds
\be\label{lug}\baa{l}
u^j_r(t-t_1+\tau_1,x)-\delta e^{-\delta (t-t_1)}-\widetilde{\delta} e^{- \lambda_j (x\cdot e_j-L_1)}\psi_j(x)\vspace{3pt}\\
\qquad\qquad\qquad\le u(t,x)\le u^j_r(t-t_2+\tau_2,x)+\delta e^{-\delta (t-t_2)}+\widetilde{\delta} e^{- \lambda_j (x\cdot e_j-L_1)}\psi_j(x).\eaa
\ee

Consider now any sequence $\{t_n\}_{n\in\mathbb{N}}$ such that $t_n\rightarrow +\infty$ as $n\rightarrow +\infty$, and consider any~$j\in J_1$. For every $n\in\N$, let~$\mathcal{H}_{j}^n=\mathcal{H}_j-c^j_r t_ne_j$, $A_n(x)=A(x+c^j_r t_n e_j)$,  $q_n(x)=q(x+c^j_r t_n e_j)$ and $f_n(x,u)=f(x+c^j_r t_n e_j)$. By Assumption~\ref{assumption-limiting}, there is an infinite cylinder $\mathcal{H}_j^{\infty}$ parallel to $e_j$ such that $\mathcal{H}_j^n$ converge locally uniformly to $\mathcal{H}_j^{\infty}$ and there are $A_{\infty}(x)$, $q_{\infty}(x)$, $f_{\infty}(x,\cdot)$ satisfying \eqref{wA} such that $A_n\rightarrow A_{\infty}$, $q_n\rightarrow q_{\infty}$, $f_n\rightarrow f_{\infty}$ locally uniformly in $\mathcal{H}_j^{\infty}$ as $n\rightarrow +\infty$. Let $v_n(t,x)=u^j_r(t+t_n,x+c^j_r t_n e_j)$. Since $u^j_r(t,x)$ is an almost-planar, one has that for any $\varepsilon>0$, there is $M_{\varepsilon}>0$ such that
\begin{eqnarray}\label{vnMe}
\left\{\begin{array}{lll}
v_n(t,x)\ge 1-\varepsilon, && \hbox{ for $t\in\R$ and $x\in\overline{\mathcal{H}^n_j}$ such that $x\cdot e_j-c^i_r t\le -M_{\varepsilon}$},\\
v_n(t,x)\le \varepsilon, && \hbox{ for $t\in\R$ and $x\in\overline{\mathcal{H}^n_j}$ such that $x\cdot e_j-c^i_r t\ge M_{\varepsilon}$}.
\end{array}
\right.
\end{eqnarray}
From standard parabolic estimates, up to extraction of a subsequence, the functions $v_n(t,x)$ converge locally uniformly to a solution $v_{\infty}(t,x)$ of
\begin{eqnarray}\label{eq-infinite}
\left\{\begin{array}{lll}
&u_t-$div$(A_{\infty}(x)\nabla u)+q_{\infty}(x)\cdot \nabla u=f_{\infty}(x,u), \quad &t\in\R,\ x\in\mathcal{H}_j^{\infty},\\
&\nu A_{\infty}(x)\nabla u=0,& t\in\R,\ x\in\partial\mathcal{H}_j^{\infty}.
\end{array}
\right.
\end{eqnarray}
By \eqref{vnMe}, $v_{\infty}(t,x)$ is still an almost-planar front connecting $0$ and $1$ facing direction $e_j$ with
\be\label{sets-v}
\hbox{$\Gamma_t=\{x\in\overline{\mathcal{H}_j^{\infty}}; x\cdot e_j=c^j_r t\}$, $\Omega_t^{+}=\{x\in\overline{\mathcal{H}_j^{\infty}}; x\cdot e_j<c^j_r t\}$ and $\Omega_t^{-}=\{x\in\overline{\mathcal{H}_j^{\infty}}; x\cdot e_j>c^j_r t\}$. }
\ee
Then, $v_{\infty}(t,x)=v_r(t,x)$ up to shifts by Assumption~\ref{assumption-limiting}.

Now, let $u_n(t,y)=u(t+t_n,y+c^j_r t_ne_j)$ defined in $\R\times\overline{\Omega-c^j_r t_ne_j}$.
From standard parabolic estimates, up to extraction of a subsequence, the functions $u_n(t,x)$, converge locally uniformly in $(t,y)\in\R\times\overline{\mathcal{H}_j^\infty}$ to a solution $u_{\infty}(t,y)$~of \eqref{eq-infinite}.
It follows from \eqref{lug} that
\begin{align*}
v_{\infty}(t-t_1+\tau_1,y)\le u_{\infty}(t,y)\le v_{\infty}(t-t_2+\tau_2,y)
\end{align*}
for all $(t,y)\in\R\times\overline{\mathcal{H}_j^{\infty}}$. In particular, $u_\infty$ is an almost-planar front connecting $0$ and $1$ facing direction $e_j$ of~\eqref{eq-infinite} in the cylinder $\mathcal{H}_j^\infty$ with sets $\Gamma_t$ and $\Omega_t^{\pm}$ defined by \eqref{sets-v}. By Assumption~\ref{assumption-limiting}, there is $\tau_j\in\R$ such that $u_{\infty}(t,y)=v_{\infty}(t+\tau_j,y)$ for all $(t,y)\in\R\times\overline{\mathcal{H}_j^{\infty}}$. Therefore,
\be\label{2.22}
u_n(t,y)\rightarrow v_{\infty}(t+\tau_j,y)\ \ \text{locally uniformly in $\R\times\overline{\mathcal{H}_j^{\infty}}$ as $n\rightarrow +\infty$}.
\ee

Remember that
\begin{equation*}
v_n(t+\tau_j,y)=u^j_r(t+t_n+\tau_j,y+c^j_r t_n e_j)\rightarrow v_{\infty}(t+\tau_j,y)\ \ \text{locally uniformly in $\R\times\overline{\mathcal{H}_j^{\infty}}$ as $n\rightarrow +\infty$}.
\end{equation*}
Pick now any $\varepsilon>0$, and let $M_{\varepsilon}>0$ be defined by \eqref{vnMe}. Let $L_{\varepsilon}\ge L_1$ such that
\be\label{eq-Le}
\delta e^{- \lambda_j (x\cdot e_j-L_1)}\le \frac{\varepsilon}{3} \hbox{ for $x\cdot e_j\ge L_{\varepsilon}$}.
\ee
Define $K=c^j_r\max(|t_1-\tau_1+\tau_j|, |t_2-\tau_2+\tau_j|, |\tau_j|)$. It then follows from \eqref{2.22} that
\be\label{eq+4.12}
\sup_{y\in\overline{\mathcal{H}_j^{n}},\ |y\cdots e_j-c^j_r \tau_j|\le M_{\varepsilon/2}+K}|u_n(0,y)-v_n(\tau_j,y)|\le\varepsilon\ \  \text{for $n$ large enough}.
\ee
Since $t_n\rightarrow +\infty$ as $n\rightarrow +\infty$,~\eqref{lug} and \eqref{eq-Le} imply that, for $n$ large enough and $y\in \overline{\mathcal{H}_j^n}$ such that $y\cdot e_j\ge L_{\varepsilon}-c^j_r t_n$,
$$
u^j_r(t+t_n-t_1+\tau_1,y+c^j_r t_ne_j)-\frac{\varepsilon}{2}  \le u_n(t,y)\le u^j_r(t+t_n-t_2+\tau_2,y+c^j_r t_ne_j)+\frac{\varepsilon}{2}.
$$
Therefore, by \eqref{vnMe}, one has that for $n$ large enough,
\be\label{une}\left\{\baa{ll}
0<u_n(0,y)\le \varepsilon & \text{for all $y\in\overline{\mathcal{H}_j^n}$ such that $y\cdot e_j-c^j_r \tau_j\ge M_{\varepsilon/2}+K $},\vspace{3pt}\\
1-\varepsilon\le u_n(0,y) <1 & \displaystyle\text{for all $y\in\overline{\mathcal{H}_j^n( L_{\varepsilon}-c^j_r t_n)}$ such that $ y\cdot e_j-c^j_r \tau_j\le M_{\varepsilon/2}+K$}.
\eaa\right.
\ee
Since $v_n(\tau_j,y)=u^j_r(t_n+\tau_j,y+c^j_r t_n e_j)$,  one has $0<v_n(\tau_j,y)\le \varepsilon/2\le\epsilon$ for all $y\in\overline{\mathcal{H}_j^n}$ such that $y\cdot e_j-c^j_r \tau_j\ge M_{\varepsilon/2}+K $, and $1-\varepsilon\le1-\epsilon/2\le v_n(\tau_j,y)<1$ for all $y\in\overline{\mathcal{H}_j^n}$ such that $ y\cdot e_j-c^j_r\tau_j\le -M_{\varepsilon/2}-K$. It then can be deduced from \eqref{une} that, for $n$ large enough,
$$|u_n(0,y)-v_n(\tau_j,y)|\le \varepsilon$$
for all $y\in\overline{\mathcal{H}_j^n}$ such that $y\cdot e_j-c^j_r \tau_j\ge M_{\varepsilon/2}+K$ and $y\in\overline{\mathcal{H}_j^n(L_{\varepsilon}-c^j_r t_n)}$ such that $ y\cdot e_j-c^j_r \tau_j\le -M_{\varepsilon/2}-K$.
By the definitions of $u_n(t,y)$, $v_n(t,y)$ and $\mathcal{H}_j^n$ together with \eqref{eq+4.12}, one gets that, for $n$ large enough,
\be\label{eq+4.14}
|u(t_n,x)-u^j_r(t_n+\tau_j,x)|\le \varepsilon\ \  \text{for all $x\in\overline{\mathcal{H}_j}$ such that $x\cdot e_j\ge L_{\varepsilon}$}.
\ee
It then follows from Lemma~\ref{lemma3.3} that, for $n$ large enough, $|u(t,x)-u^j_r(t+\tau_j,x)|\le N\epsilon$ for all $t\ge t_n$ and $x\in\overline{\mathcal{H}_j(L_{\varepsilon})}$, where the constant $N\ge0$ is given in Lemma~\ref{lemma3.3}. One can take $\varepsilon>0$ arbitrary small by taking $L_{\varepsilon}$ large enough. Notice that the choice of $\tau_j$ is independent of $\varepsilon$ and $L_{\varepsilon}$. Then, one concludes that
$$u(t,x)-u^j_r(t+\tau_j,x)\to 0\ \ \text{uniformly for $x\in\overline{\mathcal{H}_j}$ such that $x\cdot e_j\ge L_{\varepsilon}$, as $t\rightarrow +\infty$ and $L_{\varepsilon}\rightarrow +\infty$}.$$
In particular, we can take $L_{\varepsilon}=\mu t$ for any positive constant $\mu$.

The proof of Theorem~\ref{th2} is thereby complete.
\end{proof}
\vskip 0.3cm

\begin{proof}[Proof of Corollary~\ref{cor2}]
By Corollary~\ref{cor1}, the complete propagation of $u(t,x)$ means that $J_1\equiv J$.
We now only have to modify slightly in the proof of Theorem~\ref{th2}. By the proof of Theorem~\ref{th2}, one can get that for any $\varepsilon>0$, there exist a sequence $\{t_n\}_{n\in \mathbb{N}}$ such that $t_n\rightarrow +\infty$ and a constant $L_{\varepsilon}\ge L_1$ ($L_1$ is defined in Lemma~\ref{lemma3.1}) such that \eqref{eq+4.14} holds for large $n$. Since the propagation of $u(t,x)$ is complete, it implies that $u(t_n,x)\ge 1-\varepsilon$ for $x\in\overline{\mathcal{H}_j}$ such that $L_1\le x\cdot e_j\le L_{\varepsilon}$ for large $n$.  By the definition of $u^j_r(t,x)$, one also knows that $u^j_r(t_n+\tau_j,x) \ge 1-\varepsilon$ for $x\in\overline{\mathcal{H}_j}$ such that $L_1\le x\cdot e_j\le L_{\varepsilon}$ and large $n$. Then, by \eqref{eq+4.14}, one has $|u(t_n,x)-u^j_r(t_n+\tau_j,x)|\le \varepsilon$ for all $x\in\overline{\mathcal{H}_j}$ such that $x\cdot e_j\ge L_1$. Finally, it follows from Lemma~\ref{lemma3.3} and $\varepsilon$ being arbitrary small that
$$u(t,x)-u^j_r(t+\tau_j,x)\to 0\ \ \text{uniformly for $x\in\overline{\mathcal{H}_j}$ such that $x\cdot e_j\ge L_1$, as $t\rightarrow +\infty$}.$$
By \eqref{frontlike} and $u(t,x)$ converges to $1$ locally uniformly in $\overline{\Omega}$, one can easily get \eqref{largetime}. Then, it is elementary to check that $u$ is a transition front connecting $0$ and $1$ of \eqref{eq1.1} defined by Definition~\ref{TF} with sets $(\Gamma_t)_{t\in\R}$ and $\Omega_t^{\pm}$ defined by \eqref{eq+1.11} and \eqref{eq+1.12}.
\end{proof}

%%%%%%%%%%%%%%%%%%%%%%%%%%%%%%%%%%%%%%%%%%%%%%%
\section{Uniqueness of transition fronts}
In this section, we study the uniqueness of the transition front connecting $0$ and $1$, that is, Theorem~\ref{th3}. Let $u(t,x)$ be any transition front connecting $0$ and $1$ of \eqref{eq1.1}. In the sequel, we always assume that for every $i\in\{1,\cdots,m\}$, the entire solution $u_i$ of \eqref{front} propagates completely. Then, by Corollary~\ref{cor1}, it implies that $c^i_r>0$ for all $i\in\{1,\cdots,m\}$. One also has that $c^i_l>0$ for all $i\in \{1,\cdots,m\}$. In fact, if $i\in\{1,\cdots,m\}$ such that $c^i_l<0$, one can replace $u$ and $f(x,u)$ by $v:=1-u$ and $g(x,v):=-f(x,1-u)$. Then, $v(t,x)$ satisfies
\begin{eqnarray*}
\left\{\begin{array}{lll}
&v_t-$div$(A(x)\nabla v)+q(x)\cdot \nabla v=g(x,v), \quad &t\in\R,\ x\in\Omega,\\
&\nu A(x)\nabla v=0,& t\in\R,\ x\in\partial\Omega.
\end{array}
\right.
\end{eqnarray*}
Let $v_i(t,x)=1-u_i(t,x)$ and $v^i_r(t,x)=1-u^i_l(t,x)$. Then, $v^i_r(t,x)$ is an almost-planar front facing $e_i$ with speed $-c^i_l>0$ and $v^i_r(t,x)\rightarrow 1$ locally uniformly in $\overline{\widetilde{\mathcal{H}}_i}$ as $t\rightarrow +\infty$. Since $u_i(t,x)$ satisfies \eqref{front}, one has that for any $\varepsilon>0$ and $L_1\ge L$, there is $t_0<0$ such that
$$v_i(t_0,x)\ge 1-\varepsilon \hbox{ for $\overline{\Omega}\cap\overline{\mathcal{H}_i}$ and $x\in \overline{\mathcal{H}_i}$ such that $x\cdot e_i\le L_1$}.$$
Then, by the proof of Theorem \ref{th2}, there is $\tau\in \R$ such that $v_i(t,x)-v^i_r(t,x)\rightarrow 0$ for $x\in\overline{\mathcal{H}_i}$ such that $x\cdot e_i\ge \eta t$ as $t\rightarrow +\infty$, where $\eta$ is an arbitrary positive constant. This contradicts the complete propagation of $u_i(t,x)$.

For any $i\in \{1,\cdots,m\}$ and any $R>0$, define
$$\mathcal{H}_i(R):=\{x\in\mathcal{H}_i; x\cdot e_i\ge R\}.$$
Notice that $\mathcal{H}_i(R)\subset \Omega$ for any $R\ge L$. Let $\delta$ be a positive constant such that
\be\label{delta-s4}
0<\delta\le \min\Big(\gamma,\frac{\sigma}{3}\Big),
\ee
where $\gamma$ and $\sigma$ are defined by \eqref{eq-F}.

\subsection{Preliminaries}
In this subsection, we study some properties of entire solutions emanating from almost-planar fronts and two initial value problems.

\begin{lemma}\label{lemma4.1}
For any $\eta\in (0,1/2]$, denote
$$\Omega_{\eta}=\{(t,x)\in\R\times\overline{\Omega}; \eta\le u_i(t,x)\le 1-\eta\}.$$
There is $k>0$ such that $(u_i)_t(t,x)\ge k$ for any $(t,x)\in \Omega_{\eta}$.
\end{lemma}

\begin{proof}
It can be proved similarly by the proof of Lemma~\ref{muil}. Since $u_i(t,x)$ propagates completely, it follows from Corollary~\ref{cor2} that $u_i(t,x)$ has large time behaviour as \eqref{largetime} for $J=\{1,\cdots,m\}\setminus \{i\}$. Then, one only has to analyze one more case, that is, $t_n\rightarrow +\infty$ by similar arguments for the case $t_n\rightarrow -\infty$ in the proof of Lemma~\ref{muil}. Here, we omit the details.
\end{proof}
\vskip 0.3cm

Let $I$ be any non-empty subset of $\{1,\cdots,m\}$ such that $\{1,\cdots,m\}\setminus I\neq \emptyset$. Let $(\tau_i)_{i\in I}$ be a family of non-positive constants. Let $u_{I,\tau_i}(t,x)$ be the entire solution emanating from almost-planar fronts $u^i_l(t+\tau_i,x)$ in branches $\mathcal{H}_i$ where $i\in I$, that is,
\be\label{eq-uI}
\left\{\baa{rcll}
u_{I,\tau_i}(t,x)\!-\!u^i_l(t+\tau_i,x) & \!\!\!\to\!\!\! & 0\!\! & \text{uniformly in $\overline{\mathcal{H}_i}\cap\overline{\Omega}\,$ for every $i\in I$},\vspace{3pt}\\
u_{I,\tau_i}(t,x) & \!\!\to\!\! & 0\!\! & \displaystyle\text{uniformly in $\overline{\Omega\setminus \mathop{\bigcup}_{i\in I}\mathcal{H}_i}$},
\eaa\right.
\ee
as $t\rightarrow -\infty$.

\begin{lemma}\label{lemma-uI}
The entire solution $u_{I,\tau_i}(t,x)$ propagates completely, that is, $u_{I,\tau_i}(t,x)\rightarrow 1$ locally uniformly in $\overline{\Omega}$ as $t\rightarrow +\infty$.
\end{lemma}

\begin{proof}
Fix any $i\in I$. Let $u_i(t,x)$ be the entire solution emanating from $u^i_l(t,x)$ in the branch $\mathcal{H}_i$, that is, satisfying \eqref{front}. Then, there is $T_1<0$ such that
\be\label{uile}
\left\{\baa{lll}
u_i(t+\tau_i,x)\le u^i_l(t+\tau_i,x)+\frac{\delta}{2} && \text{for $t\le T_1$ and $x\in \overline{\mathcal{H}_i(L)}$},\vspace{3pt}\\
u_i(t+\tau_i,x)\le \frac{\delta}{2} && \text{for $t\le T_1$ and $x\in\overline{\Omega\setminus\mathcal{H}_i(L)}$}.
\eaa\right.
\ee
By \eqref{eq-uI}, even if it means decreasing $T_1<0$, one has that
\be\label{uIge}
u_{I,\tau_i}(T_1,x)\ge u^i_l(T_1+\tau_i,x)-\frac{\delta}{2} \text{ for $x\in \overline{\mathcal{H}_i(L)}$}.
\ee
Define
$$\Omega_{\xi}^+(\delta):=\{\xi\in\R, x\in\overline{\Omega}; u_i(\xi,x)\ge 1-\delta\} \hbox{ and } \Omega_{\xi}^-(\delta):=\{\xi\in\R, x\in\overline{\Omega}; u_i(\xi,x)\le \delta\}.$$
Then $\delta\le u_i(\xi,x)\le 1-\delta$ for $(\xi,x)\in\overline{\Omega}\setminus \Big(\Omega_{\xi}^+(\delta)\cup\Omega_{\xi}^-(\delta)\Big)$. By Lemma~\ref{lemma4.1}, there is $k>0$ such that $(u_i)_\xi(\xi,x)\ge k$ for $(\xi,x)\in\overline{\Omega}\setminus \Big(\Omega_{\xi}^+(\delta)\cup\Omega_{\xi}^-(\delta)\Big)$. Let $\omega>0$ such that
\be\label{eq-omega2}
k\omega\ge \delta+M,
\ee
where $M=\sup_{x\in \R^N, u\in [0,1]} |f_u(x,u)|$.

For any $t\ge 0$ and $x\in\overline{\Omega}$, define
\be\label{un-u}
\underline{u}(t,x)=\max\Big(u_{i}(\underline{\zeta}(t),x)-\delta e^{-\delta t},0\Big),
\ee
where
$$\underline{\zeta}(t)=t+T_1+\tau_i+\omega e^{-\delta t}-\omega$$
We check that $\underline{u}(t,x)$ is a subsolution of the problem satisfied by $u_{I,\tau_i}(t,x)$.

At the time $t=0$, one has that $\underline{\zeta}(0)=T_1+\tau_i\le T_1$ since $\tau_i$ is non-positive. Then, by \eqref{uile} and \eqref{uIge},
$$\underline{u}(0,x)\le \max\Big(u_{i}(T_1+\tau_i,x)-\delta,0\Big)\le u^i_l(t,x)-\frac{\delta}{2}\le u_{I,\tau_i}(T_1,x) \hbox{ for $x\in\overline{\mathcal{H}_i(L)}$},$$
and
$$\underline{u}(0,x)\le \max\Big(\frac{\delta}{2}-\delta,0\Big)=0\le u_{I,\tau_i}(T_1,x) \hbox{ for $x\in\overline{\Omega\setminus\mathcal{H}_i(L)}$}.$$
Thus, $\underline{u}(0,x)\le u_{I,\tau_i}(T_1,x)$ for all $x\in\overline{\Omega}$. It is obvious that $\nu A(x)\nabla\underline{u}= 0$ on $x\in\partial \Omega$.

Now, let us check that
$$N(t,x):=\underline{u}_t-\hbox{div}(A(x)\nabla\underline{u})+q(x)\cdot \nabla\underline{u}-f(x,\underline{u})\le 0,$$
for any $t\ge 0$ and $x\in\overline{\Omega}$ such that $\underline{u}(t,x)>0$. It follows from some calculation that
\begin{align*}
N(t,x)= -\omega \delta e^{-\delta t} (u_{i})_t(\underline{\zeta}(t),x) +\delta^2 e^{-\delta t}+f((u_{i}(\underline{\zeta}(t),x))-f(x,\underline{u}(t,x)).
\end{align*}
For $t\ge 0$ and $x\in\overline{\Omega}$ such that $(\underline{\zeta}(t),x)\in\Omega_{\underline{\zeta}(t)}^-(\delta)$, it follows that $0<u_{i}(\underline{\zeta}(t),x)\le \delta$ and $\underline{u}(t,x)\le \delta$. Then, by \eqref{eq-F} and \eqref{delta-s4}, one has that
$$f(x,u_{i}(\underline{\zeta}(t),x))-f(x,\underline{u}(t,x))\le -\gamma \delta e^{-\delta t}.$$
Thus,
\begin{align*}
N(t,x)\le -\omega\delta e^{-\delta t}(u_{i})_t(\underline{\zeta}(t),x) +\delta(-\gamma+\delta) e^{-\delta t}\le 0,
\end{align*}
by $(u_{i})_t(\overline{\zeta}(t),x)>0$ and $\delta\le \gamma$. For $t\ge 0$ and $x\in\overline{\Omega}$ such that $(\underline{\zeta}(t),x)\in\Omega_{\underline{\zeta}(t)}^+(\delta)$, it follows that $u_{i}(\underline{\zeta}(t),x)\ge 1-\delta$ and $\underline{u}(t,x)\ge 1-2\delta$. Then, by by \eqref{eq-F} and \eqref{delta-s4},
$$f(x,u_{i}(\underline{\zeta}(t),x))-f(x,\underline{u}(t,x))\le -\gamma\delta e^{-\delta t}.$$
Thus,
$$N(t,x)\le -\omega\delta e^{-\delta t} (u_{i})_t(\underline{\zeta}(t),x)+\delta(-\gamma+\delta) e^{-\delta t}\le 0,$$
by $(u_{i})_t(\underline{\zeta}(t),x)>0$ and $\delta\le \gamma$.

Finally, for $t\ge 0$ and $x\in\overline{\Omega}$ such that $(\underline{\zeta}(t),x)\in\overline{\Omega}\setminus \Big(\Omega_{\underline{\zeta}(t)}^+(\delta)\cup\Omega_{\underline{\zeta}(t)}^-(\delta)\Big)$, there is $k>0$ such that $(u_{i})_t(\underline{\zeta}(t),x)\ge k>0$. It is obvious that
$$f(x,u_{i}(\underline{\zeta}(t),x))-f(x,\underline{u}(t,x))\le M\delta e^{-\delta t}.$$
Therefore,
\begin{align*}
N(t,x)\le& -k\omega \delta e^{-\delta t}+\delta^2 e^{-\delta t}+M\delta e^{-\delta t}\\
\le& -k\omega \delta e^{-\delta t} +\delta (\delta +M) e^{-\delta t} \le 0,
\end{align*}
by \eqref{eq-omega2}.

By the comparison principle, one concludes that
$$u_{I,\tau_i}(t,x)\ge \underline{u}(t-T_1,x)\ge u_{i}(t+\tau_i+\omega e^{-\delta (t-T_1)}-\omega,x)-\delta e^{-\delta (t-T_1)}, \hbox{ for $t\ge T_1$ and $x\in\overline{\Omega}$}.$$
Since $u_i(t,x)$ propagates completely, one has that $u_{I,\tau_i}(t,x)\rightarrow 1$ locally uniformly in $\overline{\Omega}$ as $t\rightarrow +\infty$.
 This completes the proof.
\end{proof}
\vskip 0.3cm

Similar as Lemma~\ref{lemma4.1}, we have the following corollary.

\begin{corollary}\label{cor-m}
For any $\eta\in (0,1/2]$, denote
$$\Omega^I_{\eta}=\{(t,x)\in\R\times\overline{\Omega}; \eta\le u_{I,\tau_i}(t,x)\le 1-\eta\}.$$
There is $k>0$ such that $(u_{I,\tau_i})_t(t,x)\ge k$ for any $(t,x)\in \Omega^I_{\eta}$.
\end{corollary}

Since $u_{I,\tau_i}(t,x)$ propagates completely, it is a transition front connecting $0$ and $1$ by Corollary~\ref{cor2}. By \eqref{eq-uI}, one can easily check that there is $T_1<0$ such that $\Gamma_t$, $\Omega_t^{\pm}$ ($t\le T_1$) of $u_{I,\tau_i}(t,x)$ can be denoted by
$$\Gamma_t=\cup_{i\in I} \{x\in \mathcal{H}_i; x\cdot e_i=-c^i_l(t+\tau_i)\}, \hbox{ for $t\le T_1$},$$
and
$$\Omega_t^+=\cup_{i\in I} \{x\in \mathcal{H}_i; x\cdot e_i>-c^i_l(t+\tau_i)\}, \ \Omega^-=\Omega\setminus\overline{\Omega_t^+}, \hbox{ for $t\le T_1$}.$$
Here, even if it means decreasing $T_1$, we assume that $T_1$ is negative enough such that
$$-c^i_l(t+\tau_i)\ge L \hbox{ for all $t\le T_1$ and $i\in I$}.$$
Since $\tau_i\le 0$ for all $i\in I$ and by looking back at the construction of sub- and supersolutions for the existence of the entire solution $u(t,x)$ satisfying Theorem~\ref{Th1}, one knows that $T_1$ can be taken independent of the choice of $I$ and $\tau_i$. Moreover, for $\delta>0$ defined by \eqref{delta-s4}, it follows from Definition~\ref{TF} and \eqref{disO} that there is $A_{\delta}>0$ such that
\begin{eqnarray}\label{Ad}
\left\{\begin{array}{lll}
u_{I,\tau_i}(t,x)\ge 1-\delta &&\hbox{ for $t\le T_1$ and $x\in \overline{\cup_{i\in I} \mathcal{H}_i(-c^i_l(t+\tau_i)+A_{\delta})}$},\\
u_{I,\tau_i}(t,x)\le \delta && \hbox{ for $t\le T_1$ and $x\in \overline{\Omega\setminus\cup_{i\in I} \mathcal{H}_i(-c^i_l(t+\tau_i)-A_{\delta})}$}.
\end{array}
\right.
\end{eqnarray}
Notice that $A_{\delta}$ is independent of the choice of $I$ and $\tau_i$. One may need to decrease $T_1$ again such that the sets in \eqref{Ad} are well-defined for $t\le T_1$.

\begin{lemma}\label{lemma-s1}
Let $I$ be a non-empty subset of $\{1,\cdots,m\}$ and $\{1,\cdots,m\}\setminus I\neq \emptyset$. For any $L_i\ge L$ ($i\in I$), let $v_0(x)$ be an initial value satisfying
\begin{eqnarray*}
v_0(x)=\left\{\begin{array}{lll}
1, &&\hbox{ for $x\in\overline{\cup_{i\in I}\mathcal{H}_i(L_i)}$},\\
\delta, &&\hbox{ for $x\in\overline{\Omega\setminus\cup_{i\in I}\mathcal{H}_i(L_i)}$},
\end{array}
\right.
\end{eqnarray*}
and $v(t,x)$ be the solution of \eqref{eq1.1} for $t\ge 0$ with $v(0,x)=v_0(x)$.
Then, there exist $R>0$ and $\omega>0$ such that for any family of constants $(L_i)_{i\in I}$ satisfying $L_i\ge L+A_{\delta}+R$ for all $i\in I$, there holds that
$$v(t,x)\le  u_{I,\tau_i}(t+\tau+\omega,x)+\delta e^{-\delta t} \hbox{ for $t\ge 0$ and $x\in\overline{\Omega}$},$$
where
$$\tau=\max_{i\in I}\Big(\frac{A_{\delta}-L_i}{c^i_l}\Big) \hbox{ and } \tau+\tau_i=\frac{A_{\delta}-L_i}{c^i_l}.$$
\end{lemma}

\begin{proof}
Let $T_1$ be defined in \eqref{Ad}. Let $R$ be defined by $R:=\min_{i\in\{1,\cdots,m\}} (-c^i_l T_1)>0$ which implies 
\be\label{eq-R}
\max_{i\in \{1,\cdots,m\}} (-R/c^i_l)=T_1.
\ee
For any family of constants $(L_i)_{i\in I}$ satisfying $L_i\ge L+A_{\delta}+R$, let
$$\tau=\max_{i\in I}\Big(\frac{A_{\delta}-L_i}{c^i_l}\Big) \hbox{ and } \tau+\tau_i=\frac{A_{\delta}-L_i}{c^i_l}.$$
Notice that $\tau\le T_1$ by \eqref{eq-R} and $\tau_i\le 0$ for all $i\in I$.
Let $u_{I,\tau_i}(\xi,x)$ be the entire solution satisfying \eqref{eq-uI}. Define
$$\Omega_{\xi}^+(\delta):=\{\xi\in\R, x\in\overline{\Omega}; u_{I,\tau_i}\ge 1-\delta\} \hbox{ and } \Omega_{\xi}^-(\delta):=\{\xi\in\R, x\in\overline{\Omega}; u_{I,\tau_i}\le \delta\}.$$
Then $\delta\le u_{I,\tau_i}(\xi,x)\le 1-\delta$ for $(\xi,x)\in\overline{\Omega}\setminus \Big(\Omega_{\xi}^+(\delta)\cup\Omega_{\xi}^-(\delta)\Big)$. By Corollary~\ref{cor-m}, there is $k>0$ such that $(u_{I,\tau_i})_\xi(\xi,x)\ge k$ for $(\xi,x)\in\overline{\Omega}\setminus \Big(\Omega_{\xi}^+(\delta)\cup\Omega_{\xi}^-(\delta)\Big)$. Let $\omega>0$ such that
$$
k\omega\ge \delta+M,
$$
where $M=\sup_{x\in \R^N, u\in [0,1]} |f_u(x,u)|$.

For any $t\ge 0$ and $x\in\overline{\Omega}$, define
$$\overline{u}(t,x)=\min\Big(u_{I,\tau_i}(\overline{\zeta}(t),x)+\delta e^{-\delta t},1\Big),$$
where
$$\overline{\zeta}(t)=t+\tau-\omega e^{-\delta t}+\omega$$
We only have to check that $\overline{u}(t,x)$ is a supersolution of the problem satisfied by $v(t,x)$.

At the time $t=0$, one has that $\overline{\zeta}(0)=\tau\le T_1$ and
$$\overline{u}(0,x)\ge \min\Big(u_{I,\tau_i}(\tau,x)+\delta,1\Big)\ge \delta\ge v_0(x) \hbox{ for $x\in\overline{\Omega\setminus\cup_{i\in I}\mathcal{H}_i(L_i)}$}.$$
 By the definition of $\tau$ and $\tau_i$, $x\in\overline{\mathcal{H}_i(L_i)}$ implies that $x\cdot e_i\ge -c^i_l (\tau+\tau_i)+A_{\delta}$ for any $i\in I$. It then follows from \eqref{Ad} that
$$\overline{u}(0,x)\ge \Big(1-\delta+\delta,1\Big)=1\ge v_0(x) \hbox{ for  $x\in\overline{\cup_{i\in I}\mathcal{H}_i(L_i)}$}.$$
Thus, $\overline{u}(0,x)\ge v_0(x)$ for all $x\in\overline{\Omega}$. Moreover, it is obvious that $\nu A(x)\nabla\overline{u}= 0$ on $x\in\partial \Omega$.

Similar as the proof of Lemma~\ref{lemma-uI}, one can easily check that
$$N(t,x):=\overline{u}_t-\hbox{div}(A(x)\nabla\overline{u})+q(x)\cdot \nabla\overline{u}-f(x,\overline{u})\ge 0,$$
for any $t\ge 0$ and $x\in\overline{\Omega}$ such that $\overline{u}(t,x)<1$.

Then, by the comparison principle and $(u_{I,\tau_i})_{\xi}(\xi,x)>0$, one concludes that
$$v(t,x)\le \overline{u}(t,x)\le u_{I,\tau_i}(t+\tau+\omega,x)+\delta e^{-\delta t}, \hbox{ for $t\ge 0$ and $x\in\overline{\Omega}$}.$$
 This completes the proof.
\end{proof}
\vskip 0.3cm

Fix any $i\in\{1,\cdots,m\}$. Let $u_i(t,x)$ be the entire solution emanating from $u^i_l(t,x)$ in $\mathcal{H}_i$ satisfying \eqref{front}.
Define $ J:=\{1,\cdots,m\}\setminus \{i\}$.
Let $\widetilde{u}_i(t,x)$ be the entire solution emanating from $u^j_l(t,x)$ in $\mathcal{H}_j$ for all $j\in J$, that is,
\be\label{eq-wui}
\left\{\baa{rcll}
\widetilde{u}_i(t,x)\!-\!u^j_l(t,x) & \!\!\!\to\!\!\! & 0\!\! & \text{uniformly in $\overline{\mathcal{H}_j}\cap\overline{\Omega}\,$ for every $j\in J$},\vspace{3pt}\\
\widetilde{u}_i(t,x) & \!\!\to\!\! & 0\!\! & \displaystyle\text{uniformly in $\overline{\Omega\setminus \mathop{\bigcup}_{j\in J}\mathcal{H}_j}$},
\eaa\right.
\ee
as $t\rightarrow -\infty$. By Lemma~\ref{lemma-uI}, $\widetilde{u}(t,x)$ propagates completely. Then, by Corollary~\ref{cor2}, there exists a real number $\eta_i$ such that
\be\label{ltb-wui}
\left\{\baa{rcll}
\widetilde{u}_i(t,x)-u^i_r(t+\eta_i,x) & \!\!\!\to\!\!\! & 0\!\! & \text{uniformly in $\overline{\mathcal{H}_i}\cap\overline{\Omega}\,$},\vspace{3pt}\\
\widetilde{u}_i(t,x) & \!\!\!\to\!\!\! & 1\!\! & \displaystyle\text{uniformly in $\overline{\Omega\setminus \mathcal{H}_i}$},\eaa\right.
\ee
as $t\rightarrow +\infty$. Assume that $\eta_i=0$ even if it means shifting $\widetilde{u}_i(t,x)$ in time.
Therefore, there is $T_2>0$ and $A_{\delta}$ such that for all $t\ge T_2$,
\begin{eqnarray}\label{eq-wu-Ad}
\left\{\begin{array}{lll}
\widetilde{u}_{i}(t,x)\le \delta &&\hbox{ for $t\ge T_2$ and $x\in \overline{ \mathcal{H}_i(c^i_r t+A_{\delta})}$},\\
\widetilde{u}_{i}(t,x)\ge 1-\delta && \hbox{ for $t\ge T_2$ and $x\in \overline{\Omega\setminus  \mathcal{H}_i(c^i_r t-A_{\delta})}$}.
\end{array}
\right.
\end{eqnarray}
Even if it means increasing $T_2$, assume that $c^i_r t-A_{\delta}\ge L$ for all $t\ge T_2$ which means that the sets in \eqref{eq-wu-Ad} are well-defined.

Before we deduce some properties of $u_i(t,x)$ and $\widetilde{u}_i(t,x)$, we need the following lemma.

\begin{lemma}\label{uil-lemma}
There exist $C>0$, $r_i>0$, $\varphi_i(x)>0$ and $D>0$ such that
$$u^i_l(t,x)\ge 1-C e^{-r_i(x\cdot e_i+c^i_l t)} \varphi_i(x), \hbox{ for $t\in \R$ and $x\in\overline{\widetilde{\mathcal{H}}_i}$ such that $x\cdot e_i\ge -c^i_l t+D$}.$$
\end{lemma}

\begin{proof}
By a similar proof of Lemma~\ref{lemma2.1}, one knows that there exist $r_i>0$ and $\varphi_i(x)\in C^2$ satisfying
\begin{eqnarray}\label{eq-vphi}
\left\{\begin{array}{lll}
&-$div$(\widetilde{A}(x) \nabla \varphi_i) +r_i ($div$(\widetilde{A}(x) e_i \varphi_i )+e_i\widetilde{A}\nabla \varphi_i) +\widetilde{q}(x)\cdot \nabla\varphi_i&\\
&-r_i \widetilde{q}(x) e_i \varphi_i -r_i^2 (e_i\widetilde{A}(x) e_i) \varphi_i-r_i c^i_l \varphi_i \ge -\gamma \varphi_i, & x\in \widetilde{\mathcal{H}}_i,\\
&\nu \widetilde{A}(x)(r_i \varphi_i(x) e_i +\nabla \varphi_i)\ge 0, &  x\in\partial\widetilde{\mathcal{H}}_i,
\end{array}
\right.
\end{eqnarray}
where $\gamma$ is defined by \eqref{eq-F}, and $0<\inf_{\widetilde{\mathcal{H}}_i}\varphi_i(x)\le \sup_{\widetilde{\mathcal{H}}_i}\varphi_i(x)<+\infty$. Let $\delta>0$ be defined by \eqref{delta-s4}, $\widetilde{\delta}:=\delta/\|\varphi_i\|_{L^{\infty}(\widetilde{\mathcal{H}}_i)}$ and $\delta'=\widetilde{\delta} \inf_{\widetilde{\mathcal{H}}_i}\varphi_i(x)$. By Assumption~\ref{assumption-l}, there is $D>0$ such that $u^i_l(t,x)\ge 1-\delta'$ for $t\in \R$ and $x\in\widetilde{\mathcal{H}}_i$ such that $x\cdot e_i\ge -c^i_l t+D$.
Consider the domain
$$D^+:=\{(t,x)\in\R\times\widetilde{\mathcal{H}}_i; x\cdot e_i\ge -c^i_l t+D\}.$$
Note that $u^i_l(t,x)\ge 1-\delta'$ for $(t,x)\in D^+$.
Define $\phi(t,x):=1-\widetilde{\delta} e^{-r_i(x\cdot e_i+c^i_l t-D)} \varphi_i(x)$. Then, $\phi(t,x)\ge 1-\widetilde{\delta} \|\varphi_i\|_{L^{\infty}}\ge 1-\delta$ for $t\in \R$ and $x\in\widetilde{\mathcal{H}}_i$ such that $x\cdot e_i\ge -c^i_l t+D$. By \eqref{eq-F}, \eqref{delta-s4} and \eqref{eq-vphi}, one can easily check that $\phi(t,x)$ satisfies
\begin{eqnarray}\label{phi}
\left\{\begin{array}{lll}
&\phi_t-$div$(\widetilde{A}(x)\nabla \phi)+\widetilde{q}(x)\cdot \nabla \phi\le \widetilde{f}(x,\phi), \quad &(t,x)\in D^+,\\
&\nu \widetilde{A}(x)\nabla \phi\le 0,& (t,x)\in \partial D^+,
\end{array}
\right.
\end{eqnarray}
where $\partial D^+=\{t\in R,x\in\partial\widetilde{\mathcal{H}}_i; x\cdot e_i\ge -c^i_l t+D\}$.
Moreover, one has that
\be\label{phige}
\phi(t,x)\le 1-\widetilde{\delta}\varphi_i(x)\le 1-\delta'\le  u^i_l(t,x) \hbox{ for all $(t,x)\in D^+$ such that $x\cdot e_i=-c^i_l t+D$}.
\ee

Define
$$\varepsilon_*=\inf\{\varepsilon>0; u^i_l(t,x)+\varepsilon\ge \phi(t,x) \hbox{ in $D^+$}\}.$$
Since $\phi(t,x)\le 1$, then $\varepsilon_*<+\infty$. One only has to prove that $\varepsilon_*=0$.

Assume by contradiction that $\varepsilon_*>0$. Then, there exist sequences $0<\varepsilon_n<\varepsilon_*$ and $(t_n,x_n)_{n\in \mathbb{N}}$ in $D^+$ such that
\be\label{en}
\varepsilon_n\rightarrow \varepsilon_* \hbox{ as $n\rightarrow +\infty$ and } u^i_l(t_n,x_n)+\varepsilon_n < \phi(t_n,x_n) \hbox{ for all $n$}.
\ee
We claim that $x_n\cdot e_i+c^i_l t_n<+\infty$. Otherwise, $u^i_l(t_n,x_n)\rightarrow 1$ and $\phi(t_n,x_n)\rightarrow 1$ which contradicts the above inequality.
Now, define $v(t,x)=u^i_l(t,x)+\varepsilon^*-\phi(t,x)$. Then, $v(t,x)\ge 0$ in $D^+$. By \eqref{en}, one has that $v(t_n,x_n)\rightarrow 0$ as $n\rightarrow +\infty$. By \eqref{phige}, one knows that  $v(t,x)\ge \varepsilon^*>0$ for $(t,x)\in D^+$ such that $x\cdot e_i=-c^i_l t +D$. Since standard parabolic estimates imply that $\|(u^i_l)_t\|_{\infty}$, $\|\nabla u^i_l\|_{L^{\infty}}$ are bounded, there is $\rho>0$ such that $x_n\cdot e_i+c^i_l t_n\ge D+\rho$. Take $\tau>0$ and $y_n$ such that $y_n\cdot e_i+c^i_l (t_n-\tau)= D$. Then, $|x_n-y_n|<+\infty$ by $|x\cdot e_i+c^i_l t_n|<+\infty$ and
\be\label{vge}
v(t_n-\tau,y_n)\ge \varepsilon_*>0.
\ee
Since $\widetilde{f}(x,u)$ is decreasing in $u$ for $x\in \R^N$ and $u\in [1-\sigma,1]$, it follows from \eqref{eq-extension} and $u^i_l(t,x)\ge 1-\delta'\ge 1-\delta\ge 1-\sigma$ in $D^+$ that
\be\label{uil+e}
(u^i_l+\varepsilon_*)_t-\hbox{div}(\widetilde{A}(x)\nabla (u^i_l+\varepsilon_*))+\widetilde{q}(x)\cdot \nabla (u^i_l+\varepsilon_*)\ge \widetilde{f}(x,u^i_l+\varepsilon) \hbox{ in $D^+$}.
\ee
By \eqref{phi} and \eqref{uil+e}, one gets that  $v(t,x)$ satisfies
$$v_t-\hbox{div}(\widetilde{A}(x)\nabla v)+\widetilde{q}(x)\cdot \nabla v\ge b(t,x)v, \hbox{ in $D^+$},$$
where $b(t,x)$ is bounded in $D^+$.
Then, by linear parabolic estimates, one has that
$$v(t_n-\tau, y_n)\rightarrow 0, \hbox{ as $n\rightarrow +\infty$},$$
which contradicts \eqref{vge}. Therefore, $\varepsilon_*=0$ and $u^i_l(t,x)\ge \phi(t,x)$ in $D^+$. This completes the proof.
\end{proof}

\begin{remark}\label{remark-uir}
Similar as Lemma~\ref{uil-lemma}, one has the following property of $u^i_r(t,x)$, that is,
$$u^i_r(t,x)\ge 1-C e^{r_i(x\cdot e_i-c^i_r t)} \varphi_i(x), \hbox{ for $t\in \R$ and $x\in\overline{\widetilde{\mathcal{H}}_i}$ such that $x\cdot e_i\le c^i_r t-D$}.$$
\end{remark}

By Lemma~\ref{lemma2.1} and Lemma~\ref{uil-lemma}, one has the following corollary.

\begin{corollary}\label{cor4.7}
There exist $T_1<0$, $A>0$, $r_i>0$ and $\lambda_i>0$ such that
$$u_i(t,x)\ge 1- \delta e^{-r_i(x\cdot e_i+c^i_l t-A)}-\delta e^{-\lambda_i (x\cdot e_i-L)} \hbox{ for $t\le T_1$ and $x\in\overline{\mathcal{H}_i(-c^i_l t+A)}$}.$$
\end{corollary}

\begin{proof}
Let $\delta$ be defined by \eqref{delta-s4}. By Lemma~\ref{lemma2.1}, there exist $T_1<0$, $\lambda_i>0$, $\psi_i(x)$ and $\omega>0$ such that
\be\label{uigeuil}
u_i(t,x)\ge u^i_l(t-\omega,x)-\widetilde{\delta} e^{-\lambda_i (x\cdot e_i-L)} \hbox{ for $t\le T_1$ and $x\in\overline{\mathcal{H}_i(L)}$},
\ee
where $\widetilde{\delta}=\delta/\|\psi_i\|_{L^{\infty}(\overline{\mathcal{H}_i(L)})}$. Even if it means decreasing $T_1$, assume that $-c^i_l t+D\ge L$ for all $t\le T_1$ where $D$ is defined by Lemma~\ref{uil-lemma}. Then, $x\in\widetilde{\mathcal{H}}_i$ such that $x\cdot e_i\ge -c^i_l t+D$ means that $x\in \mathcal{H}_i(L)$. It follows from Lemma~\ref{uil-lemma} and \eqref{uigeuil} that there exist $C>0$, $r_i>0$, $\varphi_i(x)>0$ and $D>0$ such that
$$u_i(t,x)\ge 1-C e^{-r_i(x\cdot e_i+c^i_l (t-\omega))}\varphi_i(x)-\delta e^{-\lambda_i(x\cdot e_i-L)},$$
for $t\le T_1$ and $x\in \overline{\mathcal{H}_i}$ such that $x\cdot e_i\ge -c^i_l t+D$.
By taking a constant $A>0$ sufficiently large, one can have the conclusion.
\end{proof}

\begin{lemma}\label{lemma-wu-e}
There exist $T_2>0$, $0<\mu\le \delta$ and $D>0$ such that for any $t_0\ge T_2$,
$$\widetilde{u}_i(t,x)\ge 1-3\delta e^{-\mu (t-t_0)} \hbox{ for $t\ge t_0$ and $x\in\overline{\Omega\setminus\mathcal{H}_i(c^i_r t_0+\frac{c^i_r}{2} (t-t_0)-D)}$}.$$
\end{lemma}

\begin{proof}
Remember that $\widetilde{u}_i(t,x)$ propagates completely. By Lemma~\ref{lemma3.1}, there exist $T>0$ and $\tau_1\in \R$ such that
$$\widetilde{u}_i(t,x)\ge u_r^i(t-T+\tau_1,x)-\delta e^{-\delta (t-T)}-\widetilde{\delta} e^{- \lambda_i (x\cdot e_i-L)} \psi_i(x) \hbox{ for $t\ge T$ and $x\in \overline{\mathcal{H}_i(L)}$},$$
and
$$\widetilde{u}_i(t,x)\ge 1-\delta, \hbox{ for all $t\ge T$ and $x\in \overline{\Omega\setminus\mathcal{H}_i(L)}$}.$$
By Remark~\ref{remark-uir} and $\widetilde{\delta}\psi_i(x)\le \delta$, there exist $C>0$, $r_i>0$, $\varphi_i(x)>0$ and $D>0$ such that
\be\label{eq-le4.8}
\widetilde{u}_i(t,x)\ge 1- C e^{r_i(x\cdot e_i-c^i_r(t-T+\tau_1))}\varphi_i(x)-\delta e^{-\delta(t-T)}-\delta e^{-\lambda_i(x\cdot e_i-L)},
\ee
for $t\ge T$ and $x\in\overline{\mathcal{H}_i(L)}$ such that $x\cdot e_i\le c^i_r t-D$. By \eqref{eq-le4.8}, even if it means increasing $D$, one can assume that
$$\widetilde{u}_i(t,x)\ge 1-3\delta \hbox{ for $t\ge T$ and $x\in\overline{\mathcal{H}_i(L)}$ such that $x\cdot e_i\le c^i_r t-D$}.$$
Take $T_2\ge T>0$ large enough such that
$$C e^{r_i(-\frac{c^i_r}{2}T_2-D+c^i_r(T-\tau_1))}\varphi_i(x)\le \delta,\ e^{-\delta(T_2-T)}\le 1 \hbox{ and } e^{-\lambda_i(\frac{c^i_r}{2}T_2-D-L)}\le 1.$$
For any $t_0\ge T_2$, $t\ge t_0$ and $x\in \overline{\mathcal{H}_i(L)}$ such that $x\cdot e_i=\frac{c^i_r}{2}t-D$, one has that
$$\widetilde{u}(t,x)\ge 1- \delta e^{-\frac{r_1c^i_r}{2}(t-t_0)}-\delta e^{-\delta (t-t_0)}-\delta e^{-\frac{\lambda_ic^i_r}{2}(t-t_0)}.$$
Let $\mu:=\min(\delta,r_ic^i_r/2,\lambda_i c^i_r/2)$. Then,
$$\widetilde{u}_i(t,x)\ge 1-3\delta e^{-\mu (t-t_0)} \hbox{ for $t\ge t_0$ and $x\in \overline{\mathcal{H}_i(L)}$ such that $x\cdot e_i=c^i_r t_0+\frac{c^i_r}{2}(t-t_0)-D$}.$$
By $\mu\le \delta\le \gamma$, \eqref{eq-F} and \eqref{delta-s4}, one can easily check that $1-3\delta e^{-\mu (t-t_0)}$ is a subsolution of the problem satisfied by $\widetilde{u}(t,x)$ for $t\ge t_0$ and $x\in \overline{\Omega\setminus\mathcal{H}_i(c^i_r t_0+\frac{c^i_r}{2}(t-t_0)-D)}$. Therefore, it follows from the comparison principle that
$$\widetilde{u}_i(t,x)\ge 1-3\delta e^{-\mu(t-t_0)} \hbox{ for $t\ge t_0$ and $x\in \overline{\Omega\setminus\mathcal{H}_i(c^i_r t_0+\frac{c^i_r}{2}(t-t_0)-D)}$}.$$
This completes the proof.
\end{proof}
\vskip 0.3cm

Let $\mu$ be the constant such that Lemma~\ref{lemma-wu-e} holds for $\delta=\min(\gamma,\frac{\sigma}{3})$. From now on, we reset the constant $\delta$ such that
\be\label{delta-reset}
0<\delta\le \min_{i\in \{1,\cdots,m\}}(\gamma,\frac{\sigma}{4},\mu,\frac{\lambda_i c^i_r}{2},\frac{r_i c^i_r}{2}),
\ee
where $\lambda_i$ and $r_i$ are as in Corollary~\ref{cor4.7}.
By increasing $T_2$ and $D$, Lemma~\ref{lemma-wu-e} still holds for such $\delta$ and $\mu$. Then, all of above results hold for such $\delta$.

\begin{lemma}\label{lemma-s2}
Fix any $i\in \{1,\cdots,m\}$. For any $L_i\ge L$ and $R>0$ such that $L_i-R\ge L$, let $w_0(x)$ be an initial value satisfying
\begin{eqnarray}
w_0(x)=\left\{\begin{array}{lll}
1-\delta, &&\hbox{ for $x\in\overline{\mathcal{H}_i}$ such that } L_i-R\le x\cdot e_i\le L_i+R,\\
0, &&\hbox{ for $x\in\overline{\Omega\setminus\mathcal{H}_i(L_i-R)}$ and $x\in\overline{\mathcal{H}_i(L_i+R)}$}£¬
\end{array}
\right.
\end{eqnarray}
and $w(t,x)$  be the solution of \eqref{eq1.1} for $t\ge 0$ with $w(0,x)=w_0(x)$.
Then, there exist $L_0>0$, $R_0>0$ and $\omega>0$ such that for all $R\ge R_0$ and $L_i\ge L_0+2R$, there holds
\be\label{a-w}
w(t,x)\ge u_i(t+\tau_1-\omega,x)+\widetilde{u}_i(t+\tau_2-\omega,x)-1-\delta e^{-\delta t}, \hbox{ for $0\le t\le T$ and $x\in\overline{\Omega}$},
\ee
where
$$\tau_1=\frac{R-L_i+A_{\delta}}{c^i_l},\ \tau_2=\frac{L_i+R-A_{\delta}}{c^i_r} \hbox{ and } T=\frac{L_i-2R-L}{c^i_l}.$$
Furthermore, by taking $\delta$ sufficiently small and taking $L_i$, $R$, $L_i-2R$ sufficiently large, one has
\be\label{a-w-2}
\omega(t,x)\ge 1-3\delta_1 \hbox{ for all $t\ge 0$ and $x\in\overline{\mathcal{H}_i}$ such that $L_i\le x\cdot e_i\le L_i+c^i_r t$},
\ee
for any constant $0<\delta_1\le \sigma/3$.
\end{lemma}

\begin{proof}
{\it Step 1: some parameters.} Let $T_1<0$ such that \eqref{Ad} and Corollary~\ref{cor4.7} hold. Remember that $T_2>0$, $\mu>0$ and $D>0$ are constants such that \eqref{eq-wu-Ad} and Lemma~\ref{lemma-wu-e} hold. By Corollary~\ref{cor-m}, there is $k>0$ such that $(u_i)_{\xi}(\xi,x)\ge k$ for $\xi\le T_1$ and $x\in\overline{\mathcal{H}_i}$ such that $-c^i_l \xi-A_{\delta}\le x\cdot e_i\le -c^i_l\xi+A_{\delta}$, and $(\widetilde{u}_i)_{\xi}(\xi,x)\ge k$ for $\xi\ge T_2$ and $x\in\overline{\mathcal{H}_i}$ such that $c^i_r \xi-A_{\delta}\le x\cdot e_i\le c^i_r \xi+A_{\delta}$. Let $\omega>0$ such that
$$k\omega\ge \delta+4M.$$
Let
$$R_0= \max_{i\in\{1,\cdots,m\}}(A_{\delta}+A+c^i_l \omega, A_{\delta}+D+c^i_r \omega,L,c^i_l T_1-L-A_{\delta})$$
 and
$$
L_0=\max_{i\in \{1,\cdots,m\}} (A_{\delta}-T_1 c^i_l,c^i_rT_2+A_{\delta},L).$$ For any $R\ge R_0$ and $L_i\ge L_0+2R$, let $\tau_1$ and $\tau_2$ be defined as
$$\tau_1:=\frac{R-L_i-A_{\delta}}{c^i_l}-\omega \hbox{ and } \tau_2:=\frac{L_i+R-A_{\delta}}{c^i_r}-\omega.$$
Notice that $\tau_1+\omega\le T_1$ and $\tau_2\ge T_2$. Let $T$ be
$$T=\frac{L_i-2R-L}{c^i_l}.$$

{\it Step 2: proof of \eqref{a-w}.} For any $0\le t\le T$ and $x\in \overline{\Omega}$, we set
$$\underline{u}(t,x)=\max\Big(u_i(\underline{\zeta}_1(t),x)+\widetilde{u}_i(\underline{\zeta}_2(t),x)-1-\delta e^{-\delta t},0\Big),$$
where
$$\underline{\zeta}_1(t)=t+\tau_1+\omega e^{-\delta t} \hbox{ and } \underline{\zeta}_2(t)=t+\tau_2+\omega e^{-\delta t}.$$
We prove that $\underline{u}(t,x)$ is a subsolution of the problem satisfied by $w(t,x)$ for $0\le t\le T$ and $x\in\overline{\Omega}$.

At the time $t=0$, one has $\underline{\zeta}_1(0)=\tau_1+\omega\le T_1$, $\underline{\zeta}_2(0)=\tau_2+\omega\ge T_2$ and
$$\underline{u}(0,x)\le \max\Big(1+1-1-\delta,0\Big)=1-\delta\le w_0(x), \hbox{ for $x\in\overline{\mathcal{H}_i}$ such that $L_i-R\le x\cdot e_i\le L_i+R$}.$$
By the definition of $\tau_1$, one has that $L_i-R=-c^i_l(\tau_1+\omega)-A_{\delta}$ and hence $x\in\overline{\Omega\setminus \mathcal{H}_i(L_i-R)}$ implies $x\in\overline{\Omega\setminus\mathcal{H}_i(-c^i_l (\tau_1+\omega)-A_{\delta})}$. It follows from $\tau_1+\omega\le T_1$ and \eqref{Ad} that $u_i(\tau_1+\omega,x)\le \delta$ for $x\in\overline{\Omega\setminus\mathcal{H}_i(L_i-R)}$. Thus,
$$\underline{u}(0,x)\le \max\Big(\delta+1-1-\delta,0\Big)=0\le w_0(x), \hbox{ for $x\in\overline{\Omega\setminus\mathcal{H}_i(L_i-R)}$}.$$
By the definition of $\tau_2$, one has that $L_i+R=c^i_r(\tau_2+\omega)+A_{\delta}$ and hence $x\cdot e_i\ge L_i+R$ implies $x\cdot e_i\ge c^i_r(\tau_2+\omega)+A_{\delta}$. It follows from $\tau_2\ge T_2$ and \eqref{eq-wu-Ad} that $\widetilde{u}_i(\tau_2+\omega,x)\le \delta$ for $x\in\overline{\mathcal{H}_i(L_i+R)}$. Thus,
$$\underline{u}(0,x)\le \max\Big(\delta+1-1-\delta,0\Big)=0\le w_0(x), \hbox{ for $x\in\overline{\mathcal{H}_i(L_i+R)}$}.$$
Therefore, $\underline{u}(0,x)\le w_0(x)$ for all $x\in \overline{\Omega}$. It is obvious that $\nu A(x)\nabla\underline{u}(t,x)=0$ for all $0\le t\le T$ and $x\in \partial\Omega$.

Let us now check that
$$N(t,x)=\underline{u}_t-\hbox{div}(A(x)\nabla\underline{u})+q(x)\cdot \nabla\underline{u}-f(x,\underline{u})\le 0,$$
for $0\le t\le T$ and $x\in\overline{\Omega}$ such that $\underline{u}(t,x)>0$. After some calculation, one has
\begin{align*}
N(t,x)=&-\omega \delta e^{-\delta t} \Big((u_i)_t(\underline{\zeta}_1(t),x)+(\widetilde{u}_i)_t(\underline{\zeta}_2(t),x)\Big) +\delta^2 e^{-\delta t}+f(x,u_i(\underline{\zeta}_1(t),x))\\
&+f(x,\widetilde{u}_i(\underline{\zeta}_2(t),x))-f(x,\underline{u}(t,x)).
\end{align*}
We first deal with the part $x\in \overline{\Omega\setminus \mathcal{H}_i(L_i+\frac{c^i_r}{2} t)}$ for $0\le t\le T$. Notice that $\underline{\zeta}_2(t)\ge \tau_2+\omega e^{-\delta t}\ge T_2$ for all $0\le t\le T$ and
$$L_i+\frac{c^i_r}{2} t\le c^i_r (\tau_2+\omega e^{-\delta t}) +\frac{c^i_r}{2}(\underline{\zeta}_2(t)-(\tau_2+\omega e^{-\delta t}))-D \hbox{ for all $t\ge 0$}.$$
By Lemma~\ref{lemma-wu-e}, it implies that
$$\widetilde{u}_i(\underline{\zeta}_2(t),x)\ge 1-3\delta \hbox{ and } 1-\widetilde{u}_i(\underline{\zeta}_2(t),x)\le 3\delta e^{-\mu t} \hbox{ for all $0\le t\le T$ and $x\in \overline{\Omega\setminus\mathcal{H}_i(L_i+\frac{c^i_r}{2}t)}$}.$$
Also notice that $\underline{\zeta}_1(t)\le T+\tau_1+\omega\le (-R-L-A_{\delta})/c^i_l\le T_1$ for all $0\le t\le T$. Then, for $0\le t\le T$ and $x\in\overline{\Omega\setminus\mathcal{H}_i(-c^i_l \underline{\zeta}_1(t)-A_{\delta})}$, it follows from \eqref{Ad} that $u_i(\underline{\zeta}_1(t),x)\le \delta$ and hence $\underline{u}(t,x)\le \delta$. Thus, by \eqref{delta-reset} and \eqref{eq-F},
$$f(x,u_i(\underline{\zeta}_1(t),x))-f(x,\underline{u}(t,x))\le -\gamma (1-\widetilde{u}_i(\underline{\zeta}_2(t),x)+\delta e^{-\delta t})\le -4\gamma \delta e^{-\delta t},$$
and $f(x,\widetilde{u}_i)\le 0$.
It then follows from $(u_i)_t>0$, $(\widetilde{u}_i)_t>0$ and \eqref{delta-reset} that
\begin{align*}
N(t,x)\le &-\omega \delta e^{-\delta t} \Big((u_i)_t(\underline{\zeta}_1(t),x)+(\widetilde{u}_i)_t(\underline{\zeta}_2(t),x)\Big) +\delta^2 e^{-\delta t}-4\gamma \delta e^{-\delta t}\le 0.
\end{align*}
For $0\le t\le T$ and $x\in\overline{\mathcal{H}_i}$ such that $-c^i_l \underline{\zeta}_1(t)-A_{\delta}\le x\cdot e_i\le -c^i_l \underline{\zeta}_1(t)+A_{\delta}$, one has that $(u_i)_t(\underline{\zeta}_1(t),x)\ge k$. It is obvious that 
$$f(x,u_i(\underline{\zeta}_1(t),x))-f(x,\underline{u}(t,x))\le -\gamma (1-\widetilde{u}_i(\underline{\zeta}_2(t),x)+\delta e^{-\delta t})\le 4M\delta e^{-\delta t}.$$
Thus, 
\begin{align*}
N(t,x)\le &-\omega k \delta e^{-\delta t} +\delta^2 e^{-\delta t} +4M\delta e^{-\delta t}\le 0.
\end{align*}
For $0\le t\le T$ and $x\in\overline{\mathcal{H}_i}$ such that $-c^i_l\underline{\zeta}_1(t)+A_{\delta}\le x\cdot e_i\le L_i+\frac{c^i_r}{2}t$, it follows from \eqref{Ad} that $u_i(\underline{\zeta}_1(t),x)\ge 1-\delta$ and hence $\underline{u}(t,x)\ge 1-3\delta$. Thus, by \eqref{delta-reset} and \eqref{eq-F},
$$f(x,u_i(\underline{\zeta}_1(t),x))-f(x,\underline{u}(t,x))\le -\gamma (1-\widetilde{u}_i(\underline{\zeta}_2(t),x)+\delta e^{-\delta t})\le -4\gamma\delta e^{-\delta t}.$$
It then follows from $(u_i)_t>0$, $(\widetilde{u}_i)_t>0$ and \eqref{delta-reset} that
\begin{align*}
N(t,x)\le &-\omega \delta e^{-\delta t} \Big((u_i)_t(\underline{\zeta}_1(t),x)+(\widetilde{u}_i)_t(\underline{\zeta}_2(t),x)\Big) +\delta^2 e^{-\delta t}-4\gamma \delta e^{-\delta t}\le 0.
\end{align*}

We then deal with the part $x\in \overline{\mathcal{H}_i(L_i+\frac{c^i_r}{2} t)}$ for $0\le t\le T$. Notice that $\underline{\zeta}_1(t)\le T_1$ for all $0\le t\le T$ and
$$L_i+\frac{c^i_r}{2} t\ge -c^i_l \underline{\zeta}_1(t)+A \hbox{ for all $t\ge 0$}.$$
By Corollary~\ref{cor4.7} and \eqref{delta-reset}, it implies that
\begin{align*}
u_i(\underline{\zeta}_1(t),x)&\ge 1-\delta e^{-r_i(x\cdot e_i+c^i_l \underline{\zeta}_1(t)-A)}-\delta e^{-\lambda_i(x\cdot e_i-L)}\\
&\ge 1-2\delta e^{-\delta t}\hskip 2cm \hbox{ for all $0\le t\le T$ and $x\in \overline{\mathcal{H}_i(L_i+\frac{c^i_r}{2}t)}$}.
\end{align*}
Then, for $0\le t\le T$ and $x\in\overline{\mathcal{H}_i(L_i+\frac{c^i_r}{2}t)}$ such that $x\cdot e_i\le c^i_r \underline{\zeta}_2(t)-A_{\delta}$, it follows from \eqref{eq-wu-Ad} that $\widetilde{u}_i(\underline{\zeta}_2(t),x)\ge 1-\delta$ and hence $\underline{u}(t,x)\le 1-3\delta$. Thus, by \eqref{delta-reset} and \eqref{eq-F},
$$f(x,\widetilde{u}_i(\underline{\zeta}_2(t),x))-f(x,\underline{u}(t,x))\le -\gamma (1-u_i(\underline{\zeta}_1(t),x)+\delta e^{-\delta t})\le -3\gamma \delta e^{-\delta t},$$
and $f(x,u_i(\underline{\zeta}_1(t),x))\le 0$.
It then follows from $(u_i)_t>0$, $(\widetilde{u}_i)_t>0$ and \eqref{delta-reset} that
\begin{align*}
N(t,x)\le &-\omega \delta e^{-\delta t} \Big((u_i)_t(\underline{\zeta}_1(t),x)+(\widetilde{u}_i)_t(\underline{\zeta}_2(t),x)\Big) +\delta^2 e^{-\delta t}-3\gamma \delta e^{-\delta t}\le 0.
\end{align*}
For $0\le t\le T$ and $x\in\overline{\mathcal{H}_i}$ such that $c^i_r \underline{\zeta}_2(t)-A_{\delta}\le x\cdot e_i\le c^i_r \underline{\zeta}_2(t)+A_{\delta}$, one has that $(\widetilde{u}_i)_t(\underline{\zeta}_2(t),x)\ge k$. It is obvious that 
$$f(x,\widetilde{u}_i(\underline{\zeta}_2(t),x))-f(x,\underline{u}(t,x))\le M (1-u_i(\underline{\zeta}_1(t),x)+\delta e^{-\delta t})\le 3M\delta e^{-\delta t}.$$
Thus,
\begin{align*}
N(t,x)\le &-\omega k \delta e^{-\delta t} +\delta^2 e^{-\delta t} +3M\delta e^{-\delta t}\le 0.
\end{align*}
For $0\le t\le T$ and $x\in\overline{\mathcal{H}_i(c^i_r \underline{\zeta}_2(t)+A_{\delta})}$, it follows from \eqref{eq-wu-Ad} that $\widetilde{u}_i(\underline{\zeta}_2(t),x)\le \delta$ and hence $\underline{u}(t,x)\le \delta$. Thus, by \eqref{delta-reset} and \eqref{eq-F},
$$f(x,\widetilde{u}_i(\underline{\zeta}_2(t),x))-f(x,\underline{u}(t,x))\le -\gamma (1-u_i(\underline{\zeta}_1(t),x)+\delta e^{-\delta t})\le -3\gamma\delta e^{-\delta t}.$$
It then follows that
\begin{align*}
N(t,x)\le &-\omega \delta e^{-\delta t} \Big((u_i)_t(\underline{\zeta}_1(t),x)+(\widetilde{u}_i)_t(\underline{\zeta}_2(t),x)\Big) +\delta^2 e^{-\delta t}-3\gamma \delta e^{-\delta t}\le 0.
\end{align*}

Consequently, by the comparison principle and $(u_i)_t>0$, $(\widetilde{u}_i)_t>0$, one gets that
$$w(t,x)\ge \underline{u}(t,x)\ge u_i(t+\tau_1,x)+\widetilde{u}_i(t+\tau_2,x)-1-\delta e^{-\delta t},$$
for $0\le t\le T$ and $x\in\overline{\Omega}$.
This completes the proof of \eqref{a-w}.

{\it Step 3: proof of \eqref{a-w-2}.} Let $A_{\delta/3}$ be the constant such that \eqref{Ad} and \eqref{eq-wu-Ad} hold for $\delta$ replaced by $\delta/3$. By taking $R$ sufficiently large, one can make that for $0\le t\le T$ and $x\in\overline{\mathcal{H}_i}$ such that $x\cdot e_i=L_i$,
$$x\cdot e_i+c^i_l(t+\tau_1-\omega)\ge L_i+R-L_i+A_{\delta}-c^i_l\omega=R+A_{\delta}-c^i_l\omega\ge A_{\delta},$$
and
$$c^i_r(t+\tau_2-\omega)-x\cdot e_i\ge L_i+R-A_{\delta}-c^i_r\omega-L_i=R-A_{\delta}-c^i_r\omega\ge A_{\delta}.$$
Then, by \eqref{Ad}, \eqref{eq-wu-Ad} and \eqref{a-w}, one has that
$$w(t,x)\ge 1-3\delta \hbox{ for all $0\le t\le T$ and $x\in\overline{\mathcal{H}_i}$ such that $x\cdot e_i=L_i$}.$$

On the other hand, by taking $L_i-2R$ sufficiently large, one can make that $\delta e^{-\delta T}\le \delta/3$ and for $x\in\overline{\mathcal{H}_i}$ such that $L_i-R\le x\cdot e_i\le L_i+R$,
$$x\cdot e_i+c^i_l(T+\tau_1-\omega)\ge L_i-2R-L+A_{\delta}-c^i_l\omega\ge A_{\delta/3},$$
and
$$c^i_r(T+\tau_2-\omega)-x\cdot e_i\ge \frac{c^i_r}{c^i_l}(L_i-2R-L)-A_{\delta}-c^i_r\omega\ge A_{\delta/3}.$$
Then, one has that
$$\omega(T,x)\ge 1-\delta, \hbox{ for $x\in\overline{\mathcal{H}_i}$ such that $L_i-R\le x\cdot e_i\le L_i+R$}.$$

One can do the same arguments as in Step~2 to get that
$$w(T+t,x)\ge \underline{u}(t,x)\ge u_i(t+\tau_1,x)+\widetilde{u}_i(t+\tau_2,x)-1-\delta e^{-\delta t},$$
for $0\le t\le T$ and $x\in\overline{\Omega}$. Then, $\omega(T+t,x)\ge 1-3\delta$ for all $0\le t\le T$ and $x\in\overline{\mathcal{H}_i}$ such that $x\cdot e_i=L_i$
By iteration, one has that
\be\label{wge1-3d}
w(t,x)\ge 1-3\delta \hbox{ for all $t\ge 0$ and $x\in\overline{\mathcal{H}_i}$ such that $x\cdot e_i=L_i$}.
\ee

Now,  let $\lambda_i>0$ and $\psi_i(x)>0$ be the constant and function satisfying Lemma~\ref{lemma-psi} for the extension $\widetilde{\mathcal{H}_i}$ of $\mathcal{H}_i$ and $\beta=\gamma$ where $\gamma$ is defined by \eqref{eq-F}. We take any constant $\delta_1$ such that
$$0<\delta_1\le \min_{i\in\{1,\cdots,m\}} (\lambda_i c^i_r,\gamma,\sigma/3),$$
(remember that $c^i_r>0$ for all $i\in\{1,\cdots,m\}$ in this section). Define $\widetilde{\delta}_1:=\delta_1/\|\psi_i\|_{L^{\infty}(\overline{\mathcal{H}_i(L)})}$ and $\delta'_1:=\widetilde{\delta}\inf_{\overline{\mathcal{H}_i(L)}}\psi_i(x)$. Since $\delta$ satisfying \eqref{delta-reset} can be arbitrarily taken, we take $\delta$ sufficiently small such that $3\delta<\delta'_1\le \delta_1$. Then, by \eqref{wge1-3d}, one has
$$w(t,x)\ge 1-\delta'_1 \hbox{ for all $t\ge 0$ and $x\in\overline{\mathcal{H}_i}$ such that $x\cdot e_i=L$}.$$
Moreover,
$$w(0,x)\ge 1-\delta\ge 1-\delta_1 \hbox{ for $x\in\overline{\mathcal{H}_i}$ such that $L_i-R\le x\cdot e_i\le L_i+R$}.$$
Then, by the proof of Lemma~\ref{lemma3.1} and since $R$ is sufficiently large, one can easily check that the function
$$\underline{u}_1(t,x)=\max\Big(u^i_r(t+\omega e^{-\delta_1 t}-\omega+\tau_1,x)-\delta_1 e^{-\delta_1 t}-\widetilde{\delta}_1 e^{-\lambda_i(x\cdot e_i-L_i)} \psi_i(x),0\Big),$$
for $t\ge 0$ and $x\in\overline{\mathcal{H}_i(L)}$, where $\omega$ is defined by \eqref{eq-omega1} for $\delta_1$, $\tau_1=\frac{L_i+R-M_{\delta_1}}{c^i_r}$ and $M_{\delta_1}$ is defined by \eqref{Md-s3}, is a subsolution of the problem satisfied by $w(t,x)$. Thus, by the comparison principle, it follows that
$$w(t,x)\ge \underline{u}_1(t,x), \hbox{ for $t\ge 0$ and $x\in\overline{\mathcal{H}_i(L_i)}$}.$$
Then, it is elementary to check that $w(t,x)$ satisfies \eqref{a-w-2}.
\end{proof}

\subsection{Proof of Theorem~\ref{th3}}

Now, we are ready to prove the uniqueness of the transition front connecting $0$ and $1$. We consider any transition front $u$ connecting $0$ and $1$ for \eqref{eq1.1} associated with some sets $(\Omega^{\pm}_{t})_{t\in \R}$ and $(\Gamma_t)_{t\in \R}$. We first derive that the interfaces $\Gamma_t$ are located far away from the origin at very negative time.

\begin{lemma}\label{lemma4.3}
For every $\rho\ge 0$, there exists a real number $T$  such that
$$\Omega\cap B(0,L+\rho)\subset \Omega_t^-\ \ \text{for all $t\le T_1$}.$$
\end{lemma}

\begin{proof}
Thanks to Lemma~\ref{lemma-s2}, it can be proved similarly by the arguments of the proof of \cite[Lemma~4.5]{GHS}. Actually, in the arguments of the proof of \cite[Lemma~4.5]{GHS}, the key is to apply Lemma~4.1 of \cite{GHS} whose conclusions are similar as those in our Lemma~\ref{lemma-s2}.
\end{proof}
\vskip 0.3cm

By Definition~\ref{TF}, one can assume without loss of generality, even if it means redefining $\Omega^\pm_t$ and $\Gamma_t$, that, for every $t\in\R$ and $i\in\{1,\cdots,m\}$, there is an non-negative integer $n_{i,t}\in\{0,\cdots,n\}$ and some real numbers $L<\xi_{i,t,1}<\cdots<\xi_{i,t,n_{i,t}}$ (if $n_{i,t}\ge1$) such that
\be\label{interfaces}
\Gamma_t\cap\mathcal{H}_i=\bigcup_{k=1}^{n_{i,t}}\big\{x\in\mathcal{H}_i: x\cdot e_i=\xi_{i,t,k}\big\},
\ee
where $n$ is as in~\eqref{eq1.6} and with the convention $\Gamma_t\cap\mathcal{H}_i=\emptyset$ if $n_{i,t}=0$. By \eqref{eq1.3}, every $\Omega_t^+$ contains a half-infinite branch. Then, by continuity of $u(t,x)$, there is a set $I\subset\{1,\cdots,m\}$ such that $n_{i,t}\neq 0$ for all $t\le T$ and $i\in I$. Notice that Lemma~\ref{lemma4.3} also implies that $\xi_{i,t,1}\rightarrow +\infty$ as $t\rightarrow -\infty$ for every $i\in I$.

\begin{lemma}\label{lemma-xige}
For every $i\in I$, there holds
$$\liminf_{t\rightarrow -\infty} \Big(\xi_{i,t,1}+c^i_l t \Big)>-\infty.$$
\end{lemma}

\begin{proof}
Assume by contradiction that there exist $i\in I$ and a sequence $\{t_k\}_{k\in \mathbb{N}}$ such that $t_k\rightarrow -\infty$ as $k\rightarrow +\infty$ and $\xi_{i,t_k,1}+c^i_l t_k\rightarrow -\infty$. Since $\xi_{i,t_k,1}\rightarrow +\infty$ as $k\rightarrow +\infty$, it follows from Definition~\ref{TF}, \eqref{eq1.5} and \eqref{disO} that for any $R>0$ and $\delta$ defined by \eqref{delta-reset}, there is $D_{\delta}>0$ such that
$$\{x\in\mathcal{H}_i; \xi_{i,t_k,1}+D_{\delta}-R\le x\cdot e_i\le \xi_{i,t_k,1}+D_{\delta}+R\}\subset \Omega^+_{t_k},\ \xi_{i,t_k,1}+D_{\delta}-R\ge L \hbox{ for large $k$},$$
and
$$u(t_k,x)\ge 1-\delta, \hbox{ for all $x\in \overline{\mathcal{H}_i}$ such that $\xi_{i,t_k,1}+D_{\delta}-R\le x\cdot e_i\le \xi_{i,t_k,1}+D_{\delta}+R$ and large $k$}.$$
Then, by Lemma~\ref{lemma-s2}, one has that
\be\label{uge}
u(t,x)\ge u_i(t-t_k+\tau_1-\omega,x)+\widetilde{u}_i(t-t_k+\tau_2+\omega,x)-1-\delta e^{-\delta (t-t_k)}, \hbox{ for $0\le t-t_k\le T$ and $x\in\overline{\Omega}$},
\ee
where
$$\tau_1=\frac{R-\xi_{i,t_k,1}+D_{\delta}+A_{\delta}}{c^i_l},\ \tau_2=\frac{\xi_{i,t_k,1}+D_{\delta}+R-A_{\delta}}{c^i_r} \hbox{ and } T=\frac{\xi_{i,t_k,1}+D_{\delta}-2R-L}{c^i_l}.$$
Notice that $t_k+T\rightarrow -\infty$ and $\tau_2\rightarrow +\infty$ as $k\rightarrow +\infty$ since $\xi_{i,t_k,1}\rightarrow +\infty$, $\xi_{i,t_k,1}+c^i_l t_k\rightarrow -\infty$ as $k\rightarrow +\infty$. Moreover, one has that
$$u(t_k+T,x)\ge u_i(\frac{2D_{\delta}+A_{\delta}-R-L}{c^i_l}-\omega,x)+\widetilde{u}_i(T+\tau_2-\omega,x)-1-\delta, \hbox{ for $x\in\overline{\Omega}$}.$$
Since $\widetilde{u}_i(t,x)\rightarrow 1$ as $t\rightarrow +\infty$ locally uniformly for $x\in\overline{\Omega}$ and $u_i(t,x)\rightarrow 1$ as $x\cdot e_i\rightarrow +\infty$ for any $t\in \R$ and $x\in\overline{\mathcal{H}_i}$, there are $L_1$ and $L_2$ such that
$$u(t_k+T,x)\ge 1-3\delta, \hbox{ for $x\in\overline{\mathcal{H}_i}$ such that $L_1\le x\cdot e_i\le L_2$ and large $k$}.$$
Since $\delta$ can be taken arbitrarily small, it implies that 
$$\{x\in \mathcal{H}_i; L_1\le x\cdot e_i\le L_2\}\subset \Omega^+_{t_k+T}.$$
Together with Lemma~\ref{lemma4.3}, one has that $\xi_{i,t_k+T,1}<+\infty$.
This contradicts $\xi_{i,t,1}\rightarrow +\infty$ as $t\rightarrow -\infty$.
\end{proof}

\begin{lemma}\label{lemma-xile}
For every $i\in I$, there holds
$$\limsup_{t\rightarrow -\infty} \Big(\xi_{i,t,1}+c^i_l t\Big)<+\infty.$$
\end{lemma}

\begin{proof}
Take any sequence $\{t_k\}_{k\in \mathbb{N}}$ such that $t_k\rightarrow -\infty$ as $k\rightarrow +\infty$.
By Definition~\ref{TF}, there is $D_{\delta}>0$, \eqref{eq1.5}, \eqref{disO} and Lemma~\ref{lemma4.3} such that
$$u(t_k,x)\le \delta \hbox{ for $x\in \overline{\Omega\setminus\cup_{i\in I}\mathcal{H}_i(\xi_{i,t_k,1}-D_{\delta})}$ and large $k$},$$
where $\delta$ is defined by \eqref{delta-reset}.
By Lemma~\ref{lemma-s1}, one has
\be\label{ule}
u(t,x)\le u_{I,\tau_i}(t-t_k+\tau+\omega,x)+\delta e^{-\delta(t-t_k)} \hbox{ for $t\ge t_k$ and $x\in\overline{\Omega}$},
\ee
where
$$\tau=\max_{i\in I}\Big(\frac{A_{\delta}-\xi_{i,t_k,1}+D_{\delta}}{c^i_l}\Big) \hbox{ and } \tau+\tau_i=\frac{A_{\delta}-\xi_{i,t_k,1}+D_{\delta}}{c^i_l}.$$
Notice that $\tau\rightarrow -\infty$ as $k\rightarrow +\infty$ since $\xi_{i,t,1}\rightarrow +\infty$ as $t\rightarrow -\infty$ for every $i\in I$ and $\tau-t_k<+\infty$ by Lemma~\ref{lemma-xige}.

Assume by contradiction that there exist $i_0\in I$ and a sequence $\{t_{k}\}_{k\in \mathbb{N}}$ such that $t_k\rightarrow -\infty$ as $k\rightarrow +\infty$ and $\xi_{i_0,t_k,1}+c^{i_0}_l t_k\rightarrow +\infty$.
Then, one has that
\be\label{xile}
-c^{i_0}_l(t-t_k+\tau+\omega+\tau_{i_0})\rightarrow +\infty \hbox{ for any fixed $t$ as $k\rightarrow +\infty$}.
\ee
Since $\tau-t_k<+\infty$, one can pick any $t\le T$ such that $t-t_k+\tau+\omega\le T_1$ where $T_1$ is defined in \eqref{Ad}. By passing $k\rightarrow +\infty$, it follows from \eqref{Ad}, \eqref{ule} and \eqref{xile} that
$$u(t,x)\le \delta, \hbox{ for $x\in\overline{\mathcal{H}_{i_o}(L)}$}.$$
which contradicts $i_0\in I$ and $n_{i_0,t}\neq 0$ for $t\le T$.
\end{proof}

\begin{lemma}\label{trap}
For every $i\in I$, there are constants $\sigma_i$, $\tau_i$ and $\eta$ such that
$$u_i(t+\sigma_i,x)\le u(t,x)\le u_{I,\tau_i}(t+\eta,x) \hbox{ for all $t\in \R$ and $x\in \overline{\Omega}$}.$$
\end{lemma}

\begin{proof}
By Lemma~\ref{lemma-xige} and Lemma~\ref{lemma-xile}, one has that $|\xi_{i,t,1}+c^i_l t|\le +\infty$ for all $t\le T$ and $i\in I$. Take any sequence $\{t_k\}_{k\in\mathbb{N}}$ such that $t_k\rightarrow -\infty$ as $k\rightarrow +\infty$. Consider \eqref{uge} and notice that $|t_k+T|<+\infty$ and $|\tau_1-t_k|<+\infty$. By passing to the limit $k\rightarrow +\infty$, there is $\sigma_i\in \R$ such that
$$u(t,x)\ge u_i(t+\sigma_i,x), \hbox{ for all $t\le t_0$ with some $t_0\in \R$ and $x\in\overline{\Omega}$}.$$
Since $u_i(t+\sigma_i,x)$ is a solution of \eqref{eq1.1}, then it follows from the comparison principle that $u(t,x)\ge u_i(t+\sigma_i,x)$ for all $t\in \R$ and $x\in \overline{\Omega}$.

Consider \eqref{ule} and notice that $|\tau-t_k|<+\infty$ and $|\tau_i|<+\infty$ for all $i\in I$. Then, by passing the limit $k\rightarrow +\infty $ in \eqref{ule}, there is $\eta\in \R$ such that
$$u(t,x)\le u_{I,\tau_i}(t+\eta,x) \hbox{ for all $t\in \R$ and $x\in\overline{\Omega}$}.$$
This completes the proof.
\end{proof}
\vskip 0.3cm

\begin{proof}[Proof of Theorem~\ref{th3}]
By \eqref{front}, \eqref{eq-uI} and Lemma~\ref{trap}, one can get that $u(t,x)$ is trapped by shifts of $u^i_l(t,x)$ with some small perturbations for every $i\in I$ as $t\rightarrow -\infty$. Consider any sequence $\{t_n\}_{n\in\mathbb{N}}$ such that $t_n\rightarrow -\infty$ as $n\rightarrow +\infty$. By applying similar arguments as in the proof of Theorem~\ref{th2}, one can get that for any $\varepsilon>0$, there is $L_{\varepsilon}$ such that
$$|u(t_n,x)-u^i_l(t_n+m_i,x)|\le \varepsilon  \hbox{ for some $m_i$, large $n$ and $x\in\overline{\mathcal{H}_i(L_{\varepsilon})}$, every $i\in I$}.$$
Since $u_{I,\tau_i}(t,x)\rightarrow 0$ as $t\rightarrow -\infty$ locally uniformly for $x\in\overline{\Omega}$, then $u(t,x)\rightarrow 0$ as $t\rightarrow -\infty$ locally uniformly for $x\in\overline{\Omega}$ by Lemma~\ref{trap} and hence one has that
$$|u(t_n,x)-u^i_l(t_n+m_i,x)|\le \varepsilon \hbox{ for some $m_i$, large $n$ and all $x\in\overline{\mathcal{H}_i(L)}$, every $i\in I$}.$$
Let $\hat{\varepsilon}=\varepsilon/\delta$ where $\delta$ is defined as in Lemma~\ref{lemma-s1}. By the proof of Lemma~\ref{lemma-s1}, one can easily check that the functions
$$u^+(t,x)=\min\Big(u_{I,m_i}(t-t_n-\hat{\varepsilon}\omega e^{-\delta(t-t_n)}+\omega,x)+\hat{\varepsilon}\delta e^{-\delta (t-t_n)},1\Big),$$
and
$$u^-(t,x)=\max\Big(u_{I,m_i}(t-t_n+\hat{\varepsilon}\omega e^{-\delta(t-t_n)}-\omega,x)-\hat{\varepsilon}\delta e^{-\delta (t-t_n)},0\Big),$$
are sup- and subsolutions of the problem satisfied by $u(t,x)$ for $t\ge t_n$ and $x\in\overline{\Omega}$. It follows from the comparison principle that
$$u^-(t,x)\le u(t,x)\le u^+(t,x), \hbox{ for $t\ge t_n$ and $x\in\overline{\Omega}$}.$$
Thus, one has
$$|u(t,x)-u_{I,m_i}(t,x)|\le \omega\hat{\varepsilon}\|(u_{I,m_i})_t\|_{L^{\infty}}+2\hat{\varepsilon}\delta \hbox{ for all $t\ge t_n$}.$$
Since $\varepsilon>0$ can be arbitrarily small by taking $t_n$ negative enough, then one gets that $u(t,x)\equiv u_{I,m_i}(t,x)$. Since $u_{I,m_i}(t,x)$ is a transition front emanating from almost-planar fronts $u^i_l(t+m_i,x)$ in branches $\mathcal{H}_i$ for $i\in I$, it completes the proof of Theorem~\ref{th3}.
\end{proof}

%%%%%%%%%%%%%%%%%%%%%%%%%%%%%%%%%%%%%%%%%%%%%%%
\section{Some applications}

In this section, we give two simple examples to which our results can be applied.

{\it Example 1:} Consider the following equation
\begin{eqnarray}\label{ex1}
\left\{\begin{array}{lll}
&u_t-\Delta u=f(u), \quad &t\in\R,\ x\in\Omega,\\
&\partial_{\nu}u=0,& t\in\R,\ x\in\partial\Omega,
\end{array}
\right.
\end{eqnarray}
where $\nu$ denotes the outward unit normal to $\partial\Omega$ and $f(u)$ is a bistable nonlinearity, that is, satisfying \eqref{F1}. In this case, we make the branches of $\Omega$ are straight, that is,
$$\mathcal{H}_i:=\{x\in\R^N; x-(x\cdot e) e\in \omega_i, x\cdot e_i>0\}+x_i,$$
for every $i\in \{1,\cdots,m\}$, where $\omega_i\subset \R^{N-1}$ is a fixed set and $x_i\in \R^N$ is a shift. Then, one can extend every branch $\mathcal{H}_i$ such that the extension $\widetilde{\mathcal{H}}_i$ is still a straight cylinder, that is,
$$\widetilde{\mathcal{H}_i}:=\{x\in\R^N; x-(x\cdot e) e\in \omega_i\}+x_i.$$
By \cite{FM}, one knows that there are planar fronts $\phi_f(x\cdot e_i-c_f t)$ facing to direction $e_i$ and $\phi_f(-x\cdot e_i-c_f t)$ facing to direction $-e_i$ (where $(\phi_f,c_f)$ satisfies \eqref{phi_f}) for \eqref{ex1} with $\Omega$ replaced by $\widetilde{\mathcal{H}}_i$. Therefore, Assumptions~\ref{assumption-r} and~\ref{assumption-l} hold in this case. Since $\widetilde{\mathcal{H}}_i$ is invariant by shifts along with the direction $e_i$, it is obvious that Assumption~\ref{assumption-limiting} also holds. Thus, by Theorem~\ref{Th1}, there exist entire solutions of \eqref{ex1} emanating from planar fronts. If additionally every entire solution emanating from a planar front in every branch $\mathcal{H}_i$ propagates completely, that is, satisfying \eqref{front} (see \cite[Corollaries~1.11, 1.12]{GHS} for some sufficient geometrical conditions), then it follows from Theorem~\ref{th3} that the entire solution emanating from planar fronts is the only type of transition fronts connecting $0$ and $1$.

{\it Example 2:} Consider \eqref{ex1} in the domain $\Omega$ with branches $\mathcal{H}_i$ being asymptotically straight. That is, $\mathcal{H}_i$ is defined by \eqref{defbranch} where $\omega_i(s)\underset{s\rightarrow +\infty}{\rightarrow} \omega_i^{\infty}\subset \R^{N-1}$ and $\omega_i^{\infty}$ is a bounded non-empty set of $\R^{N-1}$. One can extend the branch $\mathcal{H}_i$ by $\widetilde{\mathcal{H}}_i$ satisfying \eqref{wH} where $\widetilde{\omega}_i(s)=\omega_i(s)$ for $s>0$ and $\widetilde{\omega}_i(s)=\omega_i^{\infty}$ for $s<s_0$ and some $s_0<0$. By \cite{P}, one knows that there are front-like solutions facing to directions $e_i$ and $-e_i$ for \eqref{ex1} with $\Omega$ replaced by $\widetilde{\mathcal{H}}_i$, which can also be easily verified to be almost-planar fronts connecting $0$ and $1$. Notice here that one may need to make $\widetilde{\mathcal{H}}_i$ smooth and $\widetilde{\mathcal{H}_i}\cap B(0,L)$ being star-shaped\footnote{A bounded set $K$ is called star-shaped if there is $x$ in the interior $\mathrm{Int}(K)$ of $K$ such that $x+t(y-x)\in\mathrm{Int}(K)$ for all $y\in\partial K$ and $t\in[0,1)$. Then, we say that $K$ is star-shaped with respect to the point $x$.}  such that the propagation of the front-like solutions is complete. Then, Assumptions~\ref{assumption-r} and~\ref{assumption-l} hold in this case. Notice that the limiting system of Assumption~\ref{assumption-limiting} in this case is \eqref{ex1} in a straight cylinder rotated by $\R\times\omega^{\infty}_i$. Thus, Assumption~\ref{assumption-limiting} holds. Therefore, by Theorem~\ref{Th1}, there exist entire solutions emanating from those front-like solutions. If additionally \eqref{front} holds, then the entire solution emanating from those front-like solutions is the only type of transition front connecting $0$ and $1$ by Theorem~\ref{th3}. Some potential geometrical conditions such that \eqref{front} holds are that the center $\Omega\cap B(0,L)$ is star-shaped and branches $\mathcal{H}_i$ are narrowing or slowly opening to be straight.

More examples can be made, by referring to \cite{BN} for \eqref{ex1} with an advection term in a cylinder, referring to \cite{MNL} for \eqref{ex1} in cylinders with periodic boundaries and so on.

%%%%%%%%%%%%%%%%%%%%%%%%%%%%%%%%%%%%%%%%%%%%%%%
%%%%%%%%%%%%%%%%%%%%%%%%%%%%%%%%%%%%%%%%%%%%%%%

\end{document}